\documentclass[11pt]{article}
\usepackage[utf8]{inputenc}
\usepackage{a4wide}
\usepackage{graphicx, geometry}
\makeatletter

\newcommand\blfootnote[1]{%
  \begingroup
  \renewcommand\thefootnote[0]{}\footnote{#1}%
  \addtocounter{footnote}{-1}%
  \endgroup
}

\usepackage{hyperref}
\usepackage{adjustbox}
\usepackage{graphicx}
\usepackage{subfigure}
\usepackage[toc,page]{appendix}
\usepackage{xcolor}
\usepackage{pst-all}
\usepackage[english]{babel}
\usepackage{pdfsync}
\usepackage{amsmath,amssymb,amsthm}
\usepackage{mdwlist}
\usepackage{mathrsfs}
\usepackage{mathtools}
\usepackage{bm}
\usepackage{dsfont}
\usepackage{bbm}
\usepackage{imakeidx}
\usepackage{makeidx}

\usepackage{tikz}
\usepackage[framemethod=TikZ]{mdframed}

\newcommand{\PX}{\mathscr{P}}

\DeclareFontFamily{U}{matha}{\hyphenchar\font45}
\DeclareFontShape{U}{matha}{m}{n}{
  <-6> matha5 <6-7> matha6 <7-8> matha7
  <8-9> matha8 <9-10> matha9
  <10-12> matha10 <12-> matha12
  }{}
\DeclareSymbolFont{matha}{U}{matha}{m}{n}
\DeclareMathSymbol{\Lt}{3}{matha}{"CE}

\def\Xint#1{\mathchoice 
  {\XXint\displaystyle\textstyle{#1}}%
  {\XXint\textstyle\scriptstyle{#1}}%
  {\XXint\scriptstyle\scriptscriptstyle{#1}}%
  {\XXint\scriptscriptstyle\scriptscriptstyle{#1}}%
  \!\int} 
\def\XXint#1#2#3{{\setbox0=\hbox{$#1{#2#3}{\int}$} 
  \vcenter{\hbox{$#2#3$}}\kern-.5\wd0}} 
\def\-int{\Xint -}

\numberwithin{equation}{section}

\newcommand{\Div}{{\rm div}}

\newcommand{\R}{\mathbb{R}}

\newcommand{\N}{\mathbb{N}}

\newcommand{\nnu}{{\mbox{\boldmath$\nu$}}}

\newcommand{\Kliminf}{K\kern-3pt-\kern-2pt\mathop{\rm lim\,inf}\limits}  
\newcommand{\argmin}{\mathop{\rm argmin}\limits}   
\renewcommand{\d}{{\mathrm d}}
\newcommand{\dt}{{\d t}}

\newcommand{\restr}[1]{\lower3pt\hbox{$|_{#1}$}} 

\newcommand{\nchi}{{\raise.3ex\hbox{$\chi$}}}

\newcommand{\prob}[1]{\mathscr P(#1)}                   
\newcommand{\probp}[2]{\mathscr P_{#2}(#1)}                   
\newcommand{\Law}{\rm Law}

\renewenvironment{proof}{\removelastskip\par\medskip   
\noindent{\em Proof.} \rm}{\penalty-20\null\hfill$\square$\par\medbreak}

\newenvironment{refproof}{\removelastskip\par\medskip   
\noindent{\em Proof of Theorem} \rm}{\penalty-20\null\hfill$\square$\par\medbreak}
\newenvironment{refproofProp}{\removelastskip\par\medskip   
\noindent{\em Proof of Proposition} \rm}{\penalty-20\null\hfill$\square$\par\medbreak}

\numberwithin{equation}{section}

\newtheorem{thm}{Theorem}[section]
\newtheorem{prop}[thm]{Proposition}
\newtheorem{lemma}[thm]{Lemma}
\newtheorem{cor}[thm]{Corollary}
\newtheorem{dfn}[thm]{Definition}
\newtheorem{assumptions}[thm]{Assumptions}

\newtheorem{result*}{Useful Result}

\theoremstyle{remark}
\newtheorem{rmk}[thm]{Remark}

\makeindex

\title{A description based on optimal transport for a class of stochastic McKean-Vlasov control problems}

\author{Francesco C. De Vecchi\thanks{Dipartimento di Matematica “Felice Casorati”, Universit\`a degli Studi di Pavia, Via Adolfo Ferrata 5, 27100 Pavia (Italy), \emph{francescocarlo.devecchi@unipv.it}} \, and \, Chiara Rigoni\thanks{Fakultät für Mathematik, Universität Wien, Oskar-Morgenstern-Platz 1, 1090 Wien (Austria), \emph{chiara.rigoni@univie.ac.at}}}
\date{\today}

\begin{document}

\maketitle
\begin{abstract}
    We study the convergence of an $N$-particle Markovian controlled system to the solution of a family of stochastic McKean-Vlasov control problems, either with a finite horizon or Schrödinger type cost functional. Specifically, under suitable assumptions, we prove the convergence of the value functions, the fixed-time probability distributions, and the relative entropy of their path-space probability laws. These proofs are based on a Benamou-Brenier type reformulation of the problem and a superposition principle, both of which are tools from the theory of optimal transport. 
\end{abstract}

{
  \hypersetup{linkcolor=black}
  \tableofcontents
}

\section{Introduction}

In this paper we consider a closed-loop stochastic optimal control problem  with additive noise, affine control and cost function which is quadratic with respect to the control. More precisely, the controlled SDE we study is given by
\begin{equation}\label{eq:SDEintro}
\d X_t = \Big( A\big(t, X_{[0, t]}\big) + b(X_t, \mu_t)  \Big) \, \d t  + \sqrt 2 \d W_t,\end{equation}
where $b \colon \R^n \times \prob{\R^n} \rightarrow \R^n$ is a regular enough drift, $\mu_t$ is the marginal law of the process $X_t$ evaluated at time $t$, and $A$ is the closed loop control function depending on the solution process $X_s$, $0\leq s\leq t$. We call this class of problems of ``closed-loop'' type, since the controls are determined by the states themselves. Unlike in one-player stochastic control
problems, the open-loop and closed-loop equilibria are typically distinct (see \cite[Section 2.1.2]{CDI} for a thorough discussion).  We employ the dynamics described by \eqref{eq:SDEintro} to study the minimization problem associated with the cost functional given by 
\begin{equation}\label{eq:costintro}
\mathcal{C}(A)=\mathbb E\bigg[\int_0^T \bigg( \dfrac{|A(t, X_{[0, t]})|^2}{2} + \mathcal V(X_t, \mu_t)\bigg) \, \dt \bigg] + \mathcal{G}(\mu_T), 
\end{equation}
where $\mathcal{V}\colon\R^n \times \prob{\R^n} \rightarrow \R$ and $\mathcal{G}\colon \prob{\R^n} \rightarrow \R$ are regular enough functions of $x\in \mathbb{R}^n$ and $\mu_t$ is the law of the solution process $X_t$.\\

It is worth to remark that the process solution of \eqref{eq:SDEintro} is a McKean–Vlasov process, which can be obtained as a limit of a mean-field system of interacting particles, and, in fact, the main aim of this paper is to study various kinds of convergence of the natural Markovian $N$-particle approximation to its mean-field limit. Here with ``Markovian'', we refer to the fact that both, the controlled equation and the cost functional depend only on the position of the solution process at any fixed time. More precisely, the $N$-particle approximation we consider is of the form
\begin{equation}\label{eq:NpartSDEintro}
\d X_t^{N, i} = \Bigg( A^{N, i} \Big(t, X^N_t\Big) + b \Bigg(X_t^{N, i}, \dfrac1N \sum_{j=1}^N \delta_{X^{N, j}_t}\Bigg)\Bigg) \, \dt + \sqrt{2} \, \d W_t^{N, i},
\end{equation}
$X^N=(X^{N,1},...,X^{N,N})$ being a process taking values in $\mathbb R^{n N}$, and the correspondent cost function we want to minimize is given by
\begin{multline}\label{eq:Npartcostintro}
\mathcal C^N(A^{N, 1}, \dots A^{N, N})  := \\ \dfrac1N \mathbb E \Bigg[ \sum_{i = 1}^N \int_0^T \Bigg(   \dfrac{|A_t^i(X_t)|^2}{2} + \mathcal V \Bigg( X_t^i, \dfrac1N \sum_{j=1}^N \delta_{X_t^j}  \Bigg) \Bigg) \, \dt  \Bigg] + \mathbb{E}\left[\mathcal{G}\left(\frac{1}{N}\sum_{j=1}^N \delta_{X_T^{N,j}} \right) \right].
\end{multline}
In particular, we will prove that the minimum of the $N$-particle cost functional \eqref{eq:Npartcostintro} converges to the minimum of the cost functional \eqref{eq:costintro} and also that the probability law  $\mathbb P^{(N|k)}_{\min}$ of the optimal solution of the equation \eqref{eq:NpartSDEintro} converges to the optimal solution  $\mathbb P_{\min}^{(\infty|k)}$ of equation \eqref{eq:SDEintro} with respect to the Kullback–Leibler divergence (see Theorem \ref{theorem:main1} and \ref{thm:KLdiv}, respectively, for the precise statements of these results). \\

The study of McKean–Vlasov (or mean field) optimal control problems is an important topic in the field of stochastic optimal control precisely because it can be seen as a suitable limit of Markovian systems involving a growing number of interacting particles or players in a cooperative game (see, e.g., \cite[Chapter 6]{CDI}) and it was approached in many different ways, as we are going to explain in the following. We also recall that this theory is closely related to the one of mean field games, which has been extensively treated in literature (see e.g. \cite{LionsMastereq, CDI, CDII}). However, these two frameworks have some major differences, see \cite{CDL} for a detailed comparison between the two. 

From a stochastic optimal control point of view, this problem can be treated in two ways:
\begin{itemize}
\item  the Pontryagin's maximum principle, using the backward/forward McKean–Vlasov SDE or
\item through the programming principle idea using the master equation,
\end{itemize}
see for example 
\cite{BayraktarCossoPham2018,CossoGozzi2023,Cosso2023master,DjeteDylan2022,Santambrogio2018}. In particular, this second approach reduces the problem to the study of a system of coupled PDEs, which, in the case of equation \eqref{eq:SDEintro} with cost functional \eqref{eq:costintro}, is heuristically equivalent to
\[
\left\{\begin{array}{l}
\partial_t \mu_t -\Delta \mu_t + \Div( (\nabla u + b(x, \mu_t)) \mu_t  \Big)=0\\
\partial_t u + \Delta u+|\nabla u|^2=
-\langle b_t(x,\mu_t) , \nabla u \rangle +\int_{\mathbb{R}^n}\langle \delta_{\mu}b(x,\mu_t,y), \nabla u(y) \rangle \mu_t(\d y)\\
+\mathcal{V}(x,\mu_t)+ \int_{\mathbb{R}^n} \delta_{\mu}\mathcal{V}(x,\mu_t,y)\mu_t(\d y) 
\end{array}
\right. 
\]
and this equation can be handled more easily than the stochastic problem, see e.g. \cite{CardaliaguetLionsPorretta2012}.\\

In the direction of the topic investigated in this paper, the first result was obtained by Lacker in \cite{Lacker}, where the convergence of the $N$-particle Markovian approximation to its mean field limit is shown using an argument based on the theory of Martingales. However, the setting in this case is quite different from the one treated in the current paper: in fact, the problem considered in \cite{Lacker} is in the open-loop case and the coefficients in \eqref{eq:SDEintro} are much more regular.  More recent results in the open-loop case are \cite{DjeteDylan2022,CarSou2023,cardaliaguet2023sharp}.

As for the case in which the control is of closed-loop type, the convergence problem for the stochastic optimal control has been studied in \cite{Lackerclosedloop2020}, when the noise is additive  and the drift is bounded,  and in \cite{Lackerclosedloop2023}, for mean field games with common noise. \\

It is important to underline that the assumptions considered in the current manuscript are not covered by the references mentioned above and, in particular, to the best of our knowledge this is the first result without any compactness assumption on the set of controls. Moreover, this seems to be the first paper in which the convergence of the Kullback-Leibler divergence is shown (see the last section of this paper).

The main novelty of this paper consists in making use of the ideas and the techniques proper of the theory of optimal transport to study a stochastic control problem. In fact, the theory of
optimal transport has been already successfully applied when considering mean field optimal control ODEs (see e.g. \cite{RossiFrancesco2017,RossiFrancesco2019,RossiFrancesco2021,CavagnariLisiniOrrieriSavare2022, Djete2022,FornasierLisiniOrrieriSavare2019,OrrieriPorrettaSavare,SantambrogioPDE2018}), but not so much in the stochastic setting. We point out that in the case of the optimal control ODE problem, the papers \cite{CavagnariLisiniOrrieriSavare2022,FornasierLisiniOrrieriSavare2019} are devoted to the investigation of the convergence of the $N$-particle Markovian system to the corresponding mean field optimal control problem, proving a deterministic counterpart of the results shown in this paper. However, the literature on the study of the McKean-Vlasov optimal control using optimal transport techniques focuses more on the study of the Schr\"odinger problem, see e.g. \cite{ConfortiLeonard,claisse2023mean,conforti2024hamilton} (see also \cite{Leonard2014} for an exhaustive discussion of the relation between the Schr\"odinger problem and the theory of optimal transport). In particular, an approach similar to the one considered in the current paper can be found in \cite{ConfortiLeonard}, where a class of Schr\"odinger-type problems similar to \eqref{eq:SDEintro} with cost function \eqref{eq:costintro} is considered in the special case in which $b$ is a gradient of a suitable function and linear in the measure, while $\mathcal V \equiv 0$ in the cost functional. Hence our optimal transport-based approach is particularly suited to treating the Schr\"odinger problem, since it is not based on the dynamic programming approach or the forward/backward SDE approach to stochastic optimal control, which are both difficult to extend in the Schr\"odinger case.

For this reason, to the best of our knowledge, the results in the present paper are the first ones dealing with the convergence of an $N$-particle Markovian approximation to a mean-field  Schr\"odinger type problem.  Let us finally mention that in  the papers \cite{ADVRU20,ADVRU22,ADVU17} the controlled SDE \eqref{eq:SDEintro} with cost functional of the form \eqref{eq:costintro} is analysed in the ergodic setting.\\

The idea of the main convergence results in Theorem \ref{theorem:main1} and \ref{thm:KLdiv} relies on the observation that a generic mean-field or Markov problem having affine control and cost functional which is quadratic in the control is equivalent to the minimization problem of a suitable energy, which is in turn provided by the Benamou-Brenier formulation of the problem, see \cite{BB}. This new approach involves only the marginal probability laws at a fixed time and and admissible velocity in the sense of \cite[Chapter 8]{AGSBook}. In particular, in the case of equation \eqref{eq:SDEintro} with cost function \eqref{eq:costintro}, the Benamou-Brenier formulation essentially consists in minimizing the functional
\begin{equation}\label{eq:BBintro}
\begin{split}
\mathcal E  \Big(\{\mu_t, w_t \}_{t \in [0, T]}\Big) = & \frac{1}{2}\int_0^T \int_{\R^n} \big(   |w_t(x)|^2 + |b(x, \mu_t)|^2\big) \mu_t(\d x) \, \dt +\frac{1}{2} \int_0^T \mathcal I(\mu_t) \, \dt\\ & +\mathcal H(\mu_T) - \mathcal H(\mu_0) - \int_0^T \int_{\R^n} \Big(\langle w_t(x), b(x, \mu_t) \rangle + \text{div}_{\R^n} b(x, \mu_t)   \Big)\, \mu_t(\d x) \, \dt\\ &
+\int_0^T \int_{\R^n} \mathcal V(x, \mu_t)\mu_t(\d x) \, \dt+ \mathcal{G}(\mu_T),
\end{split}
\end{equation}
where $\{w_t\}_{t \in [0, T]}$ solves the continuity equation $\partial_t \mu_t + \Div(w_t \mu_t)=0$. In fact, it is possible to prove that if we can find a couple $\{\mu_t,w_t\}_{t \in [0, T]}$, where $\{\mu_t\}_{t \in [0, T]}$ is a flow of measures with logarithm derivative which is square integrable and $\{w_t\}_{t \in [0, T]}$ is an admissible square integrable velocity, then there exists a solution to equation \eqref{eq:SDEintro} with control given by $A_t(x) := w_t(x)-b(x,\mu_t)+\nabla \log(\mu_t)$ (see e.g. \cite{BarbuRo,Trevisan2016}). \\

On the other hand, using this approach based on the continuity equation, we can find different processes having the same time marginal but (generally) different joint probability distributions. Since our optimization problem does not directly involve  the solution of the stochastic equation \eqref{eq:SDEintro} (and thus the joint distributions of the process) but only the couple $\{\mu_t,w_t\}_{t \in [0, T]}$, we can replace the solution of equation \eqref{eq:SDEintro} with any process having the same time marginal. Hence we choose a Eulerian type representation with respect to a probability measure $\bm{\lambda}^{\infty}$ under which the following equation holds 
\[\mathbb{E}_{\bm{\lambda}^{\infty}}\left[\left.\frac{\d X_t}{\dt}\right|X_t\right]=w_t(X_t) \]
almost surely. As for  the $N$-particle system, we make use of the superposition principle (see \cite[Theorem 8.2.1]{AGSBook} and \cite[Section 7]{AmbrosioTrevisan}) to find a probability measure $\bm{\lambda}^N$ for which it holds
\[\frac{\d X^{(N)}_t}{\dt}=w_t^N(X^{(N)}_t),\]
almost surely. It is important to note that the probability laws $\bm{\lambda}^{\infty}$ and $\bm{\lambda}^N$ are different from the laws $\mathbb{P}_{\min}^{(\infty)}$ and $\mathbb{P}_{\min}^{(N)}$ of the solutions of the stochastic equations \eqref{eq:SDEintro} and \eqref{eq:NpartSDEintro} but they are linked by the following important identities
\[e_{t, \sharp}(\bm{\lambda}^{\infty})=\mu^{\infty}_t=e_{t, \sharp}(\mathbb{P}_{\min}^{(\infty)}) \quad \text{ and }\quad e_{t, \sharp}(\bm{\lambda}^{N})=\mu^N_t=e_{t, \sharp}(\mathbb{P}_{\min}^{(N)}).\]
 Since the couple $\{\mu^N_t,w^N_t\}_{t \in [0, T]}$ minimizes a suitable functional of the form \eqref{eq:BBintro}, the probability laws $\{\bm{\lambda}^{N,(k)}\}_{N \in \N}$ are tight and they actually converge, up to a subsequence, to the limit $\bm{\lambda}^{\infty, (k)}$, which in turns minimizes the functional \eqref{eq:BBintro}, see Theorem \ref{theorem:pnconvergence} for the precise statement. This implies the convergence of the minimum of the cost functions.
 
A simple result which follows from the proof of Theorem \ref{theorem:main1} when we assume also the uniqueness of the limit is that the particle system $X^{(N)}_t$ is Kac-chaotic (see \cite{KacChaos} for the precise definition of this notion), i.e., for every fixed time $t\in[0,T]$, the probability law $\mu^{N,(k)}_t$ of $X_t^{(N|k)}$ converges weakly to the probability law $\mu^{\infty,(k)}$ (see Theorem \ref{thm:Kacchaotic}).  
Furthermore, when the limit system admits a unique minimizer,  Girsanov's theorem and the convergence of the minimum of the cost functionals (see Theorems \ref{thm:compact} and \ref{thm:compact2}) allow to infer from the convergence of the laws $\{\bm{\lambda}^N\}_{N \in \N}$ to $\bm{\lambda}^{\infty}$ the strong convergence in $\probp{\Omega^k}{p}$. Finally, with some additional conditions, we show the convergence with respect to the Kullback–Leibler divergence of the probability laws $\{\mathbb{P}_{\min}^{(N|k)}\}_{N \in \N}$ to the solution of the original stochastic problem. This last result is usually called strong-Kac chaoticity, see \cite{Lacker2018}.

\subsection*{Acknowledgements}

The first author was supported by INdAM (Istituto Nazionale di Alta Matematica, Gruppo Nazionale per l’Analisi Matematica, la Probabilit\`a e le loro Applicazioni and Gruppo Nazionale per la Fisica Matematica), Italy. The second named author gratefully acknowledges funding by the Austrian Science Fund (FWF) through project ESP 224-N.

\section{Background, notation and preliminary results}\label{section:preliminary}

We start this section by introducing the basic notation that we are going to use in the following. Let $\Omega = C^0([0, T], \R^n)$ be the space of all continuous $\R^n$-valued paths defined on $[0, T]$, endowed with the uniform topology, and  $(X_t)_{0 \le t \le T}$ be the canonical process on $\Omega$. Recall that the mapping $X = (X_t)_{0 \le t \le 1} \colon \Omega \to \Omega$ is actually the identity on $\Omega$. Hence we equip $\Omega$ with the canonical filtration
$\mathcal F_t = \sigma \big(X_t; 0 \le t \le T\big)$. Notice that the filtration $\mathcal F_t$ is actually generated by the time projections
 \[X_t(\omega) = \omega_t \in \R^n, \qquad \text{for any } t \in [0, T], \, \omega = \left( \omega_s  \right)_{0 \le s \le T} \in \Omega.\]
For any $t \in [0, T]$, we denote by $\mu_t$ the law of $X_t$. We then set $\Omega^N :=( C([0, T], \R^{n}))^N$ and $\Omega^\infty := (C([0, T], \R^n))^\infty$, and we observe that 
\[ C([0, T], \R^{n})^N \sim C([0, T], \R^{nN}) \text{ as well as } C([0, T], \R^n)^\infty \sim C([0, T], \R^{n \infty}),\]
where $\R^{n \infty}$ is the Cartesian product of a countable number of $\R^n$. In a similar way we denote by $\omega^N$ the elements in $\Omega^N$ and $X^N$ the canonical process on $\Omega^N$.
Usually a measure $\mathbb P^{(N)}$ is a probability law on $\left( \Omega^N, \mathcal F_t^N\right)$, where $\mathcal F_t^N$ is the tensor product $\sigma$-algebra on $\Omega^N$, and $\mathbb P^{(\infty)}$ is a probability measure on $\Omega^\infty$. In particular, we denote by $\mu^N_t$ the law of $X_t^N$. For any $\mu \in \PX(\R^n)$ we denote by $\mu^{\otimes k}$ the probability measure on $\R^{nk}$ given by
\[\mu^{\otimes k} := \underbrace{\mu \otimes \dots \otimes \mu}_{k\text{-times}}.\]

Let $\sigma_{ij} \colon \mathbb{R}^{nN} \rightarrow \mathbb{R}^{n N}$ be the map switching the $i$-th and $j$-th particle.
Hereafter if $\bf{\sigma}$ is a measure on $Y^N$, the space obtained by taking $N$ copies  of some measure space $Y$, $N\in\{1,...,k, ...,\infty\}$, for any $k \le N$ we denote by $\sigma^{(k)}$ the push-forward of the measure $\sigma$ with respect to the natural projection $\pi^{N,k}\colon Y^N\rightarrow Y^k$. Furthermore, if $X^{(N)}$ is the natural path on some topological space $Y^N$,  we denote by $X^{(N|k)}$ the vector made by the first $k$ coordinates.

\subsection*{Absolutely continuous curves, Wasserstein space and continuity equation}
The definitions and the results of this section are given in the more general setting of a complete and separable metric space $({\rm X}, {\rm d})$.\\

Let $\gamma \colon (a, b) \to {\rm X}$ be a curve, $a, b \in \R$. We say that $\gamma$ is absolutely continuous and we write $\gamma \in {\rm AC}((a, b), {\rm X})$ if there exists $w \in L^1\left((a, b)\right)$ such that
\begin{equation}\label{eq:metricder}
\d(\gamma(s), \gamma(t)) \le \int_s^t w(r) \, \d r, \qquad \forall a < s \le t < b.
\end{equation}
Then for any $\gamma \in {\rm AC}((a, b), {\rm X})$ the limit
\[
|\gamma'|(t) := \lim_{s \to t} \dfrac{\d\big(\gamma(s), \gamma(t)\big)}{|s-t|}
\]
exists for $\mathscr{L}^1$-a.e. $t \in (a, b)$, where, hereafter, we donete by $\mathscr{L}^k$ the Lebesgue measure on $\R^k$. We call the function $t \mapsto |\gamma'|(t)$ the metric derivative of $\gamma$ and we observe that it belongs to $L^1\left((a, b)\right)$ and it is the minimal admissible integrand for the right hand side of \eqref{eq:metricder}, in the sense that
\[
|\gamma'|(t)  \le w(t) \, \text{ for } \mathscr{L}^1\text{-a.e. } t \in (a, b), \forall w\in L^1\left((a, b)\right) \text{ satisfying } \eqref{eq:metricder}.
\]
We refer to \cite[Theorem 1.1.2]{AGSBook} for a detailed proof of these results.\\

We denote by $\mathcal B({\rm X})$ the family of the Borel subsets of $\rm X$ and by $\prob{\rm X}$ the family of all Borel probability measures on $\rm X$. The support supp$\mu \subset {\rm X}$ of $\mu \in \prob{\rm X}$ is the closed set defined by
\[
\text{supp} \mu := \left\{ x \in {\rm X} : \mu(U) > 0 \, \text{ for each neighborhood } U \text{ of } x  \right\}.
\]
We say that a sequence $\{ \mu_n \} \subset \prob{\rm X}$ weakly converges to $\mu \in \prob{\rm X}$ as $n \to \infty$ if
\[
\lim_{n \to \infty}\int_{\rm X} f(x) \, \mu_n(\d x) = \int_{\rm X} f(x) \mu(\d x)
\]
for any bounded and continuous real function $f \in C_b({\rm X})$ defined on $\rm X$.

For  $p \ge 1$ we consider the subset of the probability measures defined by
\[
\probp{{\rm X}}{p}  := \left\{ \mu \in \prob{{\rm X}} : \int_{\rm X} {{\rm d}}^p (x, x_0) \mu(\d x) < + \infty \text{ for some } x_0 \in \rm X  \right\}
\]
endowed with the $p$-th Wasserstein distance between $\mu, \nu \in \probp{\rm X}{p}$
\[
W_p^p(\mu, \nu) := \min \left\{  \int_{\rm X}^2 {\rm d}^p(x, y) \pi (\d x, \d y) \, : \,  \pi \in \Gamma(\mu, \nu) \right\}
\]
where $\Gamma(\mu, \nu) := \{ \pi \in \prob{{\rm X}^2} : {\rm p}^1_{\sharp} \pi = \mu \text{ and } {\rm p}^2_{\sharp} \pi = \nu \}$ is an admissible transport plan between $\mu$ and $\nu$. Notice that $\Gamma(\mu, \nu) = \{ \mu \times \nu  \}$ as soon as $\mu$ or $\nu$ is a Dirac mass. Moreover, we remark that $\left(\probp{{\rm X}}{p}, W_p\right)$ is a metric space, as shown in \cite[Section 7.1]{AGSBook}, which is complete and separable since $(\rm X, \rm d)$ satisfies these properties.

In particular, a subset $\mathcal K \subset \probp{\rm X}{2}$ is relatively
compact w.r.t. the topology induced by $W_2$ if and only if it is tight and $2$-uniformly integrable and the relation between the convergence in the $W_2$ distance and the narrow convergence in $\probp{\rm X}{2}$ is given by the following criterion
\begin{equation*}
    W_2(\mu_n, \mu) \to 0 \iff 
    \begin{cases}
        \mu_n \rightharpoonup \mu \quad \text{ narrowly }\\
        \int_{\rm X} {\rm d}^2(\cdot, x_0) \d \mu_n \to \int_{\rm X} {\rm d}^2(\cdot, x_0) \d \mu \quad \text{for some } x_0 \in \rm X.
    \end{cases}
\end{equation*}
whose proof can be found e.g. in \cite[Proposition 7.1.5]{AGSBook}.\\

In order to minimize technicalities, from now on we focus on our specific cases of interest, that is $\R^n$ equipped with the Euclidean distance. A fundamental result in the theory of optimal transport states that the class of absolutely continuous curves in $\probp{\R^n}{p}$, $p>1$, coincides with the solutions of the \emph{continuity equation}. More precisely, if $\mu \colon (a, b) \to \probp{\R^n}{p}$, $p > 1$ is an absolutely continuous curve with metric derivative $\mu'(t) \in L^1((a, b))$, then there exists a Borel vector field $w \colon (x, t) \mapsto w_t(x)$ such that
\[
w_t \in L^p({\rm X}, \mu_t) \, \text{ and } \, || w_t ||_{L^p(\R^n, \mu_t)} \le |\mu'|(t) \, \text{ for }  \mathscr{L}^1 \text{-a.e.} t \in (a, b),
\]
and the continuity equation
\begin{equation}\label{eq:CE}
    \partial_t \mu_t + \nabla \cdot (w_t \mu_t) = 0 \quad \text{in } \R^n \times (a, b)
\end{equation}
holds in the sense of distribution, namely for any $\varphi \in \text{Cyl}(\R^n \times (a, b))$ we have
\[
\int_a^b \int_{\R^n} \Big( \partial_t \varphi(x, t) + \langle w_t(x), \nabla_x \varphi(x, t)\rangle  \Big) \, \d \mu_t(x) \, \dt = 0.
\]
Vice versa, if a weakly continuous curve $\mu_t \colon (a, b) \to \probp{\R^n}{p}$ satisfies the continuity equation for some Borel vector field $w_t$ with $||w_t||_{L^p(\R^n, \mu_t)} \in L^1((a, b))$, then $\mu_t \colon (a, b) \to \probp{\R^n}{p}$ is actually absolutely continuous and $ |\mu'|(t)  \le || w_t ||_{L^p(\R^n, \mu_t)}$ for $ \mathscr{L}^1$-a.e. $t \in (a, b)$.\\

Since in the following we will often need to indicate that a pair given by a curve of measures $\{ \mu_t \}_{t \in [0, T]} \subset \probp{\R^n}{p}$, with prescribed initial value $\tilde \mu_0$, and a vector field $\{ w_t \}_{t \in [0, T]} \subset L^p(\R^n, \mu_t)$ satisfies the continuity equation \eqref{eq:CE}, we introduce the following set:
\[
\begin{split}
\Lambda^p_{\tilde \mu_0} := \Big\{  \big\{ \mu_t, w_t  \big\}_{t \in [0, T]} \, \big| \, \mu \in {\rm AC} \big([0, T], \probp{\R^n}{p}\big), \mu_0 = \tilde \mu_0 \text{ and } w_t \in L^p(\R^n, \mu_t)  
 \text{ for which \eqref{eq:CE} holds}  \Big\}.
\end{split}
\]
We will omit the exponent $p$ in the case in which $p=2$.

We conclude this section by providing the statement of a fundamental result, the so-called superposition principle, which allows us to lift, not canonically in general, nonnegative solutions
of the continuity equation to measures on paths.

\begin{thm}[Superposition principle]\label{thm:SP}
Let $\mu_0 \in \probp{\R^n}{2}$ be a fixed measure.  For every $\mu \in AC([0, T], \PX_2(\R^n))$ and every $w \in C([0, T], L^2(\R^n, \mu_{\cdot}))$ for which $\{\mu_t, w_t\}_{t \in [0, T]} \in \Lambda_{\mu_0}$, there exists a probability measure ${\bm \lambda}$ in $\R^n \times \Omega$ such that
    \begin{itemize}
    \item[i)] for any $t \in [0, T]$ it holds $(X_t)_\sharp {\bm \lambda} = \mu_t$ and
    \item[ii)] for $ \mathscr{L}^1$-a.e. $t \in (0, T)$ and for ${\bm \lambda}$-a.e. $(x, \omega) \in \R^n \times \Omega$ it holds
    \begin{equation}\label{eq:suppLambda}
    \dfrac{\d}{\d t} X_t(\omega) = w_t\big(X_t(\omega)\big) \, \text{ with }\, X_0(\omega) = x.
    \end{equation}
    \end{itemize}
    We call this measure ${\bm {\lambda}}$ the lift of the curve $\{ \mu_t, w_t  \}_{t \in [0, T]}$.
\end{thm}

\begin{rmk}\label{rmk:suppLambda}
 We observe that \eqref{eq:suppLambda} ensures that for any lift measure $\bm \lambda$ it holds
  \begin{equation}
      \int \| X_t (\omega) \|^2_{(H^1 ([0, T], \mathbb{R}^n))} \bm \lambda
      (\d x, \d \omega) = \int_0^T \int_{\mathbb{R}^{n}} | w_t (x)
      |^2_{\mathbb{R}^{n}} \mu_t ( \d x) \dt.
      \label{eq:boundh1N}
    \end{equation}
    Thus the measure $\bm \lambda$ is supported in $\R^n \times \big(H^1 ([0, T], \mathbb{R}^n)\big)$. Moreover, since $\R^n \times \big(H^1 ([0, T], \mathbb{R}^n)\big) \subset \R^n \times
    \big({\rm{AC}} ([0, T], \mathbb{R}^n)\big) \subset \R^n \times \big(C^0 ([0, T],
    \mathbb{R}^n)\big)$, in the following we will not always specify the set in which the measure $\bm \lambda$ is actually defined because the continuity of the inclusions between these spaces ensures that we can uniquely extend $\bm \lambda$ using the push-forward operator.\qed
\end{rmk}

\subsection*{Fractional Sobolev spaces}

A possible definition of fractional Sobolev spaces on general, possibly non smooth, open sets $U \subset \R$ is given by setting
\[
H^s(U) := \left\{ f \in L^2(U) \Bigg|  \dfrac{f(t + h) - f(t)}{h^{\frac12 + s}} \in L^2(U \times U) \right\},
\]
where $s \in (0, 1)$ is a  fixed fractional exponent. We endow this set with the natural norm
\[
|| f ||^2_{H^s(U)} := ||f ||^2_{L^2(U)} + \int_U \int_U \dfrac{|f(t + h) - f(t)|^2}{h^{1 + 2 s}} \, \d h \, \dt,
\]
where the term
\[
[f]^2_{H^s(U)} := \int_U \int_U \dfrac{|f(t + h) - f(t)|^2}{h^{1 + 2 s}} \, \d h \, \dt
\]
is the so-called Gagliardo semi-norm of $f$.
Intuitively, $H^s(U)$ constitutes an intermediary Banach space between $L^2(U)$ and $H^1(U)$. 
These spaces are usually called fractional Sobolev spaces (see, e.g.,  \cite{DNPV12}) or Besov spaces (see, e.g., \cite{BookBesov}). We can extend the definition of fractional Sobolev spaces also  for values $s\in [-1,0)$ by defining $H^{-s}(U)$ to be the dual of $H^{s}(U)$ with respect to the duality given by the standard inner product in $L^2(U)$.  We also define the H\"older space $C^{\alpha}(U)$, for some fixed $\alpha \in (0,1)$, to be the (Banach) space of all the continuous functions on $U$ for which the norm 
\[ \| f\|_{C^{\alpha}}=\| f\|_{\infty}+\sup_{t,s\in U}\frac{|f(t)-f(s)|}{|t-s|^{\alpha}} \quad \text{ is finite}.\]

\begin{prop}\label{proposition:embedding}
    Let $U\subset \R$ be an open set. The following properties hold true:
    \begin{itemize}
        \item[i)] for any $s \leq s'\in(-1,1)$ we have $H^{s'}(U) \subset H^{s}(U)$ and the immersion is compact when $s<s'$;
        \item[ii)] for any $s \leq \alpha$ we have $C^{\alpha}(U) \subset H^{s}(U)$ and the immersion is compact when $s <\alpha$;
        \item[iii)] for any $s,\alpha\in(0,1)$ such that $\alpha \leq s-\frac{1}{2}$, we have $H^{s}(U) \subset C^{\alpha}(U) $, and the immersion is compact when $\alpha < s-\frac{1}{2}$;
        \item[iv)] for any $s\in(0,1]$ the classical derivatives $\frac{\d}{\d t}:C^{\infty}(U) \rightarrow C^{\infty}(U)$ can be extended in a unique continuous way to a linear map $\frac{\d}{\d t}:H^{s}(U) \rightarrow H^{s-1}(U)$. 
    \end{itemize} 
\end{prop}
\begin{proof}
The proof of the first three points can be found in \cite[Section 7 and 8]{DNPV12}. The last point is a special case of \cite[Proposition 2.78]{BookBesov}.    
\end{proof}

As a consequence of the embedding properties and of point \emph{iv)} of Proposition \ref{proposition:embedding} we get the following result:

\begin{cor}\label{corollary:continuity}
    Let $U \subset \R$ be a bounded open set and consider the bilinear functional $B \colon C^{\infty}(U) \times C^{\infty}(U) \rightarrow \mathbb{R}$ defined as
    \[B(f,g)=\int_U f(t) \frac{\d g(t)}{\d t} \d t.\]
    Then, for any $s,s'\in(0,1)$ such that $s+s'-1 \geq 0$,  $B$ can be extended in a unique continuous way to a functional defined on $H^{s}(U) \times H^{s'}(U)$. Furthermore, for any $\alpha,s'\in(0,1)$ such that $\alpha+s'-1$, $B$ can also be extended in a unique continuous way to a functional defined on $C^{\alpha}(U)\times H^{s'}(U)$.
\end{cor}

\subsection*{Entropy, Fisher information and Kullback-Leibler divergence}

We start this section by recalling the definition of Fisher information associated with a probability measure and the one of entropy with respect to the Lebesgue measure in $\R^n$. 

\begin{dfn}[Normalized Entropy]
For any probability measure $\mu^N$ on $\R^{n N}$, absolutely continuous with respect to the Lebesgue measure $\mathscr{L}^{n N}$, having density $\mu^N(\d x^{(N)})=\rho_N(x^{(N)}) \d x^{(N)}$, we define the \emph{entropy (with respect to the Lebesgue measure)} by setting
\[\mathcal{H}_N(\mu^N)=\int_{\R^N}\rho_N \log(\rho_N) \d  \mathscr{L}^{nN}.\]
We define also the \emph{normalized entropy} as $H_N(\mu^N)=\frac{1}{N}\mathcal{H}_N(\mu^N)$.
\end{dfn}

Hereafter, if it is clear from the context, we drop the index $N$ from the symbol of the entropy $\mathcal{H}_N$ by simply writing $\mathcal{H}$. For the functional $\mathcal{H}$ the following result holds:

\begin{prop}
   The functional $\mathcal{H}$ is convex and lower-semicontinuous with respect to the weak convergence of measures (independently from $n, N \in \N$).
\end{prop}
\begin{proof}
    The proof can be found in \cite[Theorem 3.4]{KacChaos}.
\end{proof}

\begin{dfn}[Normalized Fisher information, \cite{KacChaos}] Let $\mu_N \in \prob{\R^{nN}}$, absolutely continuous with respect the Lebesgue measure with  $\mu^N(\d x^{(N)})=\rho_N(x^{(N)}) \d x^{(N)}$. If $\rho_N \in W^{1, 1}(\R^{nN})$, we set
\[
\mathcal{I}_N(\mu_N) := \int_{\R^{nN}} \dfrac{|\nabla \rho_N|^2}{\rho_N} \, \d \mathscr{L}^{nN}
\]
otherwise we set $\mathcal{I}_N(\mu_N) := +\infty$. As in the case of the entropy, the \emph{normalized Fisher information} is given by $I_N (\mu^N) := \dfrac1N \mathcal{I}_N (\mu^N)$. In the following we will use the notation $I :=  I_1$.
\end{dfn}

Hereafter, if $\rho_N$ is a probability density on $\R^{nN}$, we denote by $\rho_N^{(N| k)}$ the projection of $\rho_N$ on the first $k$-coordinates, i.e.
\[
\rho_N^{(N|k)}(x_1, \dots, x_k) := \int_{\R^{N - k}} \rho_N (x_1, \dots, x_k, y_{k + 1}, \dots, y_N) \, \d y_{k+1} \cdots \d y_N.
\]
Accordingly, if $\mu_N = \rho_N \mathcal L^{nN} \in \probp{\R^{nN}}{p}$, we define $\mu_N^{(N| k)} := \rho_N^{(N| k)} \mathcal L^{nk}$ to be the correspondent measure in $\probp{\R^{nk}}{p}$. We summarize some useful properties satisfied by the Fisher information $\mathcal I_N$:
\begin{prop}
    Let $\mu_N \in \probp{\R^{nN}}{2}$ be invariant with respect to coordinate permutations. Then we have
    \begin{itemize}
        \item[i)] $\mathcal I_N(\mu_N)$ is a convex and lower semi-continuous functional with respect to the  $W_2$ convergence;
        \item[ii)] $ I_k\left(\mu_N^{(N | k)}\right) \le I_N(\mu_N)$ for any $1 \le k \le N$;
        \item[iii)] $\mathcal{I}_N$ is super-additive, meaning that for any $k = 1, \dots, N$ it holds
        \[
         \mathcal{I}_N\left(\mu_N\right) \ge  \mathcal{I}_k\left(\mu_N^{(N | k)}\right) + \mathcal{I}_{N - k}\left(\mu_N^{(N | N-k)}\right).
        \]
        If $ \mathcal{I}_k\left(\mu_N^{(N | k)}\right) +  \mathcal{I}_{N - k}\left(\mu_N^{(N-k)}\right) < +\infty$, then the case of the equality in the above inequality can occur if and only if $\mu_N = \mu_N^{(N | k)} \mu_N^{(N| N-k)}$;
        \item[iv)] if  $I(\mu_N^{(N|1)}) < + \infty$, then $ I\left(\mu_N^{(N|1)}\right) = I_N(\mu_N)$ if and only if $\mu_N = \left(\mu_N^{(N|1)}\right)^{\otimes N}$.
    \end{itemize}
\end{prop}
\begin{proof}
    The proof can be found in \cite[Lemma 3.5, Lemma 3.6, Lemma 3.7]{KacChaos}.
\end{proof}

Hereafter, if it is clear from the context, we drop the index $N$ from the symbol of the Fisher information $\mathcal{I}_N$ by simply writing $\mathcal{I}$. The Entropy functional and Fisher information are related by the following proposition.

\begin{prop}\label{proposition:EntropyFisher}
Suppose that $\mu^N_t$ is an absolutely continuous curve of measure in $\probp{\R^{n N}}{2}$, and consider $(\mu^N,w^N) \in \Lambda_{\mu_0^N}$. Suppose further that $\mathcal{H}(\mu^N_0)<+\infty$ and that $\int_0^T{\mathcal{I}(\mu_t^N)\dt}<+\infty$, then for any $t\in [0,T]$ we get
\[\mathcal{H}(\mu_t^N) \leq \mathcal{H}(\mu_0^N) +\int_0^t\int_{\R^{n N}}|w_s^N(x^{(N)})|^2 \mu^N_s(\d x^{(N)}) \d s+\int_0^t\mathcal{I}(\mu^N_s) \d s. \]
Furthermore we have the equality
\[\mathcal{H}(\mu_t^N)-\mathcal{H}(\mu_s^N)=\int_s^t\langle w^N_{\tau}(x^{(N)}), \nabla \mu^N_{\tau}(x^{(N)}) \rangle \d x^{(N)} \d \tau. \]
\end{prop}
\begin{proof} The proof is a consequence of 
   \cite[Theorem 10.4.6]{AGSBook}.
\end{proof}

\begin{rmk}\label{rmk:EntFis}
We recall that if a curve $\{\mu_t\}_{t \in [0, T]}$ is a solution to a Fokker-Planck equation where the drift $b \in L^2(\mathbb{R}^n\times[0,T], \mu_t(\d x)\otimes \dt)$ and the initial condition $\mu_0$ has finite entropy, $\mathcal H(\mu_0)<+\infty$, then the logarithmic gradient of $\mu_t$ is a well defined measurable function. In particular, it belongs to $L^2(\mathbb{R}^n\times[0,T], \mu_t(\d x)\otimes \dt)$, meaning that $\int_0^T{\mathcal{I}(\mu_t)\dt}<+\infty$. This observation follows from \cite[Theorem 7.4.1]{BogachevRockner2015}. Moreover,  as a consequence of Proposition \ref{proposition:EntropyFisher}, if $\{\mu_t\}_{t \in [0, T]}$ is a solution to a Fokker-Planck equation where the drift $b \in L^2(\mathbb{R}^n\times[0,T], \mu_t(\d x)\otimes \dt)$ and the initial condition $\mu_0$ has finite entropy, $\mathcal H(\mu_0)<+\infty$ and $\mu_0 \in \probp{\R^{n}}{2}$, then for any $t\in[0,T]$ we have $\mathcal{H}(\mu_t)<+\infty$.\qed
\end{rmk}

We conclude this section by introducing the notion of \emph{Kullback-Leibler divergence}, also called \emph{relative entropy}.

\begin{dfn}
Let $\mathbb P$ and $\mathbb Q$ be two probability laws in the same probability space $(\Omega,\mathcal{F})$ such that $\mathbb P$ is absolutely continuous with respect to $\mathbb Q$. Then we define the Kullback-Leibler divergence between $\mathbb P$ and $\mathbb Q$ by setting
\[
D_{KL}\left(\mathbb P | \mathbb Q\right) := \int_{\Omega} \log \left( \dfrac{\d \mathbb P}{\d \mathbb Q} (\omega)\right) \, \mathbb P (\d \omega).
\]
In the case in which $\mathbb P$ is not absolutely continuous with respect to $\mathbb Q$, we set $D_{KL} \left(\mathbb P | \mathbb Q\right) := +\infty$.
\end{dfn}

We end this section by recalling an important property of Kullback-Leibler divergence used in the following.

\begin{prop}\label{proposition:DKLcompactness}
    Fix a probability measure $\mathbb{Q}$ on a probability space $(\Omega,\mathcal{F})$, then for any $c>0$ the set of probability measures $\{\mathbb{P}\;| \mathcal{H}(\mathbb{P}|\mathbb{Q})<c\}$ is compact with respect to the topology given by the weak convergence of probability measures.
\end{prop}
\begin{proof}
    See, e.g., \cite[Lemma 1.4.3]{DupuisEllis1997}.
\end{proof}

\subsection*{de Finetti's theorem}

In this section we discuss some properties of (infinite) sequence $Y_1,...,Y_N,...$ of \emph{exchangeable} random variables taking values in some Polish space ${S}$. Namely, we say that and infinite sequence of random variable is exchangeable if for any set $i_1,...,i_k\in \mathbb{N}$ such that $i_j\not= i_{\ell}$ for any $j\not= \ell$, we have 
\[(Y_{i_1},...,Y_{i_k}) \sim (Y_1,...,Y_k), \]
i.e. the random vector $(Y_{i_1},...,Y_{i_k})$ has a law independent of $i_1,...,i_k$ (which in turn depends only on $k$). This is equivalent to say that $Y_1,...,Y_N,...$ has a law which is invariant with respect to permutations. We refer to \cite[Section 1.1]{Kall} for the proof of the following fundamental result:
\begin{thm}[de Finetti's theorem]\label{theorem:deFinetti}
Let $S$ be a Polish space and let $\mathcal{S}$ be the associated Borel $\sigma$-algebra. Consider a countable sequence $Y_1,...,Y_N,...$ of random variables (defined on some probability space $(\Omega,\mathcal{F},\mathbb{P})$) taking values in $S$ and whose law is invariant with respect to permutations. Then there is a random probability measure $\nu\colon\Omega\rightarrow \mathscr{P}(S)$ such that 
\begin{enumerate}
\item  for any $B\in\mathcal{S}$ and $i\in\mathbb{N}$ we have $\mathbb{P}(Y_i\in B|\nu)=\nu(B)$,
\item the sequence $Y_1,...,Y_N,...$ is a sequence of independent random variables conditioned with respect to $\nu$.
\end{enumerate}
Furthermore we have 
\[\lim_{N\rightarrow +\infty}\frac{1}{N}\sum_{j=1}^N\delta_{Y_j} \rightharpoonup \nu\]
weakly and almost surely with respect to $\mathbb{P}$.
\end{thm}

\begin{rmk}
    An important consequence of Theorem \ref{theorem:deFinetti} is the following: under the same hypotheses of Theorem \ref{theorem:deFinetti}, let $\mu^N$ be the probability law on $\mathcal{S}^N$ of the random vector $(Y_1,...,Y_N)$. Then it holds
    \[\mu^{(N)}(\d x^{(N)})=\mathbb{E}_{\mathbb{P}}[\nu(\d x^{(N),1}) \otimes \cdots \otimes \nu(\d x^{(N),N})].\]
    This last identity should be understood in a distributional sense, meaning that for any Borel measurable function $f$ on $\mathcal{S}^N$, we have
    \[\int_{\mathcal{S}^N}f(x^{(N)}) \mu^{(N)}(\d x^{(N)})=\mathbb{E}_{\mathbb{P}}\left[\int_{\mathcal{S}^N}f(x^{(N)}) \nu^{\otimes N}(\d x^{(N)})\right].\]\qed
\end{rmk}

\subsection*{Some useful convergence results}

Aim of this section is to present some general results, which will be useful in the following.

\begin{lemma}\label{lemD}
    Let $\mathbb P$ be a fixed probability measure on $\Omega$ and $\mathcal D \colon [0, T] \times \Omega \to \R$ be a measurable function such that $\mathbb E_{\mathbb P}\big[ \int_0^T |\mathcal D(t, \cdot)|^2 \, \dt \big] < +\infty$. Then there exists another measurable function $D \colon [0, T]\times \R^n \to \R$ such that $\int_0^T \int_{\R^n} |D(t, x)|^2 \, \mu(\d x) \, \dt < +\infty$ and for any $f \in C_b([0, T] \times \R^n)$ it holds
    \begin{equation}\label{eq:Dexp=}
    \int_0^T \int_{\R^n} D(t, x) f(t, x) \, \mu_t(\d x) \, \dt = \mathbb E_{\mathbb P}\bigg[ \int_0^T \mathcal D (t, \cdot) f(t, X_t( \cdot )) \, \dt \bigg].
    \end{equation}
    We also have that
    \begin{equation}\label{eq:Dexple}
    \int_0^T \int_{\R^n} D^2(t, x)  \, \mu_t(\d x) \, \dt \le \mathbb E_{\mathbb P}\bigg[ \int_0^T \mathcal D^2 (t, \cdot) \, \dt \bigg].
    \end{equation}
\end{lemma}
\begin{proof}
Let us consider the space $\mathcal{L}:=L^2([0,T] \times \mathbb{R}^n, \dt \otimes \mu_t(\d x))$ and the map $T\colon\mathcal{L} \rightarrow L^2([0,T]) \otimes L^2(\mathbb{P}) ) $ that sends any $f(x,t) \in \mathcal{L}$ into the function $f(X_t,t)$. 
Since the probability law of $X_t$ at time $t$ is $\mu_t$, the map $T$ is a linear isometry of Hilbert spaces, in fact we have that
\begin{equation}\label{eq:isomorphism}
\int_0^T\int_{\mathbb{R}^n}f(x,t)g(x,t)\mu_t(\d x)\dt=\int_0^T\mathbb{E}_{\mathbb P}[f(X_t,t)g(X_t,t)]\dt.
\end{equation}
Thus $T$ is an isomorphism into the subspace of $L^2(\mathbb{P})\otimes L^2([0,T])$ which has the form
\[\tilde{\mathcal{L}}=\overline{\{f(X_t,t)|f\in C^0_b(\mathbb{R}^n\times[0,T]) \} }. \]
If we consider $\mathcal{D}(t, \omega)\in L^2([0,T])\otimes L^2(\mathbb{P})$, we can project the process $\mathcal{D}$ onto the subspace $\tilde{\mathcal{L}}$ and then map it back to the space $\mathcal{L}$, i.e. we consider the function $D \in \mathcal{L}$ defined as
\[D=T^{-1}(P_{\tilde{\mathcal{L}}}(\mathcal{D})),\]
where $P_{\tilde{\mathcal{L}}}: L^2([0,T])\otimes L^2(\mathbb{P}) \rightarrow \tilde{\mathcal{L}}$ is the orthogonal projection of $ L^2([0,T])\otimes L^2(\mathbb{P})$ onto $\tilde{\mathcal{L}}$. Now we notice that for any $f\in C^0_b([0,T] \times \mathbb{R}^n)$ we have
\[\int_0^T\mathbb{E}[\mathcal{D}(t,\cdot)f(t,X_t)]\dt=\int_0^T\mathbb{E}[P_{\tilde{\mathcal{L}}}(\mathcal{D})f(t,X_t)]\dt,\] 
since $f(t,X_t)\in\tilde{\mathcal{L}}$. By applying to the previous equality the relation \eqref{eq:isomorphism}, we get the thesis.
\end{proof}

\begin{rmk}\label{rmk:expP}
    Slightly abusing the notation, we say that $D(t, x)$ is a realization of the conditional expectation of $\mathcal D(t, \omega)$ with respect to the random variable $X_t$ if
    \[D(t, X_t) = \mathbb E_{\mathbb P} \big[ \mathcal D(t, \cdot) | X_t \big] \text{ for any } t \in [0, T]. \]
    $\hfill \qed$
\end{rmk}

\begin{lemma}
For any $\bm{\lambda} \in \mathscr{P}_2(H^1([0,T],\mathbb{R}^n))$ there exists a curve $\{\bm{\lambda}_t\}_{t\in [0,T]} \in H^1([0,T],\mathscr{P}_2(\mathbb{R}^n))$ such that 
\[(e_t){\sharp}{\bm \lambda}={\bm \lambda}_t. \]
\end{lemma}
\begin{proof}
If we consider
\[w_t(x)=\mathbb{E}\left[\left.\frac{\d X_t}{\dt}\right| X_t=x \right],\]
it is simple to prove that $w_t(x)$ satisfies the continuity equation associated to  $\bm \lambda_t$. In fact, for any $f \in C^1(\R^n)$ it holds
\[
\begin{split}
\int f(x) \, {\bm \lambda}_t(\d x) - \int f(x) \, & {\bm \lambda}_s(\d x) = \mathbb E\Big[ f(X_t) - f(X_s) \Big] = \mathbb E \Bigg[ \int_s^t \nabla f(X_r) \, \cdot \, \dfrac{\d X_r}{\d r} \, \d r \Bigg]\\
& = \int_s^t \mathbb E\Bigg[  \nabla f(X_r) \cdot \mathbb E \bigg[ \dfrac{\d X_r}{\d r} \bigg| X_r\bigg] \Bigg] \, \d r = \mathbb E \bigg[ \int_s^t   \nabla f(X_r) \cdot w_t(x) \, \d r \bigg],
\end{split}
\]
where in the last equality we have used Fubini's theorem. Finally,  by Lemma \ref{lemD} we have 
\[\int_0^T\int_{\mathbb R^n}{|w_t(x)|^2{\bm \lambda}_t(\d x)\dt} \leq \mathbb{E}\left[ \int_0^T \left|\frac{\d X_t }{\dt}\right|^2 \dt \right], \]
which ensures that $ t \mapsto \bm{\lambda}_t$ is a map in $ H^1([0,T],\mathscr{P}_2(\mathbb{R}^n))$.
\end{proof}

\begin{lemma}\label{lemma:Simon}
Let $X,Y$ be two Banach spaces such that $X \hookrightarrow Y$, namely $X$ is compactly immersed in $Y$. Consider $Z_n$ a sequence in $L^p(\Omega,X)$ such that $\sup_{n\in\mathbb{N}}\mathbb{E}[\| Z_n\|^p_X]<+\infty$, and for which there exists a subspace $D\subset X^*$, separating the points of $Y$, and $Z\in L^p(\Omega,X)$ such that for any $f\in D$ it holds $f(Z_n) \rightarrow f(Z)$ almost surely. Then, for any $1\leq q <p$, the sequence $Z_n$ converges to $Z$ in $L^q(\Omega,Y)$.
\end{lemma}
\begin{proof}
The sequence of random variables $Z_n-Z$ is tight in $Y$ (or better the sequence of probability laws of $Z_n-Z$ is tight in $Y$). Indeed, by hypothesis, there is a constant $C>0$ such that
\[\mathbb{E}[\|Z_n-Z \|^p_X]  <C \quad \text{ for any } n \in \mathbb{N}.\]
Therefore, denoting by $B_R(0)$ the ball of $X$ radius $R>0$ and center $0\in X$, a direct application of Markov inequality provides
\[\mu_n(B_R(0)^c)<\frac{C}{R^p},\]
where $\mu_n$ is the probability law of $Z_n-Z$. Since $X$ compactly embedded in $Y$, and thus $B_R(0)$ is a compact subset of $Y$, then the sequence of $\mu_n$ is tight in $Y$. Let $\bar{\mu}$ be any limit of a weakly convergent subsequence of $\mu_n$. Since, for any $x \in D$, we have $x(Z_n-Z) \rightarrow 0$ almost surely and since $D$ separates the points of $Y$, we get that that $\bar{\mu}=\delta_0$. Since the convergence of a sequence of probability measures to a Dirac delta is equivalent to the convergence in probability, we obtain that $Z_n$ converges in probability to $Z$ in the Banach space $Y$. This fact implies that $\| Z_n- Z\|_{Y}^q$ converges to $0$ in probability. Furthermore, since the sequence of random variable $\| Z_n- Z\|_{Y}^q$ is uniformly integrable, by \cite[Chapter 13, Section 13.7]{Williams}, we have $\mathbb{E}[\| Z_n- Z\|_{Y}^q]\rightarrow 0$.
\end{proof}

\begin{rmk}\label{Rmk:Simon}
We can modify the statement and the proof of Lemma \ref{lemma:Simon} in the following way. Consider a Banach space $X$ and let $\{\nu_{n}\}_{n \in \N}$ be a sequence of probability measures on $X$, weakly convergent to some probability measure $\nu$ on $X$. Suppose further that there is some $p>1$ for which 
\[\sup_{n \in \mathbb{N}}\int_X \| x\|^p_X \nu_n(\d x) < +\infty.\]
Then we get that the sequence $\{\nu_n\}_{n \in \N}$ converges to $\nu$ in $\probp{X}{q}$ for any $1\leq q <p$. In fact, we do not need to consider a bigger space $Y$ to guarantee the tightness of the sequence $\{\nu_n\}_{n \in \N}$ since this property follows from the weak convergence of $\{\nu_n\}_{n \in \N}$ to $\nu$.   $\hfill \qed$
\end{rmk}

\begin{lemma}\label{lemma:Wpcompact}
Let $X$ and $Y$ be two Polish spaces, for which $X$ is immersed in $Y$ and the bounded subsets of $X$ are compact subset of $Y$, then, for any $1 < q < p<+\infty$,  the bounded subsets of $\probp{X}{p}$ are compact in $\probp{Y}{q}$.
\end{lemma} 
\begin{proof} 
The proof is similar to the one of Lemma \ref{lemma:Simon}.
\end{proof}

\begin{cor}\label{cor:convpromossa}
Let $\{\mu_n\}_{n \in \N}$ be a bounded sequence of measures in $\probp{X}{p}$ and suppose that there exists a measure $\mu \in \probp{X}{p}$ and a subset $K\subset C^0_b(X)$, separating the points of $\probp{X}{p}$, for which $\lim_{n \rightarrow +\infty}\int_X f(x) \mu_n(\d x)=\int_X f(x) \mu(\d x)$. Then the sequence $\{\mu_n\}_{n \in \N}$ converges to $\mu$ in $\probp{Y}{q}$.
\end{cor}

\begin{cor}\label{cor:convpromossaE}
Let $\{\mu_k\}_{k \in \N}$ be a sequence of random measures defined in some probability space $(\Omega,\mathcal{F},\mathbb{P})$ such that 
$\mu_k \in \probp{\mathbb{R}^n}{2}$ almost surely, for any $k \in \N$, which in particular means that
\[\sup_{k\in\mathbb{N}}\mathbb{E}[W_2^2(\delta_0,\mu_k)]<+\infty.\]
If there exists $\mu \in \probp{\mathbb{R}^n}{2}$ such that, for any $f\in C_b(\mathbb{R}^n)$ we have 
\[ \lim_{k\rightarrow +\infty}\int_{\mathbb{R}^n} f(x) \mu_k(\d x)=\int_{\mathbb{R}^n} f(x) \mu(\d x),\] 
then, for any $1\leq p <2$,
\[\lim_{k\rightarrow +\infty}\mathbb{E}[W_p^p(\mu_k,\mu)]=0.\]
\end{cor}
\begin{proof}
The validity of the following inequality is proven in \cite[Theorem 6.9]{Villani}: 
\begin{equation}
    \label{eq:inequalityWppi}
    W_p^p(\mu_k,\mu)\leq \int(|x-y|^p\wedge R^p)\pi_k(\d x,\d y)+2 
 ^p\int_{B_R^c} |x|^p \mu_k(\d x)+ 2^p \int_{B_R^c} |x|^p \mu(\d x),
\end{equation}
where $\pi_k$ is an optimal coupling between $\mu_k$ and $\mu$, which converges weakly to the Dirac delta on the diagonal when $k \rightarrow +\infty$, since $\mu_k$ converges to $\mu$ weakly.
Taking the expectation and then the limsup as $k\rightarrow +\infty$ in inequality \eqref{eq:inequalityWppi}, we get 
\begin{equation}
\limsup_{k\rightarrow +\infty} \mathbb{E}[W_p^p(\mu_k,\mu)]\leq 2^p \limsup_{k\rightarrow +\infty} \mathbb{E}\left[\int_{B_R^c} |x|^p \mu_k(\d x)\right]+ 2^p \mathbb{E}\left[\int_{B_R^c} |x|^p \mu(\d x),\right] \label{eq:limsupWp} 
\end{equation}
where we use that, by Lebesgue's dominate convergence theorem, it holds 
\[\lim_{k\rightarrow +\infty}\mathbb{E}\left[\int(|x-y|^p\wedge R^p)\pi_k(\d x,\d y)\right]=0.\]
Taking $p<2$ and applying Holder inequality, we obtain
\begin{equation}
\mathbb{E}\left[\int_{B_R^c} |x|^p \mu_k(\d x)\right] \leq \left(\mathbb{E}\left[\int |x|^2 \mu_k(\d x)\right]\right)^{\frac{p}{2}} (\mathbb{E}[\mu_k(B_{R}^c)])^{\frac{2-p}{2}}.\label{eq:inequalityHolder}
\end{equation}
Now, recalling that $\mu_k$ converges weakly to $\mu$, using again the Lebesgue's dominated converge theorem and combining equations \eqref{eq:limsupWp} and \eqref{eq:inequalityHolder}, we obtain
\begin{equation}\limsup_{k\rightarrow +\infty} \mathbb{E}[W_p^p(\mu_k,\mu)]\leq 2^p(C+1) \left(\mathbb{E}\left[\int_{B_R^c} |x|^p \mu(\d x),\right]\right)^{\frac{2-p}{p}} \label{eq:limitWp}
\end{equation}
where $C=\sup_{k\in \mathbb{N}}\left(\mathbb{E}\left[\int |x|^2 \mu_k(\d x)\right]\right)^{\frac{p}{2}}$. Since 
\[\lim_{R \rightarrow +\infty} \int_{B_R^c} |x|^p \mu(\d x)=0 \quad \text{ and } \quad \sup_{k\in \mathbb{N}}\mathbb{E}\left[\int |x|^p \mu_k(\d x)\right] <+\infty,\]
by taking the limit as $R\rightarrow 0$ in equation \eqref{eq:limitWp}, we get 
\[\limsup_{k \rightarrow +\infty} \mathbb{E}[W_p^p(\mu_k,\mu)]=0\]
which concludes the proof.
\end{proof}

\begin{lemma}[Ascoli-Arzel\'a in non-locally compact spaces]\label{lem:AscArz}
Let $(X, d)$ be a metric space and let $\{f_{n}\}_{n \in \N}$ be a sequence of continuous functions, $f_n \colon [0,T] \rightarrow X$ for any $n \in \N$. Let us suppose that $\{f_{n}\}_{n \in \N}$ are equicontinuous and take values in a compact subset of $X$. Then there is a subsequence $n_k \rightarrow +\infty$ and a continuous function $f \colon [0,T]\rightarrow X$ such that $f_{n_k} \overset{k \to +\infty}{\longrightarrow} f$ uniformly. 
\end{lemma}
\begin{proof}
We can apply the same argument  as in \cite[Theorem 7.23]{Rudin}. In fact, the statement of this theorem holds also for functions taking values in a generic metric spaces $(X, d)$, provided that their images are contained in a compact subset of $(X, d)$.
\end{proof}

We conclude this section with a result on the convergence of stochastic integrals under weaker conditions than the usual convergence with respect to the semimartingale norm. We use the notation $\mathbb{P}_W$ to denote the probability law of the Brownian motion on $\Omega$.

\begin{lemma}\label{lemma:convergencestochasticintegral}
 Consider a sequence of measures $\{\mathbb{P}_M^{(k)}\}_{M \in \N}$ on $\Omega^k$ converging weakly to $\tilde{\mathbb{P}}^{(k)}$. Suppose that each $\mathbb{P}_M^{(k)}$ is absolutely continuous with respect to the Wiener measure $\mathbb{P}^{\otimes k}_W$ on $\Omega^k$ with $\sup_M D_{KL}(\mathbb{P}_M^{(k)}|\mathbb{P}^{\otimes k}_W) <+\infty $. Then, for any continuous bounded function $K:\mathbb{R}_+ \times \mathbb{R}^{k n} \rightarrow \mathbb{R}^{k n}$, we have
 \begin{equation}\label{eq:limK}\lim_{M \to +\infty}\mathbb{E}_{\mathbb{P}_M^{(k)}}\left[\int_0^T K\left(t,X^{(k)}_t\right) \cdot \d X^{(k)}_t\right]= \mathbb{E}_{\tilde{\mathbb{P}}^{(k)}}\left[\int_0^T K\left(t,X^{(k)}_t\right) \cdot \d X^{(k)}_t\right].\end{equation}
\end{lemma}
\begin{proof}
We prove the theorem in the case in which $k=1$ and the function $K$ is a $C^1$ bounded function. Hence the general theorem follows by similar methods, thanks to the density with respect to the pointwise convergence of the class of $C^1$ functions in the space of continuous functions. Hereafter, to simplify the notation, we write $\mathbb{P}_M :=\mathbb{P}_M^{(k)}$ and $\bar{\mathbb{P}} :=\bar{\mathbb{P}}^{(k)}$.\\

For every $M\in \mathbb{N}$ we have that
\[ \int_0^T K (t, X_t) \cdot \d X_t = \lim_{| \pi |
   \rightarrow 0} \sum_{t_i \in \pi} K (t, X_{t_{i - 1}}) (X_{t_i}
   - X_{t_{i - 1}})\]
where $\pi$ is a partition of $[0, T]$ and the limit is taken in $L^2(\mathbb{P}_M)$ as the size of the partition $|\pi|$ is going to zero. On the other hand 
\[\lim_{M\to +\infty}\mathbb{E}_{\mathbb{P}_M} \left[ \sum_{t_i \in \pi} K (t, X_{t_{i - 1}}) (X_{t_i}
   - X_{t_{i - 1}})\right]=\mathbb{E}_{\bar{\mathbb{P}}} \left[ \sum_{t_i \in \pi} K (t, X_{t_{i - 1}}) (X_{t_i}
   - X_{t_{i - 1}}) \right]\]
   since $\mathbb{P}_M$ converges to $\bar{\mathbb{P}}$ weakly and $\sup_M D_{KL}(\mathbb{P}_M|\mathbb{P}_W) <+\infty $. Thus the only thing that we need to prove to conclude the argument is to show that we can exchange the order of the limits $|\pi|\to 0$ and $M \to +\infty$.
   A sufficient condition for this is to prove that 
   \begin{equation}
    \mathbb{E}_{\mathbb{P}_M} \left[ \left| \int_0^T K (t, X_t) \cdot \d X_t - \sum_{t_i \in \pi} K (t, X_{t_{i -
   1}}) (X_{t_i} - X_{t_{i - 1}}) \right|^2 \right] \leq c
   |\pi|^{\alpha} \end{equation}
   for some $\alpha>0$ small enough and for a suitable constant $c$, independent of $M$. We start by observing that, since each $\mathbb{P}_M$ is absolutely continuous with respect to $\mathbb{P}_W$, we have 
   \[ \d X_t=r_t \dt +\d W_t
   \text{ for some function } r_t  \text{ such that } \int_0^T|r_s|^2\d s <\infty \text{ almost surely,}
\]
where $W_t$ is some $\mathbb{P}_M$-Brownian motion.
   Furthermore, since 
   \[D_{KL}(\mathbb{P}_M|\mathbb{P}_W)=\mathbb{E}_{\mathbb{P}_M}\left[\int_0^T|r_t|^2\dt \right],\]
   it follows that  $\sup_{M\in\mathbb{N}} \mathbb{E}_{\mathbb{P}^M}\left[\int_0^T|r_t|^2\dt \right] = c'<+\infty$.
   Writing $R_t=\int_0^tr_s \d s$, we have
   \scriptsize
\[ \begin{split}
\mathbb{E}_{\mathbb{P}_M} \left[ \left| \int_0^T K (t, X_t) \cdot \d X_t - \sum_{t_i \in \pi} K (t, X_{t_{i -
   1}}) (X_{t_i} - X_{t_{i - 1}}) \right|^2 \right]  \leqslant  &2\mathbb{E}_{\mathbb{P}_M} \left[ \left| \int_0^T K (t,
   X_t^{N, (1)}) \cdot \d R_t - \sum_{t_i \in \pi} K (t, X_{t_{i -
   1}}) (R_{t_i} - R_{t_{i - 1}}) \right|^2 \right] \\ +  &2\mathbb{E}_{\mathbb{P}_M} \left[ \left| \int_0^T K (t, X_t) \cdot \d W_t - \sum_{t_i \in \pi} K (t, X_{t_{i - 1}})
   (W_{t_i} - W_{t_{i - 1}}) \right|^2 \right].
   \end{split}\]
   \normalsize
By It\^o isometry, the last term in the above expression is equal to
\[ \mathbb{E}_{\mathbb{P}_M} \left[ \sum_{t_i \in \pi} \int_{t_{i -
   1}}^{t_i} | K (t, X_t) - K (t, X_{t_{i - 1}}) |^2 \dt \right]\]
   while, for the other term, we get
   \small
\[\begin{split} \mathbb{E}_{\mathbb{P}_M} \left[ \left| \int_0^T K (t, X_t) \cdot \d R_t - \sum_{t_i \in \pi} K (t, X_{t_{i-1}})
   (R_{t_i} - R_{t_{i - 1}}) \right|^2 \right]  \leq \mathbb{E}_{\mathbb{P}_M} \left[ \sum_{t_i \in \pi}
   \int_{t_{i - 1}}^{t_i} | K (t, X_t) - K (t, X_{t_{i-1}}) | |r_t| \d t \right]& \\ \leq 2 \left( \mathbb{E}_{\mathbb{P}_M} \left[ \int_0^T |r_t|^2 \d t \right] \right)^{\frac{1}{2}} \left(\mathbb{E}_{\mathbb{P}_M}
   \left[ \sum_{t_i \in \pi} \int_{t_{i - 1}}^{t_i} | K (t, X_t) -
   K (t, X_{t_{i - 1}}) |^2 \dt \right]\right)^{1/2}&. \end{split} \]
   \normalsize
At this point, we want to show that $\mathbb{E}_{\mathbb{P}_M}
   \left[ \sum_{t_i \in \pi} \int_{t_{i - 1}}^{t_i} | K (t, X_t) -
   K (t, X_{t_{i - 1}}) |^2 \dt \right]$ goes to zero uniformly in $M$, as $|\pi| \to 0$. First we note that
   \[ \begin{split} \mathbb{E}_{\mathbb{P}_M}
   \left[ \sum_{t_i \in \pi} \int_{t_{i - 1}}^{t_i} | K (t, X_t) -
   K (t, X_{t_{i - 1}}) |^2 \dt \right] \leq \| K \|_{\infty}\mathbb{E}_{\mathbb{P}_M} \left[ \sum_{t_i \in
   \pi} \int_{t_{i - 1}}^{t_i} | K (t, X_t^{N, (1)}) - K (t, X_{t_{i -
   1}}^{N, (1)}) | \dt \right]& \\ \leq  \| K \|_{\infty} \| K \|_{\text{Lip}} \, T \, \sup_{t_{i - 1} \in
   \pi} \mathbb{E}_{\mathbb{P}_M} \left[\sup_{s \in [t_{i - 1}, t_i]} | X_t - X_{t_{i - 1}} |\right]. &
\end{split}\]
Furthermore we get
\[ \sup_{t \in [t_{i - 1}, t_i]} | X_t - X_{t_{i - 1}} | \le \int_{t_{i
   - 1}}^{t_i} |r_t| \dt + \sup_{t \in [t_{i - 1}, t_i]}
   (W_t - W_{t_{i - 1}}).\]
   On the other hand, by Doob inequality and the invariance of the law of Brownian motion increments with respect to time translation, it also holds
\[ \mathbb{E}_{\mathbb{P}_M} \left[\sup_{t \in [t_{i - 1}, t_i]} \left(W_t - W_{t_{i -
   1}}\right)^2\right] \le \mathbb{E}_{\mathbb{P}_M} [W_{t_i-t_i-1}^2] \leq |\pi| .\]
while, by H\"older inequality, we have
\[ \mathbb{E}_{\mathbb{P}_M} \left[ \int_{t_{i - 1}}^{t_i} |r_t|
   \dt \right] \le 2 \sqrt{t_i - t_{i - 1}}
   \mathbb{E}_{\mathbb{P}_M} \left[ \int_{t_{i - 1}}^{t_i} |r_t|^2 \dt \right] \le 2 |\pi|^{1/2} c'.\]
Summing up all these inequalities, we get \eqref{eq:limK}.
\end{proof}

{\small \subsection*{Notation}

It is convenient to summarize in this table the notation we will use in the following:
\begin{center}
\begin{tabular}{|c|c|}
\hline
\rule[-4mm]{0mm}{1cm}
$\Omega$ & $C^0([0, T], \R^n)$ endowed with the uniform topology;\\
\hline
\rule[-4mm]{0mm}{1cm}
$(X_t)_{0 \le t \le T}$ & the canonical process on $\Omega$;\\
\hline
\rule[-4mm]{0mm}{1cm}
$\mathcal F_t = \sigma \big(X_t; 0 \le t \le T\big)$ & the canonical filtration on $\Omega$;\\
\hline
\rule[-4mm]{0mm}{1cm}
$\mu_t$ & the law of $X_t$;\\
\hline
\rule[-4mm]{0mm}{1cm}
$\sigma_{ij} \colon \mathbb{R}^{nN} \rightarrow \mathbb{R}^{n N}$ & the map switching the $i$-th and $j$-th particle;\\
\hline
\rule[-4mm]{0mm}{1cm}
    $X^{(N)}$ & $N$-dimensional vector, $X^{(N)} \in \R^{nN}$; \\
\hline
\rule[-4mm]{0mm}{1cm}
   $\pi^{N,k}, k < N$ & $\pi^{N,k} \colon \R^{nN} \to \R^{nk}$ is the projection in the first $k$-coordinates; \\
\hline
\rule[-4mm]{0mm}{1cm}
    $X^{(N| k)}, k < N$ & the projection of $X^{(N)}$ in the first $k$-coordinates, $X^{(N| k)} \in \R^{nk}$;\\ 
\hline
\rule[-4mm]{0mm}{1cm}
    $X^{(N), k}$ & the $k$-th component of the vector $X^{(N)}$, $X^{(N), k} \in \R^{n}$,\\ & i.e., $X^{(N| k)} = \pi^{(N, k)} \left(  X^{(N)} \right);$ \\
\hline
\rule[-4mm]{0mm}{1cm}
$\Lambda^p_{\tilde \mu_0}$ & all the couples $\big\{ \mu_t, w_t  \big\}_{t \in [0, T]}$ such that  $\mu \in {\rm AC} \big([0, T], \probp{\R^n}{p}\big)$\\ & with $\mu_0 = \tilde \mu_0 $,  $w_t \in L^p(\R^n, \mu_t) $ and 
the continuity equation \eqref{eq:CE} holds.\\
\hline
\end{tabular}
\end{center}}

\section{Formulation of the problem}
\subsection{Mean-field problem system}\label{section:meanfieldproblem}

We start this section by introducing the hypotheses for the functions $b$, $\mathcal{V}$ and $\mathcal{G}$, which define the controlled equation \eqref{eq:SDEintro} and the cost functional \eqref{eq:costintro}, under which we will state all our results in the paper.

\begin{assumptions}\label{hyp} Hereafter we will work with three functionals
\[
\mathcal V \colon \R^n \times \PX_2(\R^n) \to \R, \,\, b \colon \R^n \times \PX_2(\R^n) \to \R^n \,\, \text{ and } \,\, \mathcal G \colon \probp{\R^n}{2} \to \R,
\]
satisfying the following properties:
\begin{itemize}
    \item[$(\mathcal{QV})$] there exist an exponent $p \in [1,  2)$ 
    for which, for any $ x \in \R^n$ and $\mu \in \PX_2(\R^n)$, it holds
    \begin{equation}\label{eq:Vp} 
    \big|\mathcal{V}(x, \mu)\big| \lesssim 1+ |x|^{p} +  \int_{\R^n} |y|^{p} \, \mu(\d y).
    \end{equation}
We also assume that  $\mathcal{V}$ can be continuously extended to $\R^n \times \probp{\R^n}{p}$ and this extension is continuous.
    \item[$(\mathcal{Q}b)$]
   there exists  $p\in [1,2)$ such that, for any $x\in \R^n$ and $\mu\in\probp{\R^n}{2}$,
    \[|b(x, \mu)|^2 \lesssim 1+|x|^{p}+\int_{\R^n}|y|^{p}\mu(\d y).\]
    Moreover we assume that
    \begin{enumerate}
    \item  the functional $ b$, defined on $\R^n \times \probp{\R^n}{2}$, can be continuously extended to $\R^n \times \probp{\R^n}{p}$ and we require that this extension is $\alpha$-H\"older continuous, with $\alpha > 0$, with respect to the natural metric of $\R^n \times \probp{\R^n}{p}$,
    \item  $b (\cdot, \mu) \in D(\Div_{\R^n})$ for any fixed $\mu \in \PX_p(\R^n)$;
    \item  $ \Div_{\R^n}b(x, \mu)$ can be continuously extended to $\R^n \times \probp{\R^n}{p}$ and this extension is continuous. Furthermore for any $x \in \R^n$ and any $\mu \in \probp{\R^n}{p}$ we assume the validity of the following inequality: \[|\Div_{\R^n}b(x, \mu)| \lesssim 1+|x|^{p}+\int_{\R^n}{|y|^{p} \mu(\d y)}.\]
    \end{enumerate}
\item[$(\mathcal Q \mathcal G)$] there is a $p\in [1,2)$ such that $\mathcal G$ can be continuously extended  continuous to $\probp{\R^n}{p}$, and, for any $\mu \in \probp{R^n}{p}$, the following bound holds true \[\mathcal G(\mu) \ge -C \left(1 + \int |y|^p \, \d \mu(y)\right) \text{ for some constant } C > 0.\]
\end{itemize}
\end{assumptions}
\vspace{0.5cm}

\begin{rmk}
\begin{enumerate}
\item Observe that Fatou's lemma guarantees the lower semicontinuity of the sequence of norms $\Big\{|| \cdot ||_{L^2(\mu_n)}\Big\}_{n \in \N}$. Therefore it is not restrictive to assume that our functionals are lower semicontinuous with respect to the notion of convergence given by the weak convergence of measures and the boundness of the second moments.
\item The Assumption $(\mathcal{Q}b)$ is not optimal. We refer to Section \ref{section:generalization} for some possible generalizations.
\end{enumerate}\qed
\end{rmk}
 
Before introducing the rigorous definition of a weak solution to an SDE, we recall that, given a probability space $(\Omega, \mathcal F, \mathbb P)$ and a filtration $\{ \mathcal F_t \}_{t\in[0,T] }$, of the $\sigma$-algebras, of $\mathcal F$ (for some a suitable set of indexes), a measurable function $\alpha \colon [0, T] \times \Omega \to \R^n$ is said to be $\mathcal F_t$-adapted if the map $\alpha_t \colon \Omega \to \R^n$ is $\mathcal F_t$-measurable for any fixed $t \in [0, T]$. The following definition is equivalent to the standard definition of weak solution of an SDE (see e.g. \cite{Watanabe1989}).

\begin{dfn} Let $\Omega  := C^0([0,  T], \R^n)$ be equipped with the topology of uniform convergence and let $\mathcal F_t$ be the filtration introduced in Section \ref{section:preliminary}.

Given a probability measure $\mathbb P \in \PX(\Omega)$, a functional $b \colon \R^n \times \probp{\R^n}{2} \to \R^n$ satisfying the hypothesis $(\mathcal Q b)$ in Assumptions \ref{hyp} and a $\mathcal F_t$-adapted function $\alpha \colon [0, T] \times \Omega \to \R^n$, we say that the probability $\mathbb P$ is a \emph{weak solution} to the additive noise SDE
\begin{equation}\label{eq:SDE0}
\d X_t(\omega) = \Big( \alpha(t,\omega) + b(X_t(\omega), \mu_t) \Big) \, \dt + \sqrt{2} \d  W_t(\omega)
\end{equation}
with coefficient $b$, control $\alpha$ and initial condition $\mu_0 \in \PX_2(\R^n)$ if the canonical process $X_t$ is such that:
\begin{enumerate}
\item  $X_0 \sim_{\mathbb P} \mu_0$; 
\item the measure $\mu_t  := \Law(X_t)$ has finite Fisher information, i.e., $\mu_t \in \mathcal D(\mathcal I)$;
\item the process defined by 
\[W_t(\omega) := \frac{1}{\sqrt 2} \bigg(X_t(\omega) - X_0(\omega) - \int_0^t \Big(\alpha\left(s, \omega\right) + b(X_s(\omega), \mu_s)\Big) \, \d s \bigg) \]
is a Brownian motion with respect to the probability $\mathbb P$, taking values in $\R^n$.
\end{enumerate}
\end{dfn}

\begin{rmk}
Since $\alpha_t$ is a $\mathcal F_t$-adapted and $X_t$ is the canonical process on $\Omega$, we can write
\[
\alpha\big(t, \omega\big) = A\big(t, X_{[0, t]}(\omega)\big)
\]
for some $A \colon [0, T] \times \Omega \to \R^n $ with the property that $A(t, \cdot)$ is measurable. Here we use the symbol $X_{[0, t]}$ to denote the process $X_{\cdot}$ restricted to the interval $[0, t]$. Using this notation, if $\mathbb P$ is a weak solution to \eqref{eq:SDE0}, then the canonical process $X_t$ induced by $(\Omega, \mathcal F, \mathbb P)$ is a solution to the (non-Markovian) SDE
\begin{equation}\label{eq:SDE}
\d X_t = \Big( A\big(t, X_{[0, t]}\big) + b(X_t, \mu_t)  \Big) \, \d t  + \sqrt 2 \d W_t. 
\end{equation}
\qed\end{rmk}

We now introduce the families of closed loop controls which we consider in the rest of the paper. For any fixed measure $\mu_0 \in \PX_2(\R^n)$, we define
\begin{equation}\label{eq:defA}
\begin{split} \mathcal A_{\mu_0} := \bigg\{\, A \colon [0, T] \times \Omega \to \R^n \, \Big| \, A(t, \cdot) \text{ is measurable and there exists } \, \mathbb P\in \PX(\Omega) \text{ which is a }& \\ \text{weak solution of \eqref{eq:SDE} with control } \, 
 A\left(t, X_{[0, t]}\right) \, \text{ and } \mathbb{E}_{\mathbb{P}}\left[\int_0^T\Big| A(t, X_{[0, t]})  \Big|^2 \, \dt \right]< +\infty\bigg\}&.\end{split}  \end{equation}
\noindent In a similar way, for two fixed measures $\mu_0, \mu_T \in \probp{\R^n}{2}$, we define the set of controls $\mathcal A_{\mu_0, \mu_T}  \subset \mathcal A_{\mu_0}$
\[
\mathcal A_{\mu_0, \mu_T} :=1\left.\Big \{ A \in \mathcal A_{\mu_0} \,\,\right| \,\,\Law(X_T) = \mu_T  \Big\}.
\]

\begin{prop}\label{prop:uniqP}
    For any control $A \in \mathcal A_{\mu_0}$, 
     there exists a unique $\mathbb P \in \PX(\Omega)$ such that 
     \begin{enumerate}
     \item $\mathbb P$ is a weak solution of the SDE \eqref{eq:SDE} associated with $A$,
     \item  $\displaystyle \mathbb{E}_{\mathbb{P}}\left[\int_0^T\Big| A(t, X_{[0, t]})  \Big|^2 \, \dt\right] < +\infty$.
     \end{enumerate}
 Moreover, $(e_{t})_{\sharp}(\mathbb P)=\mu_t$ belongs to $L^{\infty}([0,T],\probp{\R^n}{2})$, and $\mathbb P$ is absolutely continuous with respect to $\mathbb P_{W,\mu_0}$, i.e. the Wiener measure translated by the initial condition $X_0 \sim \mu_0$, with Radon-Nikodym derivative given by 
\begin{equation}\label{eq:densP}
    \dfrac{\d \mathbb P}{\d \mathbb P_{W,\mu_0}} = \exp \Bigg( \int_0^T  \Big( A(t, X_{[0, t]}) + b(X_t, \Law(X_t)) \Big) \, \d X_t - \frac{1}{2} \int_0^T | A (t, X_{[0, t]}) + b(X_t, \Law(X_t)) |^2 \, \dt  \Bigg).
    \end{equation}
\end{prop}

\begin{proof}
    First we note that, since $|b(x,\mu)|\lesssim 1+ |x|^{\frac{p}{2}} +\left(\int_{\R^n} |y|^p \mu(\d y)\right)^{\frac{1}{2}}$ for some $p<2$, by applying Ito formula to the function $|X_t|^2$ and taking the expectation, we obtain the following apriori estimate
    \[\mathbb{E}_{\mathbb{P}}\left[|X|_t^2\right]\lesssim \left(\mathbb{E}_{\mathbb{P}}[|X_0|^2]+T+\mathbb{E}_{\mathbb{P}}\left[\int_0^t{|A(s,X_{[0,s]})|^2\d s}\right]\right) + \int_0^t\mathbb{E}_{\mathbb{P}}[|X_s|^2]\d s. \]
    Thus, by an application of Gr\"onwall's inequality, we get 
    \[\mathbb{E}_{\mathbb{P}}\left[|X|_t^2\right] \lesssim \left(\mathbb{E}_{\mathbb{P}}[|X_0|^2]+T+\mathbb{E}_{\mathbb{P}}\left[\int_0^t{|A(s,X_{[0,s]})|^2\d s}\right]\right) e^{C t} \]
    for some constant $C>0$. Since $\mu_0\in \probp{\R^n}{2}$ and $\mathbb{E}_{\mathbb{P}}\left[\int_0^T{|A(s,X_{[0,s]})|^2\d s}\right] <+\infty$, we have that $\mu_{\cdot}\in L^{\infty}([0,T],\probp{\R^n}{2})$.
    Furthermore, by observing that observe $X_t$ is non-explosive, we obtain
    \begin{multline*}
        \int_0^T |b(X_t, \Law(X_t))|^2 \, \dt \overset{(\mathcal Q b.1)}{\lesssim} \int_0^T \left(1 + |X_t|^p +\int_{\R^n}|y|^p\mu_{t}(\d y)\right) \, \dt \\
        \lesssim \left(  1+ \sup_{t \in [0, T]} |X_t|^p+ \sup_{t \in [0, T]} \int_{\R^n}|y|^2 \mu_t(\d y)  \right) < +\infty
   \end{multline*}
   $\mathbb P$-almost surely. Since $A \in \mathcal A_{\mu_0}$ we actually have
    \[ \int_0^T |A(t, X_{[0, t]}) + b(X_t, \Law(X_t))|^2 \, \dt < +\infty\]
    $\mathbb P$-almost surely. As a direct consequence of \cite[Corollary 3.4]{Ruf}, we can apply Girsanov's theorem to the weak solution of the SDE \eqref{eq:SDE} in order to obtain that $\mathbb P$ is absolutely continuous with respect to $\mathbb P_{W,\mu_0}$  and has density given by \eqref{eq:densP}. From this expression it also follows that $X_t$ is a Brownian motion with initial law $\mu_0$. 
    
    Finally, the uniqueness of the measure $\mathbb P$ is a direct consequence of the fact that the measure $\mathbb P_{W,\mu_0}$ is uniquely defined by the property that the canonical process $X_t$ is a Brownian motion with initial law $\mu_0$ and that the density of $\mathbb P$ with respect to $\mathbb P_{W,\mu_0}$ is is uniquely given by \eqref{eq:densP}.
\end{proof}

We are interested in studying the controlled equation \eqref{eq:SDE} and the stochastic optimal control problems associated with the following cost functionals:
\begin{enumerate}
\item {\bf Schr\"odinger cost functional (or problem) with potential $\mathcal V$, initial condition $\mu_0$ and final condition $\mu_T$, with $\mathcal{H}(\mu_0),\mathcal{H}(\mu_T)<+\infty$}:
\begin{equation}\label{def:CSchr}
    \mathcal C_{\mu_0, \mu_T}(A) := \mathbb E_{X_0 \sim \mu_0, \, X_T \sim \mu_T}\bigg[\int_0^T \bigg( \dfrac{|A(t, X_{[0, t]})|^2}{2} + \mathcal V\big(X_t, \Law(X_t)\big)\bigg) \, \dt \bigg],
\end{equation}
where $A \in \mathcal A_{\mu_0, \mu_T}$ and $X_t$ is a weak solution of \eqref{eq:SDE} with prescribed initial and final data,
\item {\bf finite-horizon cost functional (or problem) with potential $\mathcal V$, initial condition $\mu_0$, with $\mathcal{H}(\mu_0)<+\infty$, and final condition on the cost function $\mathcal{G}$ at the final time $T$}:
\begin{equation}\label{def:Cfin}
    \mathcal C_{\mu_0, \mathcal G}(A) := \mathbb E_{X_0 \sim \mu_0} \bigg[\int_0^T \bigg( \dfrac{|A(t, X_{[0, t]})|^2}{2} + \mathcal V\big(X_t, \Law(X_t)\big)\bigg) \, \dt  + \mathcal G \big(\Law (X_T)\big)\bigg],
\end{equation}
where $A \in \mathcal A_{\mu_0}$ and $X_t$ is a weak solution of \eqref{eq:SDE} with prescribed initial data.
\end{enumerate}
We remark that the functionals $\mathcal V$ and $\mathcal G$ satisfy the hypothesis $(\mathcal{QV})$ and $(\mathcal{QG})$ in Assumptions \ref{hyp}, respectively, and that the expectation is taken with respect to the unique probability measure $\mathbb P$ obtained in Proposition \ref{prop:uniqP}.  

\begin{rmk}\label{rmk:FiniteEntropy}
It is important to notice that whenever the entropy of the initial measure is finite, i.e., $\mathcal H (\mu_0) = \int_{\mathbb{R}^n}{\log(\mu_0(x))\mu_0(\d x)}<+\infty$, there exists at least a control $A\in\mathcal{A}_{\mu_0}$ such that $\mathcal{C}_{\mu_0,\mathcal{G}}(A)<+\infty$. In fact, if $\{\tilde{\mu}_t\}_{t \in [0, T]}$ is the solution to the heat equation with initial condition $\tilde{\mu}_0 = \mu_0$, the control $\tilde{A}(X_{[0,t]},t) :=-b(X_t,\tilde{\mu}_t)$ belongs to $\mathcal{A}_{\mu_0}$  since it holds
\[\d X_t =\tilde{A}(t, X_{[0,t]})+b(X_t,\tilde{\mu_t})+\sqrt{2} \d W_t=\sqrt{2}\d W_t\]
and the growth condition, which is at most linear, of $b(\cdot,\mu)$ expressed in $(\mathcal Q b)$ guarantees that $\int_0^T\mathbb{E}[|b(X_t,\tilde{\mu}_t)|^2]\dt<+\infty$ (see, e.g., \cite[Theorem 6.6.2]{BogachevRockner2015}).

Moreover, in the case of the Schr\"odinger problem, when also the entropy of the final measure is finite, i.e., $\mathcal H(\mu_T) = \int_{\mathbb{R}^n}{\log(\mu_T(x))\mu_T(\d x)}<+\infty$, there exists a bridge $\{\bar{\mu}_t\}_{t \in [0, T]} \subset \probp{\R^n}{2}$ between $\bar{\mu}_0=\mu_0$ and $\bar{\mu}_T=\mu_T$ (see the literature related to the {S}chr\"{o}dinger problem in $\mathbb{R}^n$, for example \cite[Section 4.1]{Leonard2014}). This curve of measures $\{\bar{\mu}_t\}_{t \in [0, T]}$ is the solution to an SDE with additive noise and square integrable drift $C(t,X_t)$. As before, if we take $\bar{A}(t, X_{[0,t]})=C(t, X_t)-b(X_t,\bar{\mu}_t)$, we get that $\bar{A}(t, X_{[0,t]})\in\mathcal{A}_{\mu_0,\mu_t}$ and $\mathcal{C}_{\mu_0,\mu_T}(\bar{A})<+\infty$. In particular, this guarantees that the set $\mathcal A_{\mu_0, \mu_T}$ is non-empty provided that $\mu_0, \mu_T \in \probp{\R^n}{2}$ have finite entropy. \qed
\end{rmk}

We conclude this section by showing that the minimum of the cost functionals $\mathcal C_{\mu_0, \mu_T}$ and $\mathcal C_{\mu_0, \mathcal G}$ is attained in a set of controls with a precise form. For that purpose, we first introduce the notion of Markovian control by setting
\[
\mathcal A_{\mathcal M,\mu_0} := \Big\{\, A \in \mathcal A_{\mu_0} \Big|  \text{ there exists }  \tilde A \colon \R^+\times \R^n\to \R^n \text{ s.t. } A(t, X_{[0, t]}) = \tilde A(t, X_t) \Big\}.
\]
In a similar way we consider the set of Markovian controls $\mathcal{A}_{\mathcal{M},\mu_0,\mu_T}$ for the case of Sch\"odinger problem to be
\[
\mathcal A_{\mathcal M,\mu_0, \mu_T} := \Big\{\, A \in \mathcal A_{\mu_0, \mu_T} \Big|  \text{ there exists }  \tilde A \colon \R^+\times \R^n\to \R^n \text{ s.t. } A(t, X_{[0, t]}) = \tilde A(t, X_t) \Big\}.
\]
Hence in the below result we show that we can restrict our minimization problem to the set of  Markovian controls $\mathcal A_{\mathcal M, \mu_0}$ and $\mathcal A_{\mathcal M, \mu_0, \mu_T}$.

\begin{prop}\label{prop:Markovian}
Under the hypotheses $(\mathcal{QV})$, $(\mathcal{Q}b)$ and $(\mathcal{QG})$, for any $A\in \mathcal{A}_{\mu_0,\mu_T}$ there is a $\tilde{A}\in\mathcal{A}_{\mathcal{M},\mu_0,\mu_T}$ for which $\mathcal{C}_{\mu_0,\mu_T}(\tilde{A}) \leq \mathcal{C}_{\mu_0,\mu_T}(A)$, and the equality in the previous inequality holds if and only if $A(t,X_{[0,t]})=\tilde{A}(t,X_t)$ $\dt \otimes \d\mathbb{P}$-almost surely. A similar statement holds for the functional $\mathcal{C}_{\mu_0,\mathcal{G}}$. In particular this implies that
    \[
    \min_{A \in \mathcal A_{\mu_0, \mu_T}} \mathcal C_{\mu_0, \mu_T}(A) = \min_{A \in \mathcal A_{\mathcal M,\mu_0, \mu_T}} \mathcal C_{\mu_0, \mu_T}(A)
    \]
    and similarly for the finite-horizon problem.
\end{prop}
\begin{proof} We present this proof in the case of the Schr\"odinger problem. The one for the finite-horizon problem follows exactly the same argument. 
To prove the first point of the proposition, let us fix $A \in \mathcal A_{\mu_0, \mu_T}$ and consider $\tilde A(t, X_t) := \mathbb E_{\mathbb P}[A(t, \cdot) | X_t]$ in the sense of Lemma \ref{lemD} and Remark \ref{rmk:expP}.  Thus $\tilde A \in \mathcal A_{\mathcal M,\mu_0, \mu_T}$.  In fact this definition guarantees that $\tilde A$ is Markovian and, by Ito's formula together with \eqref{eq:Dexp=}, if $\mathbb P$ solves the equation \eqref{eq:SDE} with control $A$, then it solves also the same equation with control $\tilde A$. 
    Therefore inequality \eqref{eq:Dexple} guarantees that
    \begin{equation}\label{eq:Markovian}
    \mathbb E_{\mathbb P} \Big[ |A(t, \cdot)|^2 \Big] \ge \mathbb E_{\mathbb P} \Big[ |\tilde A(t, X_t)|^2 \Big].    \end{equation}
    This allows to conclude that for any $A\in \mathcal{A}_{\mu_0, \mu_T}$, we have $\mathcal C_{\mu_0, \mu_T}(\tilde A) \le \mathcal C_{\mu_0, \mu_T}(A)$, where the equality holds if and only if we have an equality in equation \eqref{eq:Markovian}, and thus when $A(t,X_{[0,t]})=\bar{A}(t,X_t)$ $\dt \otimes \d\mathbb{P}$-almost surely.
    
    The previous argument implies that $\min_{\alpha \in \mathcal A_{\mu_0, \mu_T}} \mathcal C_{\mu_0, \mu_T}(A) \geq \min_{\alpha \in \mathcal A_{\mathcal M,\mu_0, \mu_T}} \mathcal C_{\mu_0, \mu_T}(A)$. By noticing that since $\mathcal{A}_{\mathcal{M},\mu_0, \mu_T}\subset \mathcal{A}_{\mu_0, \mu_T}$, it immediately follows that 
   \[ \min_{A \in \mathcal A_{\mu_0, \mu_T}} \mathcal C_{\mu_0, \mu_T}(A) \leq \min_{A \in \mathcal A_{\mathcal M,\mu_0, \mu_T}} \mathcal C_{\mu_0, \mu_T}(A),\]
   and the proposition is proved.
\end{proof}

\begin{rmk}
    An important consequence of the first part of Proposition \ref{prop:Markovian} is that if $A\in\mathcal{A}_{\mu_0,\mu_T}$ (or $A\in\mathcal{A}_{\mu_0}$) is a minimizer of $\mathcal{C}_{\mu_0,\mu_T}$ (or $\mathcal{C}_{\mu_0,\mathcal{G}}$) it must belong to the set of Markovian controls $\mathcal{A}_{\mathcal{M},\mu_0,\mu_T}$ (or $\mathcal{A}_{\mathcal{M},\mu_0}$). This means that, without loss of generality, we can limit our analysis considering Markovian controls in $\mathcal{A}_{\mathcal{M},\mu_0,\mu_T}$ (or $\mathcal{A}_{\mathcal{M},\mu_0}$). For this reason, if not explicitly said otherwise, when we write $A\in\mathcal{A}_{\mu_0,\mu_T}$ (or in  $\mathcal{A}_{\mu_0}$) we implicitly assume that $A$ is Markovian.\qed
\end{rmk}

\subsection{The N-particle approximating problem}\label{sec:Npartpb}

Let us notice that the dependence on the law of $X_t$ in equation \eqref{eq:SDE} implies that the canonical process $X_t$, with respect to the weak solution $\mathbb P$, is in general \emph{not} a Markov process. However, equation \eqref{eq:SDE} can be obtained also as the limit of a suitable sequence of Markovian $N$-particle systems. For this reason we introduce and study the following $N$-particles  approximation problem.\\

Let $\Omega^N$ be the canonical probability space and $X^N_t := (X_t^{N, 1}, \dots, X_t^{N, N}) \in \Omega^N$ be the canonical process. We consider the SDE 
\begin{equation}\label{eq:NpartSDE}
\d X_t^{N, i} = \Bigg( A^{N, i} \Big(t, X^N_t\Big) + b \Bigg(X_t^{N, i}, \dfrac1N \sum_{j=1}^N \delta_{X^{N, j}_t}\Bigg)\Bigg) \, \dt + \sqrt{2} \, \d W_t^{N, i}.
\end{equation} 
where $b \colon \R^n \times \probp{\R^n}{2} \to \R^n$ is a vector field satisfying $(\mathcal Q b)$ in Assumptions \ref{hyp}, $W_t^{N, i}$, $i =1, \dots, N$ are independent Brownian motions taking values in $\R^n$ and  $A^N := (A^{N, 1}, \dots, A^{N, N}) \colon [0, T] \times \R^{nN} \to \R^{nN}$ is such that
\begin{enumerate}
\item each $A^{N, i} \in \mathcal A_{\mu_0, \mu_T}$ for $\mu_0, \mu_T \in \PX_2(\R^n)$, in the case of the Schr\"odinger problem;
\item each $A^{N, i} \in \mathcal A_{\mu_0}$ for a fixed $\mu_0 \in \PX_2(\R^n)$, in the case of the finite horizon problem.
\end{enumerate}
In the following we will write  $\mathcal A^N_{\mu_0, \mu_T}$ to denote the set of all the $A^N$ such that each $A^{N, i} \in \mathcal A_{\mu_0, \mu_T}$ and, similarly, $\mathcal A^N_{\mu_0}$ for the set of all the $A^N$ such that each $A^{N, i} \in \mathcal A_{\mu_0}$.

It is important to notice  that, despite the fact that the Brownian motions $W_t^{N, i}$ $i = 1, \dots, N$ are independent, the processes $X^{N, i}_t$, $i = 1, \dots, N$ are in general \emph{not} independent since each $A^{N, i}$ is a function of all the variables $(x^1, \dots, x^N)$ and not only of the variable $x^i$.\\ 

Given the functionals $b, \mathcal Q$ and $\mathcal V$ as above and two measures $\mu_0 \in \PX(\R^n)$ and $\mu_T \in \PX(\R^n)$, our purpose is to optimize \eqref{eq:NpartSDE} with respect to the following cost functionals, normalized with respect to the number of particles $N$:
\begin{enumerate}
\item {\bf Schr\"odinger cost functional (or problem) with potential $\mathcal V$ and initial and final condition $\mu_0$ and $\mu_T$, with $\mathcal{H}(\mu_0),\mathcal{H}(\mu_T)<+\infty$}:
\begin{equation}\label{eq:NdimSp}
\begin{split}
\mathcal C^N_{\mu_0, \mu_T} &(A^{N, 1}, \dots A^{N, N}):=  \\ &\dfrac1N \sum_{i = 1}^N \mathbb E_{X_0^{N, i} \sim \mu_0, X^{N, i}_T \sim \mu_T} \Bigg[  \int_0^T \Bigg(   \dfrac{|A^{N,i}(t,X^N_t)|^2}{2} + \mathcal V \Bigg( X_t^{N,i}, \dfrac1N \sum_{j=1}^N \delta_{X_t^{N,j}}  \Bigg) \Bigg) \, \dt  \Bigg],
\end{split}
\end{equation}
where $(A^{N, 1}, \dots A^{N, N}) \in \mathcal A^N_{\mu_0, \mu_T}$ and $X_t^N$ is a weak solution of \eqref{eq:NpartSDE} with prescribed  initial and final condition,
\item {\bf finite-horizon cost functional (or problem) with potential $\mathcal V$, initial condition $\mu_0$, with $\mathcal{H}(\mu_0)<+\infty$, and penalization final cost $\mathcal G$}:
\begin{equation}\label{eq:NdimFH}
\begin{split}
\mathcal C^N_{\mu_0, \mathcal G}  (A^{N, 1}, \dots A^{N, N})  :=
\dfrac1N \sum_{i = 1}^N & \mathbb E_{X_0^{N, i} \sim \mu_0} \Bigg[  \int_0^T \Bigg(   \dfrac{|A^{N, i}(t, X^N_t)|^2}{2} + \mathcal V \Bigg( X_t^{N, i}, \dfrac1N \sum_{j=1}^N \delta_{X_t^{N, j}}  \Bigg) \Bigg) \, \dt  \Bigg]\\ &+\mathbb{E}_{X_0^{N,i}\sim \mu_0}\left[\mathcal G \left( \dfrac 1N \sum_{j=1}^N \delta_{X_t^{N,j}}  \right) \right],
\end{split}
\end{equation}
where $(A^{N, 1}, \dots A^{N, N}) \in \mathcal A^N_{\mu_0}$ and $X_t^N$ is a weak solution of \eqref{eq:NpartSDE} with prescribed  initial condition.
\end{enumerate}

In the following we will denote the empirical measure that appears in the above functionals as
\[ \iota^N := \dfrac1N \sum_{j=1}^N \delta_{X^{N, j}}.\]

\subsection{Statement of the main results}\label{section:results}

The main goal of this paper is to prove the following results. The first of them concerns the convergence of the ``value function'' of the $N$-particle approximating problem, i.e. the minimum value attended by the $N$-particles cost function:

\begin{thm}\label{theorem:main1}
    Under Assumptions \ref{hyp}, the functionals $\mathcal C_{\mu_0, \mu_T}$ in \eqref{def:CSchr} and $\mathcal C_{\mu_0, \mathcal G}$ \eqref{def:Cfin}
 as well as the $N$-particles counterparts $\mathcal C^N_{\mu_0, \mu_T}$ in \eqref{eq:NdimSp} and $\mathcal C^N_{\mu_0, \mathcal G}$ in \eqref{eq:NdimFH} admit a minimum. In particular it holds
\begin{equation}\label{eq:thmMinC0T}
 \Theta_{\mu_0, \mu_T} := \min_{A \in \mathcal A_{\mu_0,\mu_T}} \mathcal C_{\mu_0, \mu_T} (A) = \lim_{N \to +\infty} \min_{A^N \in \mathcal A^N_{\mu_0,\mu_T}} \mathcal C^N_{\mu_0, \mu_T} (A^{N, 1}, \dots A^{N, N})
 \end{equation}
 as well as
 \begin{equation}\label{eq:thmMinC0}
 \Theta_{\mu_0} := \min_{A \in \mathcal A_{\mu_0}} \mathcal C_{\mu_0, \mathcal G} (A) = \lim_{N \to +\infty} \min_{A^N \in \mathcal A^N_{\mu_0}} \mathcal C^N_{\mu_0, \mathcal G} (A^{N, 1}, \dots A^{N, N}).
 \end{equation}
 \end{thm}
 
The other two results deal with the convergence  on the path space $\Omega$ of the probability measures which are laws of the solution to the original equations \eqref{eq:NpartSDE} and \eqref{eq:SDE}. In order to simplify the statement and the proofs of the following theorems we suppose the uniqueness of the optimal control $A^{\infty}$ for McKean-Vlasov problem (see Assumption \ref{assumption:uniqueness} for the precise statement of this hypothesis). \\

Let $\mathbb{P}^{(N)}_{\min}$ be the probability law on $\Omega^N$ of the solution $X_t^{(N)}$ to the controlled equation \eqref{eq:NpartSDE} when $A^N$ is one of the optimal controls with respect to the cost functional $\mathcal{C}_{\mu_0,\mathcal{G}}^N$ (or $\mathcal{C}_{\mu_0,\mu_T}^N$ in the Schr\"odinger case), and let $\mathbb{P}_{\min}$ be the correspondent probability law for the limit case (see Section \ref{section:KLdivergence}). We also write  $\mathbb{P}^{(\infty)}_{\min}=\mathbb{P}_{\min}^{\otimes \infty}$ and $\mathbb{P}^{(\infty|k)}_{\min}=\mathbb{P}_{\min}^{\otimes k}$

\begin{thm}\label{thm:Kacchaotic}
Let us assume Assumption \ref{assumption:uniqueness}, then, for any $1\leq p <2$, we have
\begin{equation}\label{eq:limWp}
\lim_{N\rightarrow +\infty}\sup_{t\in[0,T]} {W}_p(\mu^{N,(k)}_t,\mu^{\infty,(k)}_t) =0.\end{equation}
In other words the particle system \eqref{eq:NpartSDE} is Kac-chaotic, namely the above distance converges for each fixed time $t \in [0, T]$ to the limit \eqref{eq:SDE}.\\
Furthermore, for any $1<p<2$ and any $k\in\mathbb{N}$, we have
\[ \lim_{N\rightarrow +\infty} {W}_p(\mathbb{P}^{(N|k)}_{\min},\mathbb{P}^{(\infty|k)}_{\min}) =0.\]
\end{thm}

If we assume that the optimal control $A^{\infty}$ for the  McKean-Vlasov problem admits some regularity properties and does not grow at infinity too fast (see Assumption \ref{hyp:primo}), then we are able to prove the following stronger converge of the Kullback–Leibler divergence.

 \begin{thm}\label{thm:KLdiv}
     Let us assume Assumption \ref{hyp:primo}, then, for any $k\in\mathbb{N}$, we have
     \begin{equation}\label{eq:KLmain}\lim_{N \to +\infty} D_{K L} \big( \mathbb P^{(N|k)}_{\min} | \mathbb P_{\min}^{(\infty|k)}\big) = 0. \end{equation}
 \end{thm}

Thanks to convergence \eqref{eq:KLmain}, we can prove the convergence of $\mathbb{P}^{(N|k)}_{\min}$ to $\mathbb{P}_{\min}^{(\infty|k)}$ in total variation.

\begin{cor}
Under the same assumption of Theorem \ref{thm:KLdiv}, we have the following convergence
\[\lim_{N \to +\infty} \d_{TV} \big( \mathbb P^{(N|k)}_{\min}, \mathbb P_{\min}^{(\infty|k)}\big) = 0\]
 \end{cor}
 \begin{proof}
     The thesis is a consequence of the well known Csisz\'ar-Kullback inequality relating the Kullback–Leibler divergence and the total variation distance.
 \end{proof}

\section{Benamou-Brenier equivalent formulation of the problem}\label{sez:BB}

In this section we go back to the McKean-Vlasov optimal control problem 
described in Section \ref{section:meanfieldproblem}. We remark that the $N$-particle problem can be seen as a special case of the McKean-Vlasov system where the drift $b$ and the potential $\mathcal{V}$ do not depend on the measure $\mu$ but only on the position of the $N$-particle $X^N$.

\begin{prop}\label{prop:energy}
Let $b$ be a vector field, $\mathcal V$ be a potential  and $\mathcal G$ be a final penalization cost satisfying respectively $(\mathcal{Q}b)$, $(\mathcal{QV})$ and $(\mathcal{Q V})$ in Assumptions \ref{hyp}, and consider 
\begin{itemize}
\item $\mu_0, \mu_T \in \probp{\R^n}{2}$ such that  $\mathcal H(\mu_0), \mathcal H(\mu_T)<+\infty$ in the Schr\"odinger problem or 
\item  $\mu_0 \in \probp{\R^n}{2}$ such that $\mathcal H(\mu_0)<+\infty$ in the finite horizon problem.
\end{itemize}
For any couple $\{\tilde\mu_t, \tilde w_t\}_{t \in [0, T]} \in \Lambda_{\mu_0}$ (and additionally $\tilde \mu_T = \mu_T$  in the Schr\"odinger problem case), we write 
\begin{equation}\label{def:E2Sch}
\begin{split}
\mathcal E_{\mu_0, \mu_T}  & \Big(\{\tilde \mu_t, \tilde w_t \}\Big) :=  \dfrac12  \int_0^T \int_{\R^n}\left(| \tilde w(t, x)|^2+ | b(x, \tilde\mu_t)|^2\right)\tilde\mu_t(\d x)\dt + \mathcal H(\mu_T) - \mathcal H(\mu_0) \\
& + \frac{1}{2} \int_0^T \mathcal I(\tilde \mu_t)\dt   - \int_0^T \int_{\R^n} \left( \langle \tilde w(t, x), b(x, \tilde\mu_t) \rangle +  \Div_{\R^n} b(x, \tilde \mu_t) - \mathcal V(x, \tilde\mu_t) \right) \tilde\mu_t(\d x)\dt
\end{split}
\end{equation}
in the case of the Schr\"odinger problem and
\begin{equation}\label{def:E2FH}
\begin{split}
\mathcal E_{\mu_0, \mathcal G} & \Big(\{\tilde \mu_t, \tilde w_t \}\Big) := \dfrac12  \int_0^T \int_{\R^n}\left(| \tilde w(t, x)|^2+ | b(x, \tilde\mu_t)|^2\right)\tilde\mu_t(\d x)\dt + \mathcal H(\mu_T) - \mathcal H(\mu_0) \\
& + \frac{1}{2} \int_0^T \mathcal I(\tilde \mu_t)\dt   - \int_0^T \int_{\R^n} \left( \langle \tilde w(t, x), b(x, \tilde\mu_t) \rangle +  \Div_{\R^n} b(x, \tilde \mu_t) - \mathcal V(x, \tilde\mu_t) \right) \tilde\mu_t(\d x)\dt + \mathcal{G}(\mu_T).
\end{split}
\end{equation}
 in the case of the finite-horizon problem.
Then we have
\begin{equation}\label{eq:minCESch}
    \min_{A \in \mathcal A_{\mu_0, \mu_T}} \mathcal C_{\mu_0, \mu_T}(A) = \min_{\{\tilde \mu_t, \tilde w_t\}\in \Lambda_{\mu_0}, \tilde \mu_T = \mu_T} {\mathcal E}_{\mu_0, \mu_T}\Big(\{\tilde \mu_t, \tilde w_t\}_{t \in[0, T]}\Big)
\end{equation}
as well as
\begin{equation}\label{eq:minCESch2}
    \min_{A \in \mathcal A_{\mu_0}} \mathcal C_{\mu_0, \mathcal G}(A) = \min_{\{\tilde \mu_t, \tilde w_t\}\in \Lambda_{\mu_0}} {\mathcal E}_{\mu_0, \mathcal G}\Big(\{\tilde \mu_t, \tilde w_t\}_{t \in[0, T]}\Big).
\end{equation}
\end{prop}

\begin{rmk}\label{remark:bounbedness}
As a byproduct of Proposition \ref{prop:energy}, we get that, under the assumptions $\mathcal{QV}$, $\mathcal{Q}b$ and $\mathcal{QG}$, the functionals $\mathcal{C}_{\mu_0,\mu_T}$ and $\mathcal{C}_{\mu_0,\mathcal{G}}$ are bounded from below. Indeed, whenever $(\bar{\mu}_t,\bar{w}_t)$ solves the continuity equation, we have $W_2^2(\bar{\mu}_0,\bar{\mu}_t) \leq \int_0^t\int_{\R^n}|\bar{w}_s(x)|^2 \bar{\mu}_s(\d x) \d s$. Thus we obtain
\[\int_{\R^n}|x|^2\bar{\mu}_t(\d x) \leq \int_0^t\int_{\R^n}|\bar{w}_t(x)|^2 \bar{\mu}_s(\d x) \d s+ \int_{\R^n}|x|^2 \bar{\mu}_0(\d x).\]
From the previous inequality and the fact that, under assumptions  $\mathcal{QV}$, $\mathcal{Q}b$ and $\mathcal{QG}$, $|b(x,\bar{\mu}_t)|^2$, $\mathcal{V}(x,\bar{\mu}_t)$ and $\mathcal{G}$ are bounded from below by a (possibly negative) multiple of $|x|^p+1+\int_{\R^n}|y|^p \bar{\mu}_t(\d y)$, by the use of weighted Young inequality for products, we get that $\mathcal{E}_{\mu_0,\mu_T}$ and $\mathcal{E}_{\mu_0,\mathcal{G}}$ are bounded from below by a constant uniform with respect to $(\bar{\mu},\bar{w})\in\Lambda_{\mu_0}$. By the equations
\eqref{eq:minCESch} and \eqref{eq:minCESch2} we obtain that $\mathcal{C}_{\mu_0,\mu_T}$ and $\mathcal{C}_{\mu_0,\mathcal{G}}$ are bounded from below.\qed
\end{rmk}

\begin{refproof}{\ref{prop:energy}}
We prove the statement in the case of the Schr\"odinger problem. The result for the finite-horizon problem follows with an analogous argument.

We start by observing that  Proposition \ref{prop:Markovian} guarantees that when considering the minimization problem in \eqref{eq:minCESch}, the minimum for the functional $\mathcal C_{\mu_0, \mu_T}$ will be achieved for some control in $\mathcal{A}_{\mathcal{M},\mu_0,\mu_T}$. Moreover, we can identify any control $A\in \mathcal{A}_{\mathcal{M},\mu_0,\mu_T}$ with a vector field $A \colon \mathbb{R}^+ \times \mathbb{R}^n \rightarrow \mathbb{R}^n$.

Let us then fix a control $A\in \mathcal{A}_{\mathcal{M},\mu_0,\mu_T}$. Since the curve of measures $\{\tilde \mu_t\}_{t \in [0, T]}$ solves the Fokker-Planck equation 
\[\partial_t \tilde\mu_t=  \Delta \tilde\mu_t -\Div\left(\left(A(t, x)+b(x,\tilde\mu_t)\right)\tilde\mu_t\right), \]
in a distributional sense, with  $A(t, x)+b(x,\tilde\mu_t) \in L^2([0,T]\times\mathbb{R}^n, \dt\otimes \tilde\mu_t(\d x))$ and $\mathcal H(\mu_0)<+\infty$, then  $\int_0^T{\mathcal{I}(\tilde\mu_t)\dt}<+\infty$ by Remark \ref{rmk:EntFis}. This ensures that it is meaningful to consider  the measure vector field defined as
\begin{equation}\label{eq:Atow}
\tilde{w}(t, x) :=  A(t, x) + b(x, \tilde\mu_t) - \nabla \log(\tilde \mu_t).
\end{equation}
Obviously $\{ \tilde\mu_t, \tilde{w}_t\}_{t \in [0, T]}$ satisfies the continuity equation $\partial_t \tilde\mu_t+\Div(\tilde{w}_t \tilde\mu_t)=0$. Furthermore since $A\in\mathcal{A}_{\mathcal{M},\mu_0,\mu_t}$, $b(x,\tilde\mu_t)\lesssim |x|^p+1$, for some $p<2$, and $\int_0^T{\mathcal{I}(\tilde\mu_t)\dt}<+\infty$, it follows that $\int_0^T\int_{\mathbb{R}^n}|\tilde{w}_t(x)|^2\tilde\mu_t(\d x)\dt<+\infty$. Thus $\tilde{w}$ is an admissible vector field.

Conversely, suppose that the couple $\{\tilde\mu_t,\tilde{w}_t\}_{t \in [0, T]}$ satisfies the continuity equation $\partial_t \tilde\mu_t+\Div(\tilde{w}_t \tilde\mu_t)=0$ and that  $\int_0^T{\mathcal{I}(\tilde\mu_t)\dt}<+\infty$. Then the control 
\begin{equation}\label{controlAw}
A(t, x)= \tilde{w}_t(x)-b(x, \tilde\mu_t)+  \nabla \log(\tilde\mu_t)\end{equation}
belongs to $\mathcal{A}_{\mathcal{M},\mu_0,\mu_t}$. In fact  $\{\tilde\mu_t\}_{t \in [0 T]}$ clearly satisfies the Fokker-Plank equation with drift $A(t, x)+b(t,\tilde\mu_t)$ and, since $\tilde{w}_t,b(x,\tilde\mu_t),\nabla \log(\tilde\mu_t)\in L^2\left(\R^+ \times \R^n, \dt \otimes \tilde\mu_t(\d x)\right)$, we also have that $\int^T_0\int_{\mathbb{R}^n}{|A (t, x)|^2\tilde\mu_t(\d x) \dt} < + \infty$. Hence the superposition principle for SDEs (see e.g. \cite{BarbuRo,Trevisan2016}) ensures the existence of a probability measure $\mathbb{P} \in \prob{\Omega}$ which is a weak solution of the SDE \eqref{eq:SDE} with drift $A(t, x)+b(x,\tilde\mu_t)$.

The only thing left to prove is that whenever the control  $A$ and the vector field $\tilde{w}$ are related by equation \refeq{eq:Atow}, then $\mathcal E_{\mu_0, \mu_T}  \Big(\{\tilde \mu_t, \tilde w_t \}_{t \in [0, T]}\Big) =\mathcal{C}_{\mu_0,\mu_t}(A)$. A direct computation shows that when $A$ is of the form \eqref{controlAw}, the functional 
 $\mathcal{C}_{\mu_0,\mu_t}(A)$ is equal to the expression \begin{equation}\label{def:Enoncomplete}
 \dfrac12 \int_0^T \int_{\R^n} \left| \tilde w(t, x) + \nabla \log(\tilde\mu_t) - b(x, \tilde\mu_t)\right|^2 \, \tilde\mu_t(\d x) \, \dt + \int_0^T \int_{\R^n} \mathcal V(x, \tilde\mu_t) \, \tilde\mu_t(\d x) \, \dt.
\end{equation}
By further developing this formula, we get
\[\begin{split}
\mathcal{C}_{\mu_0,\mu_t}(A) = \dfrac12 &\int_0^T \int_{\R^n}\left(| \tilde w(t, x)|^2+ | b(x, \tilde\mu_t)|^2\right)\tilde\mu_t(\d x)\dt  + \frac{1}{2} \int_0^T \int_{\R^n} |\nabla \log(\tilde\mu_t)|^2 \tilde\mu_t(\d x)\dt \\ 
&- \int_0^T \int_{\R^n} \langle \tilde w(t, x), b(x, \tilde\mu_t) \rangle \tilde\mu_t(\d x)\dt - \int_0^T \int_{\R^n}  \langle b(x, \tilde\mu_t), \nabla \tilde\mu_t(x) \rangle \d x\dt  \\
&+ \int_0^T\int_{\mathbb{R}^n}\langle \tilde{w}_t(x), \nabla \log(\tilde\mu_t) \rangle \tilde\mu_t(\d x) \dt + \int_0^T \int_{\R^n} \mathcal V(x, \tilde\mu_t) \, \tilde\mu_t(\d x) \, \dt.
\end{split}\]
The conditions in ($\mathcal{Q}b$) allow to integrate by parts the term $\int_0^T \int_{\R^n}  \langle b(x, \mu_t), \nabla \mu_t(x) \rangle \d x\dt$, while the fact that the couple $\{ \mu_t, \tilde{w}_t \}_{t \in [0, T]}$ satisfies the continuity equation \eqref{eq:CE} together with the identity relating the entropy with the logarithm gradient of a measure (see for example \cite[Lemma 4.4]{ConfortiLeonard}) guarantees that
\[\int_0^T\int_{\mathbb{R}^n}\langle \tilde w(t, x), \nabla \log(\tilde\mu_t) \rangle \tilde\mu_t(\d x) \dt= \mathcal{H}(\mu_T)- \mathcal{H}(\mu_0).\]
Hence \eqref{def:Enoncomplete} takes the form
\begin{equation}
\begin{split}
& \mathcal{C}_{\mu_0,\mu_t}(A) = \dfrac12  \int_0^T \int_{\R^n}\left(| \tilde w(t, x)|^2+ | b(x, \tilde\mu_t)|^2\right)\tilde\mu_t(\d x)\dt + \frac{1}{2} \int_0^T \mathcal I(\tilde \mu_t)\dt \\
& \quad + \mathcal H(\mu_T) - \mathcal H(\mu_0)  - \int_0^T \int_{\R^n} \left( \langle \tilde w(t, x), b(x, \tilde\mu_t) \rangle + \Div_{\R^n} b(x, \tilde \mu_t) - \mathcal V(x, \tilde\mu_t) \right) \tilde\mu_t(\d x)\dt
\end{split}
\end{equation}
which is precisely the expression of $\mathcal E_{\mu_0, \mu_T}  \Big(\{\tilde \mu_t, \tilde w_t \}\Big)$ displayed in \eqref{def:E2Sch}.
\end{refproof}

A similar argument can be applied also in the case of the $N$-particle approximating problem in Section \ref{sec:Npartpb}, leading to the following formulation of the optimization problems in \eqref{eq:NdimSp}  and \eqref{eq:NdimFH} in terms of the minimization of suitable energies.

\begin{prop}\label{prop:CNEN}
Let $b$ be a vector field, $\mathcal V$ be a potential  and $\mathcal G$ be a final penalization cost satisfying respectively $(\mathcal{Q}b)$, $(\mathcal{QV})$ and $(\mathcal{Q V})$ in Assumptions \ref{hyp}, and consider 
\begin{itemize}
\item $\mu_0, \mu_T \in \probp{\R^n}{2}$ such that  $\mathcal H(\mu_0), \mathcal H(\mu_T)<+\infty$ in the Schr\"odinger problem or 
\item  $\mu_0 \in \probp{\R^n}{2}$ such that $\mathcal H(\mu_0)<+\infty$ in the finite horizon problem.
\end{itemize}
Then
\[
\min_{(A^{N, 1}, \dots A^{N, N}) \in \mathcal A^N_{\mu_0, \mu_T}} \mathcal C^N_{\mu_0, \mu_T} (A^{N, 1}, \dots A^{N, N}) = \min_{\{\mu^N_t, w^N_t  \}_{t \in [0, T]} \in \Lambda_{\mu_0^{\otimes N},\mu_T^{\otimes N}}} \mathcal E^N_{\mu_0, \mu_T}\big(\{ \mu_t^N, w_t^N \}\big)
\]
and similarly for the case of the finite-horizon problem
\[
\min_{(A^{N, 1}, \dots A^{N, N}) \in \mathcal A^N_{\mu_0}} \mathcal C^N_{\mu_0, \mathcal G} (A^{N, 1}, \dots A^{N, N}) = \min_{\{\mu^N_t, w^N_t  \}_{t \in [0, T]} \in \Lambda_{\mu_0^{\otimes N}}} \mathcal E^N_{\mu_0, \mathcal G}\big(\{ \mu_t^N, w_t^N \}\big),
\]
where the energy functionals defined on $\Lambda_{\mu^{\otimes N}_0}$ are given by
\[
\begin{split}
\mathcal E^N_{\mu_0, \mu_T}\big(\{ \mu_t^N, w_t^N \}\big) := &\dfrac{1}{2 N} \int_0^T \int_{\R^{nN}} \bigg| w_t^N(y) +\nabla \log(\mu_t^N) - \sum_{i=1}^N\Big( b \big(y^{(N), i}, \iota^N\big)\Big) \bigg|^2 \, \mu_t^N(\d y) \, \dt \\
& + \dfrac1N \int_0^T \int_{\R^{nN}} \sum_{i=1}^N \mathcal V \big(y^{(N), i}, \iota^N\big) \, \mu_t^N(\d y) \dt
\end{split}
\]
in the case of the Schr\"odinger problem and
\[
\begin{split}
\mathcal E^N_{\mu_0, \mathcal G}\big(\{ \mu_t^N, w_t^N \}\big) := & \dfrac{1}{2N} \int_0^T \int_{\R^{nN}} \bigg| w_t^N(y) +\nabla \log(\mu_t^N) - \sum_{i=1}^N \Big( b \big(y^{(N), i}, \iota^N\big)\Big)\bigg|^2 \, \mu_t^N(\d y) \, \dt \\
& + \dfrac1N \int_0^T \int_{\R^{nN}} \sum_{i=1}^N  \Big[ \mathcal V \big(y^{(N), i}, \iota^N\big) + \mathcal G \big( \iota^N \big) \Big] \mu_t^N(\d y) \, \dt  \end{split}
\]
for the case of the finite-horizon problem.
\end{prop}

\begin{proof}
    The result follows by applying the same argument as in the proof of Proposition \ref{prop:energy}.
\end{proof}

It is important to notice that if there exists a mimimizer for the energy functional associated to the $N$-approximating problem, then we can also assume that this minimizer is symmetric, namely it is invariant under any exchange of particles $X^{(N),i}$ and $X^{(N),j}$.

\begin{prop} In the same hypotheses as Proposition \ref{prop:CNEN}, if there exists a minimizer for the functional $\mathcal E^N_{\mu_0, \mathcal G}$ or $\mathcal E^N_{\mu_0, \mu_T}$, then there exists a (possibly different) minimizer of $\mathcal E^N_{\mu_0, \mathcal G}$ or $\mathcal E^N_{\mu_0, \mu_T}$, respectively, which is Markovian and symmetric with respect to permutations of variables.
\end{prop}
\begin{proof}
As usual we give the proof for Schr\"odinger problem case. The fact that the minimizer is Markovian follows from  Proposition \ref{prop:Markovian}. The only point left to prove is that we can take a minimizer that is symmetric with respect to permutations of variables.

  For any $f \colon \mathbb{R}^{nN} \to \mathbb{R}$ we define $\sigma_{ij}f(x^{(N)}) := f(\sigma_{ij}(x^{(N)}))$, where $\sigma_{ij};\mathbb{R}^{nN} \rightarrow \mathbb{R}^{n N}$ is the map switching the $i$-th and $j$-th particle. Similarly, for any vector field $V \colon \mathbb{R}^{nN}\to \mathbb{R}^{nN}$ we set 
  \[
  \begin{split}
  &\sigma_{ij}V\big(x^{(N)}\big) := \\
&\Big(V^{N,1}\big(\sigma_{ij}\big(x^{(N)}\big)\big),...,V^{N,i-1}\big(\sigma_{ij}\big(x^{(N)}\big)\big), V^{N,j}\big(\sigma_{ij}\big(x^{(N)}\big)\big),...,V^{N,j-1}\big(\sigma_{ij}\big(x^{(N)}\big)\big),V^{N,i}\big(\sigma_{ij}\big(x^{(N)}\big)\big),...\Big).
  \end{split}\]
  Moreover, if $S$ is a permutation of $N$ elements of the form $S = (i_1,j_1)(i_2,j_2)...(i_k,j_k)$, where $(i_a,j_a)$ is a $1$ cycle permutation, we define $\sigma_{S}=\sigma_{i_1j_1}\circ\cdots \circ \sigma_{i_kj_k}$. We denote by $\mathcal{S}^N$ the set of all such permutations and we observe that for any two different $S, S' \in \mathcal{S}^N$ it holds $\sigma_{S}\circ\sigma_{S'}=\sigma_{S\circ S' }$. It is also easy to see that the linear maps $\sigma_S$ are isometries of $\R^{nN}$ into itself.
     
Now let $\{ \mu_t^N, w^N_t  \}_{t \in [0, T]}$ be a minimizer of $\mathcal{E}^N_{\mu_0,\mu_T}$ and notice that  for any $S\in \mathcal{S}^N$ we have that $\Big\{\big(\sigma_S\big)_{\sharp} \mu^N_t, \sigma_{S}w^N_t\Big\}_{t \in [0, T]} \in\Lambda_{\mu_0^{\otimes},\mu_T^{\otimes}}$. In particular, also $\Big\{\big(\sigma_S\big)_{\sharp} \mu^N_t, \sigma_{S}w^N_t\Big\}_{t \in [0, T]}$ is a minimizer of the functional $\mathcal{E}^{N}_{\mu_0,\mu_T}$. In fact this can be easily seen from the very definition of $\mathcal{E}^{N}_{\mu_0,\mu_T}$ and taking into account that 
\[ \mathcal{I}\Big(\big(\sigma_S\big)_\sharp \mu^N\Big) = \mathcal{I}\big(\mu^N\big)  \text{ as well as } \int_{\R^{nN}}\big|\sigma_S w^N\big|^2\big(\sigma_S\big)_\sharp\mu^N(\d x^{(N)}) = \int_{\R^{nN}}|w^N|^2\mu^N(\d x^{(N)}).\]

From the 
finiteness of the Fisher information, we get that the measures $\{\mu^N_t\}_{t \in [0, T]}$ likewise  $\{ \big(\sigma_S\big)_\sharp \mu^N_t \}_{t \in [0, T]}$ are absolutely continuous with respect to the Lebesgue measure $ \mathscr{L}^N$. In particular we denote by  $\big(\rho^N_t\big)_S$ the density of $\big(\sigma_S\big)_\sharp \mu^N_t$ with respect to $ \mathscr{L}^N$. At this point we consider the curve of measures and the vector field defined for any $t \in [0, T]$ by setting
\[\bar{\mu}^N_t := \frac{1}{N!}\sum_{S\in \mathcal{S}^N}\big(\sigma_S\big)_\sharp\mu^N_t = \bar{\rho}^N_t  \mathscr{L}^N
\quad \text{ and }\quad \bar{w}^N_t := \frac{1}{N!\bar{\rho}_t^N}\sum_{S\in S^N} \big(\rho^N_t\big)_S \sigma_S w^N.\]
It is now immediate to see that $\big\{\bar{\mu}^N_t, \bar{w}^N_t\big\}_{t \in [0, T]} \in\Lambda_{\mu_0^{\otimes N},\mu_T^{\otimes N}}$, since each $\big\{\big(\sigma_S\big)_{\sharp}\mu^N_t, \sigma_{S}w^N_t\big\}_{t \in [0, T]}\in\Lambda_{\mu_0^{\otimes N},\mu_T^{\otimes N}}$. Furthermore we have 
\[\mathcal{E}^N_{\mu_0,\mu_T}\Big(\big\{\mu^N_t, w_t^N\}_{t \in [0, T]}\Big) = \frac{1}{N!}\sum_{S\in S^N}\mathcal{E}^N_{\mu_0,\mu_T}\Big(\big\{\big(\sigma_S\big)_\sharp\mu^N_t, \sigma_S w_t^N\}_{t \in [0, T]}\Big)
\geq \mathcal{E}^N_{\mu_0,\mu_T}\Big(\big\{\bar{\mu}^N_t, \bar{w}_t^N\}_{t \in [0, T]}\Big).\]
Indeed, by the symmetry of the sum $\sum_{i=1}^N \mathcal V \big(y^{(N), i}, \iota^N\big)$, it holds 
\[\int_{\R^{nN}}\mathcal V \big(y^{(N), i}, \iota^N\big) \, \bar{\mu}^N(\d y^{(N)}) = \int_{\R^{nN}}\mathcal V \big(y^{(N), i}, \iota^N\big) \, \mu^N(\d y^{(N)}),\]
and with the same observation we obtain the required equality also for the terms involving $b$ and the velocity field $w^N$. Furthermore, by the convexity of $\mathcal{I}$, we have
\[\mathcal{I}(\mu^N) = \frac{1}{N!}\sum_{S\in S^N} \mathcal{I}(\sigma_S\mu^N)=\mathcal{I}(\bar{\mu}^N).\]
Now, using also the convexity property of the term $\int_{\R^{nN}}|w^N|^2\mu^N(\d x^{(N)})$, we get 
\[\int_{\R^{nN}}|w^N|^2 \, \mu^N(\d x^{(N)}) = \frac{1}{N!}\sum_{S\in S^N} \int_{\R^{nN}}\big|\sigma_S w^N\big|^2 \big(\sigma_S\big)_\sharp \mu^N(\d x^{(N)}) \geq\int_{\R^{nN}}|\bar{w}^N|^2\bar{\mu}^N\big(\d x^{(N)}\big).\]
Since $\Big\{\mu^N_t, w_t^N\Big\}_{t \in [0, T]}$ is a minimizer for the energy, we have 
\[\mathcal{E}^N_{\mu_0,\mu_T}\Big(\big\{\mu_t^N, w^N_t\big\}\Big) \leq \mathcal{E}^N_{\mu_0,\mu_T}\Big(\big\{\bar{\mu}_t^N, \bar{w}^N_t\big\}\Big),\]
and thus the previous arguments guarantee that 
\[\mathcal{E}^N_{\mu_0,\mu_T}\Big(\big\{\mu_t^N, w^N_t\big\}\Big) = \mathcal{E}^N_{\mu_0,\mu_T}\Big(\big\{\bar{\mu}_t^N, \bar{w}^N_t\big\}\Big).\]
Observing that, by construction, $\Big\{\bar{\mu}^N_t, \bar{w}_t^N\Big\}_{t \in [0, T]}$ is symmetric, we conclude the proof.
\end{proof}

A direct application of the previous result implies the following result under the hypothesis of uniqueness of the minimizer.

\begin{cor}
    In the same hypothesis as Proposition \ref{prop:CNEN}, if the functionals $\mathcal E^N_{\mu_0, \mathcal G}$ and $\mathcal E^N_{\mu_0, \mu_T}$ admit a unique minimizer, then it is symmetric with respect to permutations of variables.
\end{cor}

Hereafter, if it is not stated otherwise, we consider only \emph{symmetric} minimizers $(\mu^N,w^N)$ of the functionals $\mathcal{E}^N_{\mu_0,\mu_T}$ or $\mathcal{E}^N_{\mu_0,\mathcal{G}}$. 

\section{Existence of the optimal control}

We begin this section by introducing the following compactness result for continuous curves of measures taking values in $\probp{\R^n}{p}$, $p < 2$.

\begin{lemma}\label{lem:AA}
Let $\mu_0 \in \probp{\R^n}{2}$ be a fixed measure and  $\{ \mu_t^m\}_{t \in [0, T], m \in \N} \subset \PX_2(\R^n)$ be a family of measures with $\mu^m_0 = \mu_0 \in \PX_2(\R^n)$ for any $m \in \N$. Let us denote by $\{ w^m_t \}_{m \in N}$ a family of vector fields with the property that for any $m \in \N$ it holds $\{\mu^m_t, w_t^m\}_{t \in [0,T]} \in \Lambda_{\mu_0}$.

If there exists $C > 0$ for which $\mathcal E_{\mu_0,\mathcal{G}} \Big(\{\mu^m_t, w^m_t \}\Big) < C$  for every $m \in \N$, then the family $\{ \mu_t^m\}_{m \in \N}$ is paracompact in the space $C^0([0, T], \PX_p(\R^n))$ for any $p < 2$.

Furthermore if $\mu_T \in \probp{\R^n}{2}$ is fixed with $\mu^m_T=\mu_T$ for any $m\in \mathbb{N}$ and $\mathcal{E}_{\mu_0,\mu_T}\Big(\{\mu^m_t, w^m_t \}\Big) < C$, then the same result holds.
\end{lemma}

\begin{proof}
    In order to apply Ascoli-Arzel\`a theorem (see Lemma \ref{lem:AscArz}), it is sufficient to prove that the family of measures $\{ \mu_t^m\}_{m \in \N}$ is uniformly bounded, equicontinuous and  contained in a compact subspace of $\probp{\mathbb{R}^n}{p}$, when $p<2$. A direct computation provides
    \[
    W_2(\mu^m_s, \mu^m_t) \le  \int_s^t |(\mu_r^m)'| \, \d r \le (t - s)^{1/2} \bigg( \int_s^t |(\mu_r^m)'|^2 \, \d r   \bigg)^{1/2}
    \le (t - s)^{1/2} \int_0^T \int_{\R^n} |w_t^m|^2 \, \mu_t(\d x) \, \dt.
    \]
    Now,  since both the energy functionals are bounded, namely $\mathcal E_{\mu_0,\mu_T}  \Big(\{\mu^m_t, w^m_t \}\Big) < C$ as well as $\mathcal E_{\mu_0,\mathcal{G}}  \Big(\{\mu^m_t, w^m_t \}\Big) < C$, also the term $\int_0^T \int_{\R^n} |w_t^m|^2 \, \mu_t(\d x) \, \dt$ is uniformly bounded in $m \in \N$ and this immediately provides the equicontinuity of the sequence $\{ \mu_t^m\}_{m \in \N}$. By using the same computation with $s = 0$, it follows that the sequence $\{\mu^m\}_{m\in\mathbb{N}}$ is uniformly bounded in $\probp{\mathbb{R}^n}{2}$, i.e. there exists a ball $B\subset \probp{\mathbb{R}^n}{2}$ for which the sequence $\{\mu^m\}_{m\in\mathbb{N}}$ is entirely contained in $B$. Recalling Lemma \ref{lemma:Wpcompact}, when $1\leq q < 2$ every ball in $\probp{\mathbb{R}^n}{2}$ is compact in $\probp{\mathbb{R}^n}{q}$. In particular this means that $B$ is compact in $\probp{\mathbb{R}^n}{q}$ and thus the sequence  $\{\mu^m\}_{m\in\mathbb{N}}$ takes values in a compact subset of $\probp{\mathbb{R}^n}{q}$. We can now conclude simply by applying Ascoli-Arzel\`a theorem in Lemma \ref{lem:AscArz}.
\end{proof}

We apply now the Superposition Principle in Theorem \ref{thm:SP} to the current setting.

\begin{lemma}\label{lemma:exLift}
    In the same notation and hypotheses as Lemma \ref{lem:AA}, for every $\mu^m \in AC([0, T], \PX_2(\R^n))$ and every $w^m\in C([0, T], L^2(\R^n,  \mu^m_{\cdot}))$ such that $\{\mu^m_t, w^m_t\}_{t \in [0, T]} \in \Lambda_{\mu_0}$, we denote by ${\bm \lambda}^m$ the lift of the curve $\{\mu^m_t, w^m_t\}_{t \in [0, T]}$. Then the sequence $\{\bm {\lambda}^m\}_{m \in \N}$  is tight.
\end{lemma}

\begin{proof}
     In view of Lemma \ref{lem:AA} and the fact that $\mathcal E  \Big(\{\mu^m_t, w^m_t \}_{t \in [0, T]}\Big) < C$, we have
    \[
    \begin{split}
   \sup_{m \in \N} &\int_{\R^n \times \Omega} \bigg( |X_0|^2 + \Big| \dfrac{\d}{\d t} X_t(\omega) \Big|^2  \bigg) \, {\bm \lambda}^m(\d x, \d \omega)\\ &= \sup_{m \in \N} \int_{\R^n} |x|^2 \, \mu_0(\d x) + \int_0^T \int_{\R^n} |w_t^m(x)|^2 \, \mu_t^m (\d x) \, \dt < +\infty.
   \end{split}
    \]
Recalling Remark \ref{rmk:suppLambda}, each of the measures $\{ \bm \lambda^m \}_{m \in \N}$ is concentrated in the paths $\omega$'s belonging to the space $(H^1([0, T], \R^n), || \, \cdot \, ||_{H^1})$, i.e., for any $m\in \mathbb{N}$, ${\bm \lambda}^m\big(\mathbb{R}^n \times H^1([0, T], \R^n)\big)=1$. Hence, for any $K > 0$, we define the set
    \[
    B_K = \Big\{ (x, \omega) \in \R^n \times \Omega : |x|^2 + || \omega ||^2_{H^{1}} < K  \Big\}.
    \]
   By the Sobolev embedding, the inclusion $H^1 \hookrightarrow\Omega$ is compact which in turns ensures that the set $B_K$ is compact in $\Omega$. Therefore a direct application of Chebyshev's inequality gives
    \[
    {\bm \lambda}^m(X \notin B_K) \le \dfrac{C}{K^2},
    \]
    where the constant $C$ is the one found in Lemma \ref{lem:AA}.
\end{proof}

\begin{rmk}
 Directly from the proof of the previous Lemma, we get that the set $\{\bm \lambda^m\}_{m \in \N}$  is actually  tight in $\R^n \times H^{1-s}\big([0, 1], \R^n\big)$ for every $s \in (0, 1]$.\qed
\end{rmk}

Using the previous compactness results, we can prove the existence of (at least) one optimal control $A_{\min}:[0,T] \times \mathbb{R}^n \rightarrow \mathbb{R}^n$ for both the finite horizon and the Schr\"odinger optimal control problems.

\begin{thm}\label{thm:existenceMin}
     Let $b, \mathcal V$ and $\mathcal G$ be three functionals satisfying  the hypothesis $(\mathcal{QV})$, $(\mathcal{Q}b)$  and $(\mathcal{QG})$ in Assumptions \ref{hyp}. Let us fix $\mu_0 \in \probp{\R^n}{2}$ such that $\mathcal{H}(\mu_0) <+\infty$ and, in the Schr\"odinger problem case, we also fix  $\mu_T \in \probp{\R^n}{2}$ with the property that $\mathcal{H}(\mu_T)<+\infty$. Then there exists at least one optimal control $A_{\min}^{h}$ and $A_{\min}^{S}$ for the finite horizon and the Schr\"odinger optimal control problems, respectively.
\end{thm}

\begin{proof} We give the proof in the case of the Schr\"odinger problem, since, under the hypothesis of continuity of $\mathcal{G}$ with respect to the $\probp{\mathbb{R}^n}{p}$-topology and by the lower semicontinuity of $\mathcal{H}$  with respect to the weak convergence, the proof in the finite horizon case is almost identical.

    First of all, let us observe that the functional $\mathcal E_{\mu_0,\mu_T}$ is bounded from below, hence it has a finite infimum value. Then let $\{\mu^m_t, w_t^m\}_{t \in [0, T]} \in \Lambda_{\mu_0,\mu_T}$ such that 
    \begin{equation}\label{eq:mumwm}
        \lim_{m \to \infty} \mathcal E_{\mu_0,\mu_T}\Big(\{ \mu^m_t, w^m_t \}\Big) = \inf_{\{\mu_t, w_t\} \in \Lambda_{\mu_0,\mu_T}} \mathcal E_{\mu_0,\mu_T}\Big(\{ \mu_t, w_t \}\Big). 
        \end{equation} 
        By the previous two lemmata, we can always assume (possibly passing to a subsequence) that $\mu^m_t$ and one of its lift ${\bm \lambda}^m$ (meaning a lift of $\{\mu_t^m, w_t^m\}_{t \in [0, T]}$ for some $w^m \in C([0, T], L^2(\R^n, \mu_\cdot))$ such that $\{\mu^m_t, w_t^m\}_{t \in [0, T]} \in \Lambda_{\mu_0,\mu_T}$) converge to some $\bar \mu_t \in C^0([0, T], \PX_2(\R^n))$ and $\bar{\bm \lambda} \in \R^n \times H^{1-s}([0, T], \R^n)$, respectively. In particular, this means that for any $t \in [0, T]$ we have $\bar \mu_t = (X_t)_\sharp \bar {\bm \lambda}$.
        
    Let $\bar w_t(x) = \mathbb E_{\bar{\bm \lambda}} \left. \left[ \displaystyle \dfrac{\d X_t}{\d t} \right| \, X_t=x  \right]$, defined making use of Lemma \ref{lemD}. Our goal is to prove that the couple $\{\bar \mu_t, \bar w_t\}_{t \in [0, T]}$ is a minimizer for the functional $\mathcal E_{\mu_0,\mu_T}$, which in particular admits a minimum value.\\

    As a first step, we prove that $\{\bar \mu_t, \bar w_t\}_{t \in [0, T]} \in \Lambda_{\mu_0,\mu_T}$. Since $\bar{\mu}$ is the weak limit of $\mu^m$, we have that $\bar{\mu}_0=\mu_0$ and $\bar{\mu}_T=\mu_T$. Then let $f \in C^1_c([0, T] \times \R^n)$ and observe that
    \[
    \begin{split}
\int_0^T \int_{\R^n} f(t, x)   \partial_t \bar{\mu}_t(\d x) = - \int_0^T \int_{\R^n} \partial_t f(t, x) \, \bar{\mu}_t(\d x) = - \mathbb E_{\bar{\bm \lambda}} \bigg[ \int_0^T \partial_t f(t, X_t) \, \dt \bigg] &\\   =  \mathbb E_{\bar{\bm \lambda}} \left[ \int_0^T \left\langle \nabla_x f (t, X_t),  \frac{\d X_t}{\d t} \right\rangle \, \dt \right] =  \int_0^T \int_{\R^n} \langle \nabla_x f (t, X_t), \bar{w}_t(x)  \rangle \, \bar{\mu}_t(\d x), &
\end{split}
    \]
    where we used that \[0=f(0,X_0)-f(T,X_T)=\int_0^T\frac{\d f(t,X_t)}{\d t}\, \d t=\int_0^T \left(\partial_t f(t,X_t)+\left \langle \nabla_x f(t,X_t), \frac{\d X_t}{\d t} \right\rangle \right) \d t . \]
 This means that $\{\bar{\mu}_t,\bar{w}_t\}_{t \in [0, T]}$ solves the continuity equation \eqref{eq:CE} and thus $\{\bar{\mu}_t,\bar{w}_t\}_{t \in [0, T]} \in \Lambda_{\mu_0,\mu_T}$.\\    
 
Next step consists in proving  that $\{\bar{\mu}_t,\bar{w}_t\}_{t \in [0, T]}$ is a minimizer of the functional $\mathcal{E}_{\mu_0,\mu_T}$ and, for this purpose, we show that 
\begin{equation}\label{eq:barmubarw}
\lim_{m \rightarrow +\infty} \mathcal{E}_{\mu_0,\mu_T}(\mu^m_t, w^m_t) \geq \mathcal{E}_{\mu_0,\mu_T}(\bar{\mu}_t, \bar{w}_t)
\end{equation} 
by analysing the convergence of each term in the sum defining $\mathcal{E}_{\mu_0,\mu_T}$.
First we observe that
    \[
    \begin{split}
    \lim_{m \to +\infty} \int_{\R^n} \int_0^T & |w_t^m(x)|^2 \, \dt \, \mu_t^m(\d x) = \lim_{m \to +\infty} \mathbb{E}_{{\bm \lambda}^m} \bigg[ \int_0^T \Big| \dfrac{\d X_t}{\dt} \Big|^2 \,\dt\bigg]\\ &\ge \mathbb{E}_{\bar{\bm \lambda}} \bigg[ \int_0^T \Big| \dfrac{\d X_t}{\dt} \Big|^2 \,\dt\bigg] \ge \int_{\R^n} \int_0^T |\bar{w}_t(x)|^2 \, \dt \, \bar{\mu}_t(\d x).
    \end{split}
    \]
  As for the term $\int_0^T{\mathcal{I}(\bar{\mu}_t)\d t}$, Fatou's lemma and lower semicontinuity of the Fisher information ensure that
\[\liminf_{m \rightarrow +\infty}\int_0^T{\mathcal{I}(\mu^m_t)\d t} \geq 
\int_0^T{\liminf_{m \rightarrow +\infty}\mathcal{I}(\mu^m_t)\d t}
    \geq \int_0^T{\mathcal{I}(\bar{\mu}_t)\d t}
\]    
    In order to deal with the other terms, we observe that since the $L^2$-norm of the vectors $\{w^m\}_{m \in \N}, \bar{w}$ is uniformly bounded, there exists a ball $B \subset \probp{\mathbb{R}^n}{2}$ such that $\mu^m_t,\bar{\mu}_t \in B$ for any $m\in \mathbb{N}$ and any $t \in[0,T]$. Furthermore, since $|b(x,\mu)|^2\lesssim |x|^p+1$ for some $p<2$, for any $\delta>0$ there is a compact set $K_{\delta}\subset\mathbb{R}^n$ for which
    \[\sup_{\nu,\nu'\in B} T \int_{\mathbb R^n \setminus K_{\delta}}{|b(x,\nu)|^2 \nu'(\d x) } \leq \delta.\]
    Using the continuity of $|b(x,\mu)|^2$ on $\mathbb{R}^n \times \mathcal{P}_p(\mathbb{R}^n)$, and thus its uniformly continuity on bounded subsets of $\mathbb{R}^n \times \probp{\mathbb{R}^n}{2}$, we get
     \[
     \begin{split}
     \lim_{m\rightarrow +\infty} &\left|\int_0^T\int_{\mathbb R^n}{|b(x,\mu^m_t)|^2\mu_t^m(\d x)\dt } -\int_0^T\int_{\mathbb R^n}{|b(x,\bar{\mu}_t)|^2\bar{\mu}_t(\d x)\dt } \right|\\
     \leq & \;\delta +\lim_{m\rightarrow +\infty}\left|\int_0^T\int_{K_{\delta}}{\Big(|b(x,\mu^m_t)|^2-|b(x,\bar{\mu}_t)|^2\Big)\mu_t^m(\d x)\dt } \right| \\
     & +\lim_{m\rightarrow +\infty} \left| \int_0^T\int_{K_{\delta}}{|b(x,\bar{\mu}_t)|^2\Big(\mu^m_t(\d x) -\bar{\mu}_t(\d x)\Big)\dt }\right|\\
     \leq & \; \delta + \lim_{m\rightarrow +\infty}\Big(\sup_{x\in K_{\delta}}|b(x,\mu^m_t)|^2-|b(x,\bar{\mu}_t)|^2\Big) \leq \; \delta.
     \end{split}
     \]
Being $\delta>0$ arbitrary, we obtain that 
\[\lim_{m\rightarrow +\infty}\int_0^T\int_{\mathbb R^n}\big|b(x,\mu^m_t)\big|^2\mu_t^m(\d x)\dt =\int_0^T\int_{\mathbb R^n}|b(x,\bar{\mu}_t)|^2\bar{\mu}_t(\d x)\dt.\]
In a similar way we can handle the convergence of the terms involving $\mathcal{V}$ and $\Div(b)$.

The only thing left to prove is the convergence of the term $\int_0^T\int_{\mathbb{R}^n}{\langle w^m_t(x), b(x,\mu^m_t)\rangle \mu^m_t(\d x) \d t}$. Since 
\[ 
\int_0^T\int_{\mathbb{R}^n}{\langle w^m_t(x), b(x,\mu^m_t)\rangle \mu^m_t(\d x) \d t}=\mathbb E_{\bm \lambda^m}\left[\int_0^T{\left\langle \frac{\d X_t}{\d t}, b(X_t,\mu^m_t)\right\rangle}\dt\right]
\]
and, for any $q>1$ small enough, it holds 
     \[
\left|\left\langle \frac{\d X_t}{\d t}, b(X_t,\mu^m_t)\right\rangle\right|^q \lesssim \left|\frac{\d X_t}{\d t}\right|^2+ |b(X_t,\mu^m_t)|^{2+\kappa},
     \] 
with $\kappa>0$ being a constant  such that $|b(X_t,\mu^m_t)|^{2+\kappa}\lesssim |x|^{p'}+1$ for some $p'<2$, we have that
 \[\mathbb E_{\bm \lambda^m}\left[\left|\int_0^T{\left\langle \frac{\d X_t}{\d t}, b(X_t,\mu^m_t)\right\rangle}\right|^q \, \dt \right]  \lesssim \mathbb E_{\bm \lambda^m}\left[\int_0^T \left(\left|\frac{\d X_t}{\d t}\right|^2+ \cdot|b(X_t,\mu^m_t)|^{2+\kappa}\right) \, \dt \right]< C'',
     \]
     for some $C''>0$ independent from $m$. By the tightness of $\bm \lambda ^m$, this means that for any $\delta>0$ there exists a bounded set $Y_{\delta} \subset \mathbb{R}^n \times H^1 ([0,T],\mathbb R^n)$ for which it holds
     \[\mathbb E_{\bm \lambda^m}\left[\mathbb(1- I_{Y_{\delta}}) \left|\int_0^T{\left\langle \frac{\d X_t}{\d t}, b(X_t,\mu^m_t)\right\rangle} \, \dt \right|\right]< \delta.\]
     Thus we get 
     \[
     \begin{split}
    & \lim_{m \rightarrow +\infty} \Bigg| \mathbb E_{\bm \lambda^m}  \left[ \int_0^T{\left\langle \frac{\d X_t}{\d t}, b(X_t,\mu^m_t)\right\rangle \d t}\right]- \mathbb E_{\bar{\bm \lambda}}\left[ \int_0^T{\left\langle \frac{\d X_t}{\d t}, b(X_t,\bar{\mu}_t)\right\rangle \d t}\right]\Bigg|\\
      & \qquad \leq  \; \delta + C_{Y_{\delta}} \lim_{m \rightarrow +\infty} \sup_{x \in \tilde{Y}_{\delta}, t \in [0,T]} |b(x,\bar{\mu})-b(x,\mu^m)|+\\ & \qquad +\lim_{m \rightarrow 0} \left| E_{\bm \lambda^m}\left[ \mathbb{I}_{Y_{\delta}} \int_0^T{\left\langle \frac{\d X_t}{\d t}, b(X_t,\bar{\mu}_t)\right\rangle\d t}\right]- \mathbb E_{\bar{\bm \lambda}}\left[ \mathbb{I}_{Y_{\delta}} \int_0^T{\left\langle \frac{\d X_t}{\d t}, b(X_t,\bar{\mu}_t)\right\rangle \d t}\right]\right|\\
    & \qquad \leq \; \delta\end{split}  
     \]
     where $\tilde{Y}_{\delta}\subset\mathbb{R}^n$ is a compact set such that $X_t \in \tilde{Y}_{\delta}$ for every $t \in [0, T]$ whenever the path $X_{\cdot} \in Y_{\delta}$. In particular, we have used the fact that the functional $\int_0^T{\left\langle \frac{\d X_t}{\d t}, b(X_t,\bar{\mu}_t)\right\rangle \d t}$ is continuous on the subset $Y_{\delta}\subset \mathbb R^n \times H^{1-s}([0,T],\mathbb{R}^n)$ since, on $Y_{\delta}$, $\frac{\d X_t}{\d t}$ is bounded and $b(X_t,\bar{\mu}_t)$ is uniformly continuous (being $b$ continuous and $X_t$ H\"older regular). Since $\delta>0$ is arbitrary, we get the required convergence and this proves inequality \eqref{eq:barmubarw}. 
     
     Since the sequence $\{\mu_t^m,w_t^m\}_{t \in [0, T]}$ satisfies the condition \eqref{eq:mumwm}, from inequality \eqref{eq:barmubarw} we obtain that $\{\bar{\mu}_t,\bar{w}_t\}_{t \in [0, T]}$ is a minimizer of $\mathcal{E}_{\mu_0,\mu_T}$. By taking 
     \[A^S_{\min}(x,t):=\bar{w}_t(x)+\nabla\log(\bar{\mu}_t(x))-b(x,\bar{\mu}_t),\]
     we get the thesis, since the fact that $\{\bar{\mu}_t,\bar{w}_t\}_{t \in [0, T]} \in \Lambda_{\mu_0, \mu_T}$ guarantees that the so-defined $A^S_{\min}$ actually belongs to $\mathcal A_{\mu_0, \mu_T}$.
\end{proof}

\section{Convergence of the $N$-particles approximation}

In this section we focus on the convergence of the $N$-particles approximation system to the mean field limit problem, more specifically Theorem \ref{theorem:main1}. Actually what we are going to prove is the following stronger result;

\begin{thm}
Let $\left\{\mu^N,w^N\right\}_{t \in [0, T]} \in \Lambda_{\mu_0^{\otimes N}}$ be a sequence of symmetric minimizers of $\mathcal{E}^N_{\mu_0,\mathcal{G}}$ and suppose that the conditions $(\mathcal{QV})$, $(\mathcal{Q}b)$ and $(\mathcal{QG})$ in Assumptions \ref{hyp} hold. Then there exists a subsequence $N_k \rightarrow +\infty$ and a probability measure $\gamma$ on $\Lambda_{\mu_0}$ concentrated on $\argmin \mathcal{E}_{\mu_0,\mathcal{G}}$ such that 
\begin{equation}
    \lim_{N_k \rightarrow +\infty} \mu^{N,(\ell)}= \int_{\argmin\mathcal{E}_{\mu_0,\mathcal{G}}} {\bm {\nu}}^{\otimes \ell} {\bm{\gamma}}(\d {\bm{\nu}},\d v)
    \end{equation}
and
    \begin{multline}
    \lim_{N_k \rightarrow +\infty} \int_{0}^T\int_{\R^{n N_k}}{\langle K(y^{(N_{k}|1)},\iota_{N_k}}) , w^{N_k,1}_t(y^{(N_k)}) \rangle \mu^{N_k}(\d y^{(N_k)}) \\ = \int_{\argmin\mathcal{E}_{\mu_0,\mathcal{G}}} \left( \int_{0}^T\int_{\R^{n}} \langle K(y,{\bm \nu_t}) , v_t(y))  \rangle {\bf \nu}(\d y)\right) \bm{\gamma}(\d {\bm \nu}, \d v),
    \end{multline}
    where $K$ is any $\alpha$-H\"older continuous map from $\R^n \times \probp{\R^n}{p}$ into $\R^n$, such that $|K(x,\mu)|\lesssim 1+|x|^p + \int_{\R^n}|y|^p\mu(\d y)$, for some $p\in[1,2)$.
    An analogous result holds for the Schr\"odinger problem, namely if we replace $\mathcal{E}_{\mu_0,\mathcal{G}}$ by $\mathcal{E}_{\mu_0,\mu_T}$.
\end{thm}

\begin{rmk}\label{remark:lambdaN}
It is important to observe that Theorem \ref{thm:existenceMin} can be applied to the case in which the system is the one given by the Markovian particle approximation in \ref{eq:NpartSDE}, ensuring that for every $N$ there exists at least one couple $\{\mu^N_t, w_t^N\}_{t \in [0, T]} \in \Lambda_{\mu_0^{\otimes N}}$ (or in  $\Lambda_{\mu_0^{\otimes N},\mu_T^{\otimes N}}$ in the case of the Schr\"odinger problem) minimizing the energy functional $\mathcal{E}^N_{\mu_0,\mathcal{G}}$ (or $\mathcal{E}^N_{\mu_0,\mu_T}$, respectively). Hence  Lemma \ref{lemma:exLift} provides the existence of a lift $\bm{\lambda}^N \in \prob{(\R^n \times \Omega)^N}$ associated with $\{\mu^N_t, w_t^N\}_{t \in [0, T]}$ for any $N \in \mathbb N$. \hfill\qedsymbol \bigskip
\end{rmk}

Similarly to \ref{eq:thmMinC0T} and \ref{eq:thmMinC0}, we introduce the following notation:
\[\begin{split}
\Theta^N_{\mu_0,\mathcal{G}}&:= \min_{A^N \in \mathcal A^N_{\mu_0}} \mathcal C^N_{\mu_0, \mathcal{G}} (A^{N, 1}, \dots A^{N, N}) \\ 
\Theta^N_{\mu_0, \mu_N} & := \min_{A^N \in \mathcal A^N_{\mu_0,\mu_T}} \mathcal C^N_{\mu_0, \mu_T} (A^{N, 1}, \dots A^{N, N})
\end{split}
\]

\subsection{An upper bound for the $N$-particle value function}

Goal of this section is to prove a (uniform) upper bound on the miminum of the energies associated to the $N$-particles systems.

\begin{prop}\label{prop:upperbound}
Under the hypothesis $(\mathcal{QV})$, $(\mathcal{Q}b)$ and $(\mathcal{QG})$ in Assumptions \ref{hyp} and supposing that $\mu_0 \in \probp{\R^n}{2}$ is such that $\mathcal{H}(\mu_0)<+\infty$, we have
\[
\limsup_{N \to \infty} \Theta^N_{\mu_0, \mathcal{G}} \le \Theta_{\mu_0, \mathcal{G}}.
\]
Similarly, if also $\mu_T \in \probp{\R^n}{2}$ is fixed with $\mathcal{H}(\mu_T)<+\infty$, we get
\[
\limsup_{N \to \infty} \Theta^N_{\mu_0, \mu_T} \le \Theta_{\mu_0, \mu_T}.
\]
\end{prop}
\begin{rmk}\label{remark:upperbound}
We underline that to prove Proposition \ref{prop:upperbound} it would be enough to have the continuity of $b$ with respect to both the variables, instead of the stronger assumption stated in $(\mathcal{Q}b)$, i.e. the H\"older continuity. Moreover, the bounds on the $\limsup$ proved in Proposition \ref{prop:upperbound} directly imply that $\sup_{N \in \mathbb{N}} \Theta^N_{\mu_0, \mathcal{G}}$ and $\sup_{N\in \mathbb{N}} \Theta^N_{\mu_0, \mu_T}$ are bounded. \hfill \qed \bigskip
\end{rmk}

\begin{refproofProp}\ref{prop:upperbound}
   As before, we give the proof in the case of the Schr\"odinger problem. The one for the finite horizon cost function is similar.

    Let $\{\mu_t, w_t\}_{t \in [0, T]} \in \argmin{\mathcal E_{\mu_0,\mu_T}}$. Then the couple  $\{\mu_t^{\otimes N}, w_t^{\otimes N}\}_{t \in [0, T]}$ belongs to $\Lambda_{\mu_0^{\otimes N},\mu_T^{\otimes N}}$ and thus
    \[
    \mathcal E^N_{\mu_0,\mu_T} \left( \Big\{\mu_t^{\otimes N}, w_t^{\otimes N}\Big\} \right) \ge \Theta^N_{\mu_0, \mu_T}.
    \]
    Now we observe that
    \[
    \limsup_{N \to \infty} \mathcal E^N_{\mu_0,\mu_T} \left( \Big\{\mu_t^{\otimes N}, w_t^{\otimes N}\Big\} \right) = \mathcal E_{\mu_0, \mu_T} \left(\{\mu_t, w_t\}\right) = \Theta_{\mu_0, \mu_T}.
    \]
In fact, since $\Big\{\mu_t^{\otimes N}, w_t^{\otimes N}\Big\}_{t \in [0, T]}$ consists of $N$-copies of $\{\mu_t, w_t\}_{t \in [0, T]}$,  it holds
\[
\lim_{N \to \infty}  \dfrac 1 N \int _0^T \int_{\R^{nN}} \sum_{i=1}^N \left| w_t^{\otimes N} (y^{(N), i}) \right|^2 \, \mu_t^{\otimes N}(\d y) \dt = \int _0^T \int_{\R^{n}} \left| w_t (y^{(N), i}) \right|^2 \, \mu_t(\d y) \dt
\]
while \cite[Lemma 3.2, Lemma 3.6]{KacChaos} ensures that when $t \in \{ 0, T\}$ it holds
\[
 \dfrac 1 N \mathcal I\Big(\mu_t^{\otimes N}\Big) = \mathcal I(\mu_t) \,  \text{ as well as } \, \lim_{N \to \infty} \dfrac 1 N \mathcal H(\mu_t^{\otimes N}) = \mathcal I(\mu_t).
\]
Moreover, the symmetry of the tensor product measure guarantees that
\[
\dfrac1N \int_0^T \int_{\R^{nN}} \sum_{i=1}^N \mathcal V \left(y^{(N), i}, \iota_t^N\right) \mu_t^{\otimes N}(\d y) \dt = \int_0^T \int_{\R^{nN}} \mathcal V \big(y^{(N), 1}, \iota_t^N\big) \mu_t^{\otimes N}(\d y) \dt
\]
while  the law of large numbers, that can be applied since we are considering the tensor product $\mu_t^{\otimes N}$, guarantees that $\iota_t^N \rightharpoonup \mu_t$. In particular from Corollary \ref{cor:convpromossa} we immediately get  that  $\iota_t^N \rightarrow \mu_t$ holds true in every $W_p$ with $1 \le p < 2$. Hence the continuity of $\mathcal V$ in the second variable allows to conclude that
\[
\lim_{N \to \infty} \int_0^T \int_{\R^{nN}} \mathcal V \big(y^{(N), 1}, \iota_t^N\big) \mu_t^{\otimes N}(\d y) \dt = \int_0^T \int_{\R} \mathcal V \big(y, \mu_t\big) \mu_t(\d y) \dt.
\]
Arguing in a similar way and using the continuity of $b(x , \cdot)$ and $\text{div}_{\R^n} b(x, \cdot)$ in $W_p$ for $1 \le p < 2$, we also get the following convergences
\[
\begin{split}
&\lim_{N \to \infty} \dfrac{1}{N}\int_0^T \int_{\R^{nN}} \sum_{i=1}^N \left| b(y^{(N), i}, \iota^N_t)\right|^2 \mu^{\otimes N}_t(\d y) \, \dt = \int_0^T \int_{\R^{n}} \left| b(y^{(N), 1}, \mu_t)\right|^2 \mu_t(\d y) \, \dt\\
&\lim_{N \to \infty} \dfrac{1}{N}\int_0^T \int_{\R^{nN}} \sum_{i=1}^N \text{div}_{\R^n} b(y^{(N), i}, \iota^N_t) \mu^{\otimes N}_t(\d y) \, \dt = \int_0^T \int_{\R^{n}} \text{div}_{\R^n} b(y^{(N), 1}, \mu_t) \mu_t(\d y) \, \dt\\
&\lim_{N \to \infty} \dfrac{1}{N} \int_0^T \int_{\R^{nN}} \sum_{i=1}^N \langle w_t(y^{(N), i}), b(y^{(N), i}, \iota_t^N) \rangle \mu_t^{\otimes N}(\d y) \, \dt = \int_0^T \int_{\R^{n}} \langle w_t(x), b(x, \mu_t) \rangle \mu_t(\d y) \, \dt.
\end{split}
\]
As an immediate consequence we get that
$\displaystyle\lim_{N \to \infty} \mathcal E^N_{\mu_0,\mu_T} \left( \{\mu_t^{\otimes N}, w_t^{\otimes N}\} \right) = \mathcal E_{\mu_0, \mu_T} \left(\{\mu_t, w_t\}\right)$.
\end{refproofProp}

\subsection{Convergence of the flow of measures}\label{section:deFinetti}

In this section we prove that any sequence of minimizers $\{\mu^N_t,w^N_t\}_{t \in [0, T]}$ of the energy functionals $\mathcal{E}^N_{\mu_0,\mathcal{G}}$ or $\mathcal{E}^N_{\mu_0,\mu_T}$ converges to a minimizer of the functionals $\mathcal{E}_{\mu_0,\mathcal{G}}$ or $\mathcal{E}_{\mu_0,\mu_T}$, respectively.\\

In order to obtain this result, we first prove that for any sequence of minimizers $\{\mu^N_t,w^N_t\}_{N \in \N, t \in [0, T]}$ we can find a lift ${\bm \lambda}^N$ (see Remark \ref{remark:lambdaN}) which does not change under permutation of particles. 
\begin{prop}\label{prop:Lambdasym}
Let us consider the three functionals $\mathcal V, b$ and $\mathcal G$ satisfying the assumptions $(\mathcal{QV})$, $(\mathcal{Q}b)$ and $(\mathcal{QG})$. Let $\mu_0 \in \probp{\R^n}{2}$ with $\mathcal H (\mu_0) < \infty$ and $\mu_T \in \probp{\R^n}{2}$ with $\mathcal H (\mu_T) < \infty$ in the case of the Schr\"odinger problem. For any $\{\mu^N_t,w^N_t\}_{t \in [0, T]}$  minimizer of $\mathcal{E}^N_{\mu_0,\mathcal{G}}$ (or $\mathcal{E}_{\mu_0,\mu_T}$), symmetric with respect to permutations of particles, there exists a lift $\bm{\lambda}^N \in \prob{(\R^n \times \Omega)^N}$ which is also symmetric with respect to permutations of particles.
\end{prop}
The proof is the same whether we consider the finite horizon case or the Schr\"odinger problem.
\begin{proof}
In order to prove this result we follow the construction of the lift ${\bm \lambda}$ given in \cite[Theorem 8.2.1]{AGSBook}, which is based on approximating the vector field $w_t^N$ with a sequence of vector fields of the form $w^{N, \varepsilon}_t := \frac{\rho_{\varepsilon} \ast (w^N_t
  \mu^N_t)}{\rho_{\varepsilon} \ast \mu^N_t}$. Observe that, since both the vector field $w^N_t$ and the measure $\mu^N_t$ are symmetric with respect to
  permutation of particles, also the vector $w^{N, \varepsilon}_t$ is symmetric
  with respect to exchanges of particles. The vector field $w^{N, \varepsilon}_t$ is regular and thus locally Lipschitz: a direct consequence of this property is the fact that there exists a unique curve $\{\mu_t^{N, \varepsilon}\}_{t\in [0, T]}$ for which it holds $\Big\{\mu_t^{N, \varepsilon},  w^{N,
  \varepsilon}_t \Big\} \in \Lambda_{\mu^{\otimes N}_0}$. In particular notice that the measure $\mu^{\otimes N}_0$ is symmetric. Hence  the Superposition Principle in Theorem  \ref{thm:SP} guarantees that for each $\varepsilon > 0$ there exists a lift ${\bm{\lambda}}^{N, \varepsilon} \in \prob{(\R^n \times \Omega)^N}$ of $\{\mu_t^{N, \varepsilon},  w^{N,
  \varepsilon}_t \}$. Moreover, since $w^{N, \varepsilon}_t$ is symmetric with respect to permutation of particles, also the curve
  $\{\mu^{N, \varepsilon}_t\}_{t \in [0, T]}$ satisfies the same property and  this means that $\bm{\lambda}^{N,
  \varepsilon}$ is symmetric with respect to exchanges of particles as well. 
  
  Now, recalling Remark \ref{rmk:suppLambda}, we can consider each one of the measures ${\bm \lambda}^{N, \varepsilon}$ to be defined in the space $\R^n \times \rm{AC}([0, T], \R^n)$ instead of $\Omega$, and use \cite[Theorem 8.2.1]{AGSBook} to pass to the weak limit in the sequence $\{\bm \lambda^{N, \varepsilon}\}_{\varepsilon > 0}$ as $\varepsilon \to 0$. In order to conclude, we just observe that, since
  $\bm{\lambda}^N$ is the weak limit of the sequence of symmetric  measures $\{\bm \lambda^{N, \varepsilon}\}_{\varepsilon > 0}$ and the property of being symmetric passes to the limit, also
  $\bm{\lambda}^N$ is symmetric with respect to any permutation of particles.
\end{proof}

\begin{rmk}\label{remark:lambdainfty}
We point out that from the fact that the couple $(\mu^N_t, w_t^N) \in \Lambda_{\mu_0^{\otimes N}}$ (or $\Lambda_{\mu_0^{\otimes N},\mu_T^{\otimes N}}$) minimizes the energy functional, by Remark \ref{remark:upperbound}, we obtain the following inequalities:
\begin{align}
      \int \left\| \frac{\d X_t^{(N)}(x,\omega)}{\dt} \right\|_{(L^2 ([0, T], \mathbb{R}^n))^N}^2\bm{\lambda}^N(\d (x,\omega))&=\int_0^T\int_{\mathbb{R}^{n N}} | w_t^N (x)
      |^2_{\mathbb{R}^{n N}} \mu_t^N (\d x) \dt  \leqslant N C.
      \label{eq:boundh1N}\\
    \int_0^T \mathcal{I} (\mu^N_t) \dt & \leqslant N C'  
\end{align}
    for some constants $C,C'$ independent of $N$. This in particular means that, not only $\bm{\lambda}^N \in \probp{(\R^n \times \Omega)^N}{2}$, but also $\bm{\lambda}^N \in \probp{(\R^n \times H^1([0,T],\mathbb{R}^n)^N}{2}$. \hfill \qed
\end{rmk}

\begin{thm}\label{theorem:pnconvergence}
In the same hypothesis of Proposition \ref{prop:Lambdasym}, let $\bm{\lambda}^N \in \prob{(\R^n \times \Omega)^N}$ be the lift of a minimizer $\{\mu^N_t,w^N_t\}_{t \in [0, T]}$  of $\mathcal{E}^N_{\mu_0,\mathcal{G}}$ (or $\mathcal{E}_{\mu_0,\mu_T}^N$) which is also symmetric with respect to permutations of particle. Then there exist a (sub)sequence $N \rightarrow +
    \infty$ and a symmetric measure $\bm{\lambda}^{\infty}$ on $(\mathbb{R}^n \times\Omega)^{\infty}$ such that for any $k \in \N$ it holds
    \[ \bm{\lambda}^{N, (k)} \rightarrow \bm{\lambda}^{\infty, (k)} \]
    as measures on $\probp{(\R^n \times H^{1 - s} ([0, T],
    \mathbb{R}^n))^k}{p}$, for any $s > 0$ and $1 \leqslant p < 2$.
    Furthermore we have
    \[ \int \| X^{(\infty |k)}(x,\omega) \|_{(H^1 ([0, T], \mathbb{R}^n))^k}^2
       \bm{\lambda}^{\infty} (\d (x,\omega)) < + \infty , \]
       and thus $\bm \lambda^{\infty}$ is supported on $(\R^n \times H^1([0,T],\R^n))^{\infty}$.
\end{thm}

\begin{proof} The proof of this result is a consequence of Remark \ref{remark:lambdainfty}. 
More precisely, it is based on the fact
  that
  \[\int \| X^{(N|k)} \|^2_{(H^1
  ([0, T], \mathbb{R}^n))^k} {\bm \lambda}^{N, (k)} (\d(x, \omega))
  \lesssim k C,\] 
inequality that follows directly from \eqref{eq:boundh1N}. This implies that the sequence ${\bm \lambda}^{N, (k)}$ is
  bounded in $\mathscr{P}_2((\R^n \times H^1 ([0, T], \mathbb{R}^n))^k)$ and thus it is
  tight in $\mathscr{P}_p ((\R^n \times H^{1 -s} ([0, T],
  \mathbb{R}^n))^k)$ for any $p < 2$ and $s > 0$, as already observed in $ii)$ of Remark \ref{rmk:suppLambda}. By a diagonal
  argument, we can prove the existence  of  a measure $\bm{\lambda}^{\infty}$ for which, up to a sub-sequence, ${\bm \lambda}^{N, (k)}$ converges to $\bm{\lambda}^{\infty, (k)}$ in $\mathscr{P}_p ((\R^n \times H^{1 -s} ([0, T],
  \mathbb{R}^n))^k)$ for every $k \in \N$.
  Observe that $\bm \lambda^\infty$ is symmetric since the same property is satisfied for every ${\bm \lambda}^{N, (k)}$.
  
  Finally, since the norm $\| \cdot \|_{(\R^n \times H^1 ([0, T], \mathbb{R}^n))^k}$ is lower
  semi-continuous with respect to the convergence in $(\R^n \times H^{1 - s} ([0,
  T], \mathbb{R}^n))^k$, by applying Fatou's lemma we get
  \begin{multline*}
  \int \| X^{(\infty |k)}(x,\omega) \|_{(H^1 ([0, T], \mathbb{R}^n))^k}^2
     {\bm \lambda}^{\infty, {\color{purple} (k)}} (\d(x,\omega)) \\ \leqslant \liminf_{N
     \rightarrow +\infty} \int \| X^{(N|k)}(x,\omega) \|_{(H^1 ([0, T],
     \mathbb{R}^n))^k}^2 {\bm \lambda}^{N, (k)} (\d(x,\omega)) \lesssim k C, 
     \end{multline*}
     which guarantees that $\bm{\lambda}^{\infty,(k)} \in \probp{ (\R^n \times H^{1 - s} ([0,
  T], \mathbb{R}^n))^k}{2}$ for any $k \in \N$.
\end{proof}

Sometimes it is useful to improve the weak convergence  
proved in Theorem \ref{theorem:pnconvergence} to an almost sure convergence. Hence we propose the following Corollary, which is a direct consequence of Skorokhod's representation theorem.

\begin{cor}\label{cor:convas}
    There exists a probability space $(\tilde{\Omega}, \tilde{\mathcal F}, \tilde{\mathbb P})$ and a sequence of random variables $\{X^{(N)}\}_{N \in \mathbb N \cup \{ \infty\}}$, where $X^{(N)} \colon \tilde \Omega \to (\R \times \Omega)^N$ for any $N \in \N \cup\{ \infty\}$, with the property that $\Law_{\tilde{\mathbb P}}\big(X^{(N)}\big) = \bm \lambda^N$ for each $N \in \N \cup \{ \infty \}$ and
    \[
    X^{(N | k)} \to X^{(\infty| k)} \text{ almost surely in } \, H^{1-s}\big([0, T], \probp{\R^n}{2}\big) \, \text{ for any } \, s \in (0, 1).
    \]
\end{cor}
\begin{proof} The statement follows immediately from Theorem \ref{theorem:pnconvergence} and the Skorokhod's representation theorem.
\end{proof}

If we apply de Finetti's Theorem (see Theorem \ref{theorem:deFinetti}) to $((\mathbb{R}^n\times \Omega)^{\infty},\mathcal{B}^{\infty},\bm{\lambda}^{\infty})$, we obtain the existence  of a random measure $\bm{\nu}^{\infty}\colon(\mathbb{R}^n\times \Omega)^{\infty} \rightarrow \prob{\mathbb{R}^n\times \Omega}$ for which the conclusion of Theorem \ref{theorem:deFinetti} holds true. Moreover, we can also consider the flow of (random) measures  given by
\[\bm{\nu}_t=(e_t)_\sharp \bm{\nu}^\infty, \quad t \in [0, T].\]
For the sake of notation, hereafter we denote by $\bm \nu$ the flow $\{\bm{\nu}_t\}_{t \in [0,T]}$.

\begin{rmk}
We remark that in the case in which there exists a unique element $\{\mu^N_t, w_t^N\}_{t \in [0, T]}$ in $\mathcal{A}_{\mu_0^N}$ and a unique Regular Lagrangian flow for the vector field $\{w_t\}_{t \in [0, T]}$, then the measure $\bm \lambda^\infty$ is simply given by the infinite product of $\bm \nu^\infty$. \hfill \qed
\end{rmk}

At this point we consider the random variable defined by
\begin{equation}\label{eq:wtnu}
    v_t\big(x,\{\bm{\nu}_t\}_{t \in [0,T]}\big) =  v_t(x, \bm{\nu}) := \mathbb{E}_{\bm{\lambda}^{\infty}}\left[ \left.\frac{\d X^{(\infty|1)}_t}{\dt} \right| X^{(\infty|1)}_t=x,\{\bm{\nu}_t\}_{t \in [0,T]} \right]
\end{equation} 

\begin{prop}\label{propostion:nucontiuityequation}
The flow of probability measures $\{\bm{\nu}_t\}_{t \in [0,T]}$ is almost surely in $H^1\big([0,T],\probp{\mathbb{R}^n}{2}\big)$.  Furthermore the couple $\big\{\bm{\nu}_t, v_t(x,\bm{\nu}) \big\}_{t \in [0,T]}$ satisfies the continuity equation \eqref{eq:CE}.
\end{prop}
\begin{proof}
By applying Lemma \ref{lemD} and de Finetti's theorem (Theorem \ref{theorem:deFinetti}), we get  
\begin{align}
\mathbb{E}_{{\bm \lambda}^{\infty}}\left[\int_0^T\int_{\mathbb{R}^n}|v_t(x, \bm{\nu})|^2\bm{\nu}_t(\d x)\dt\right]&=\mathbb{E}_{{\bm \lambda}^{\infty}}\left[\int_0^T\Big|v_t\Big(X_t^{(\infty|1)}, \bm{\nu}\Big)\Big|^2\dt\right] \leq\mathbb{E}_{{\bm \lambda}^{\infty}}\left[\int_0^T\left|\frac{\d X_t^{(\infty|1)}}{\dt}\right|^2\dt\right], \label{eq:vtinequality}
\end{align}
which shows that $\{\bm{\nu}_t\}_{t \in [0,T]}$ is almost surely in $H^1\big([0,T],\probp{\mathbb{R}^n}{2}\big)$. The only thing left to prove is that $\big\{\bm{\nu}_t, v_t(x,\bm{\nu}) \big\}_{t \in [0,T]}$ satisfies the continuity equation.
We start noticing that every term involved in the continuity equation is measurable with respect to $\{\bm{\nu}_t\}_{t \in [0,T]}$ and this means that, for any $f \in C^1(\R^n)$ and for any test function $g \in C_b\big(H^1([0, T], \probp{\R^n}{2})\big)$, it holds
  \[ \mathbb{E}_{{\bm \lambda}^{\infty}} \left[ g (\bm \nu) \left( \int_{\mathbb{R}^n} f (x) {\bm \nu}_t (\d  x) - \int_{\mathbb{R}^n} f (x) {\bm \nu}_s (\d x) \right) \right]
     =\mathbb{E}_{{\bm \lambda}^{\infty}} \left[g (\bm \nu) \left(f \left(X^{(\infty |1)}_t\right) - f \left(X^{(\infty
     |1)}_s\right)\right)\right].
     \]
     Now, since the function $f$ is differentiable, we can rewrite this last expression as
  \[\begin{split}\mathbb{E}_{{\bm \lambda}^{\infty}} &\left[ g (\bm \nu) \int_s^t \nabla f (X_{\tau}^{(\infty
     |1)}) \cdot \frac{\d X_\tau^{(\infty |1)}}{\d \tau} \d \tau
     \right] =\mathbb{E}_{{\bm \lambda}^{\infty}} \left[ g (\bm \nu) \int_s^t \mathbb{E}_{{\bm \lambda}^{\infty}} \left[ \nabla f
     (X_{\tau}^{(\infty |1)}) \cdot \frac{\d X_\tau^{(\infty
     |1)}}{\d \tau} \Bigg| {\bm \nu} \right] \d \tau \right] \\
    & =\mathbb{E}_{{\bm \lambda}^{\infty}} \left[ g (\bm \nu) \displaystyle \int_s^t \int_{\mathbb{R}^n} \mathbb{E}_{{\bm \lambda}^{\infty}} \left[ \left.
     \nabla f (X_{\tau}^{(\infty |1)}) \cdot \frac{\d X_\tau^{(\infty
     |1)}}{\d \tau} \right| X^{(\infty |1)} = x, {\bm \nu} \right] \bm \lambda^{(\infty| 1)}
     (\d x, {\bm\nu}) \d \tau \right] \end{split}\]
where we use the notation $\bm \lambda^{(\infty|1)} := \big(X^{(\infty| 1)}\big)_\sharp \bm \lambda^\infty$. Now a direct application of de Finetti's theorem  ensures that
 \[ \mathbb{E}_{{\bm \lambda}^{\infty}} \left[ g (\bm \nu) \left( \int_{\mathbb{R}^n} f (x) {\bm \nu}_t (\d  x) - \int_{\mathbb{R}^n} f (x) {\bm \nu}_s (\d x) \right) \right] = \mathbb{E}_{{\bm \lambda}^{\infty}} \left[ g ({\bm\nu}) \int_s^t \int_{\mathbb{R}^n} \nabla f (x) \cdot 
     v_t (x, \bm\nu) {\bm \nu}_t (\d x) \d \tau \right].
     \]
     Finally, the fact that the above expression holds for any $g \in C_b(H^1([0, T], \probp{\R^n}{2}))$ guarantees that
     \[
     \int_{\mathbb{R}^n} f (x) {\bm \nu}_t (\d  x) - \int_{\mathbb{R}^n} f (x) {\bm \nu}_s (\d x) = \int_s^t \int_{\mathbb{R}^n} \nabla f (x)
     w_t (x, \bm\nu) {\bm \nu}_t (\d x) \d \tau
     \]
     almost surely.
\end{proof}

\subsection{Convergence of value functions and optimal controls}

Aim of this section is to complete the proof of Theorem \ref{theorem:main1} by showing the convergence of the quantities $\Theta^N_{\mu_0,\mathcal{G}}$ and $\Theta^N_{\mu_0,\mu_T}$ to $\Theta_{\mu_0,\mathcal{G}}$ and $\Theta_{\mu_0, \mu_T}$, respectively. For this purpose,  it is useful to consider the probability space $(\tilde{\Omega}, \tilde{\mathcal F}, \tilde{\mathbb P})$ introduced in Corollary \ref{cor:convas}. In particular, we will denote by $\mathbb{E}_{\tilde{\mathbb{P}}}$  the expectation  taken with respect to  $(\tilde{\Omega}, \tilde{\mathcal F}, \tilde{\mathbb P})$. Moreover we set
\begin{equation}\label{eq:sigmaN}
    \Sigma_t^N := \dfrac 1 N \sum_{j=1}^N \delta_{X^{(N), j}} \,\, \text{ and } \,\, \tilde{\Sigma}_t^N := \dfrac 1 N \sum_{j=1}^N \delta_{X^{(\infty). j}},
\end{equation}
and we see them as suitable random variables defined on the probability space $(\tilde{\Omega}, \tilde{\mathcal F}, \tilde{\mathbb P})$ taking values in $C^0([0,T],\probp{\R^n}{2})$.

\begin{lemma}\label{lem:convSigma}
For any $1 \le p < 2$ and any fixed $t \in [0, T]$ it holds
\[
\lim_{N \to \infty} \mathbb E_{\tilde{\mathbb{P}}} \Big[ W_p^p \big( {\Sigma}_t^N, {\bm \nu}_t  \big)  \Big] = 0.
\]
\end{lemma}
\begin{proof}
Let us start by proving that for any fixed $t \in [0, T]$ and any $1 \le  p < 2$ we have
\[
\lim_{N \to \infty} \mathbb E_{\tilde{\mathbb{P}}} \left[W_p^p \left(  \Sigma_t^N, \tilde{\Sigma}_t^N\right) \right] = 0.
\]
For this purpose, first of all we notice that any admissible coupling between $\Sigma_t^N$ and $\tilde{\Sigma}_t^N$ is induced by a transport map which is a bijection between the points $\{X^{(N), j}\}_{j = 1}^N$ and $\{X^{(\infty), j}\}_{j = 1}^N$. Therefore
\[
W_p^p \left(  \Sigma_t^N, \tilde{\Sigma}_t^N\right) \le \dfrac{1}{N} \sum_{j=1}^N \Big| X^{(N), j} - X^{(\infty), j} \Big|^p.
\]
Observe now that, since $\Law\big(\mu_t^N\big) \ll  \mathscr{L}^n$ for any fixed $N \in \N \cup \{ \infty\}$, the points in the sets $\{X^{(N), j}\}_{j = 1}^N$ are all distinct for any $N \in \N \cup \{ \infty\}$.
Hence we can apply Corollary \ref{cor:convas} which ensures that $X^{(N), 1} \to X^{(\infty), 1}$. Moreover, recalling
\cite[Chapter 13, Section 13.7]{Williams} and using the symmetry of the particles with respect to permutations, we can conclude that
\[
W_p^p \left(  \Sigma_t^N, \tilde{\Sigma}_t^N\right) \le \mathbb E_{\tilde{\mathbb{P}}} \left[  \left| X^{(N), 1} - X^{(\infty), 1} \right|^p  \right] \rightarrow 0  \,\, \text{ for }\, p < 2.
\]
At this point de Finetti's theorem  guarantees that
\[
\tilde{\Sigma}^N_t \rightharpoonup {\bm \nu}_t \, \text{ weakly, almost surely,}
\]
while the fact that $\{\mu^N_t, w^N_t\}_{t \in [0, T]} \in \argmin{\mathcal E}^N_{\mu_0,\mu_T}$ (or ${\mathcal E}^N_{\mu_0,\mathcal{G}}$) ensures that
\[
\sup_{N \in \N} \mathbb E_{\tilde{\mathbb{P}}} \left[ W_2^2 \left( \delta_0, \tilde{\Sigma}^N_t \right)  \right] = \sup_{N \in \N} \dfrac 1 N \mathbb E_{\tilde{\mathbb{P}}} \left[ \sum_{j=1}^N \left| X^{(\infty), j}_t  \right|^2  \right] < + \infty.
\]
Finally, a direct application of Corollary \ref{cor:convpromossaE} gives the convergence
\[
\mathbb E_{\tilde{\mathbb{P}}} \left[ W_p^p \left( \tilde{\Sigma}^N_t, {\bm \nu}_t \right)  \right] \to 0 \quad \text{as } N \to \infty.
\]
We get the conclusion just by applying the triangular inequality.
\end{proof}

The upper bound on the limit of the value functions $\theta^N_{\mu_0, \mathcal G}$ (or $\theta^N_{\mu_0, \mu_T}$) has already been shown in Proposition \ref{prop:upperbound}. What remains to discuss then is the lower bound for such a limit. 

\begin{prop}\label{prop:liminf}
Under the assumptions $(\mathcal{QV})$, $(\mathcal{Q}b)$ and $(\mathcal{QG})$, supposing that $\mathcal{H}(\mu_0)<+\infty$, we have
\[
\liminf_{N \to \infty} \Theta^N_{\mu_0, \mathcal{G}} \ge \Theta_{\mu_0, \mathcal{G}}.
\]
If we also have $\mathcal{H}(\mu_T)<+\infty$ we get the limit
\[
\liminf_{N \to \infty} \Theta^N_{\mu_0, \mu_T} \ge \Theta_{\mu_0, \mu_T}.
\]
\end{prop}

Before proving Proposition \ref{prop:liminf}, we prove the following useful lemma.

\begin{lemma}\label{lemma:limitwN}
Under the hypotheses of Proposition \ref{prop:liminf}, suppose that $K:\mathbb{R}_+ \times \mathbb{R}^n \times \probp{\R ^n}{p} \to \R ^n$ is a vector valued function, H\"older continuous w.r.t. all its variables and $|K(t,x,\mu)| \lesssim |x|^p +1$. Then 
\begin{equation}\label{eq:limitwN}
\begin{split}\lim_{N \rightarrow +\infty}\int_0^T \int_{\R ^{nN}} \langle w^{(N),1}(y^{(N)})\, , \, & K(t,y^{(N),1},\iota^N \rangle \mu^N_t(\d y^{(N)}) \dt =\\
& =\mathbb{E}_{{\bm \lambda}^{\infty}}\left[\int_0^T\int_{\R ^n} \langle v_t(x,{\nnu}) , K(t,x,{\bm{\nu}}_t)\rangle {\bm \nu}_t(\d x) \dt \right]. \end{split}
\end{equation}
\end{lemma}
\begin{proof}
    We start by rewriting the left hand side of \eqref{eq:limitwN} in terms of the the random variables $X_t$ as
\begin{equation}\label{eq:mixedterm}
\int_0^T \int_{\R^{nN}} \langle w^{(N), 1}(y^{(N)}),  K(t,y^{(N), 1}, \iota^N)  \rangle \, \mu_t^N(\d y^{(N)}) \, \dt = \mathbb E_{\tilde{\mathbb{P}}} \left[ \int_0^T \dfrac{\d X^{(N), 1}}{\dt} \cdot K\left(t, X^{(N), 1}_t, \Sigma_t^N\right) \, \dt \right].
\end{equation}
The convergence of $X^{(N), 1} \to X^{(\infty), 1}$ in $H^{1-s}([0, T], \probp{\R^n}{2})$ (see Corollary \ref{cor:convas}) ensures that
\[
\dfrac{\d X^{(N), 1}}{\dt} \rightarrow  \dfrac{\d X^{(\infty), 1}}{\dt} \,\,\text{ in } \, H^{-s}([0, T], \probp{\R^n}{2}).
\]
In the following we show that $\left\{K\left(t, X^{(N), 1}_t, \Sigma^N_t\right)\right\}_{N \in \N}$ 
converges in $H^{s}$, which, by Corollary \ref{corollary:continuity}, allows us to conclude the proof. We start by observing that $K\circ (X^{(N), 1}_t (\, \cdot \,), \Sigma^N_t (\, \cdot \,)) \colon \tilde \Omega \to L^2([0, T])$ and that
 \[
 \left| K(t,\left(X_t^{(N)}, \Sigma_t^N\right) - K(r,\left(X_r^{(N)}, \Sigma_r^N\right) \right| \lesssim |t-r|^{\alpha}+ \left| X_t^{(N)} - X_r^{(N)} \right|^{\alpha} + W_p\left(\Sigma_t^N, \Sigma_r^N\right)^{\alpha},
 \]
which is guaranteed by the H\"older continuity hypothesis on $K$. Hence a direct application of Fubini's theorem ensures that
\[
\mathbb E_{\tilde{\mathbb{P}}} \left[ \left|\left| K\left(X_t^{(N)}, \Sigma_t^N\right)\right|\right|^2_{H^s((0, T))} \right] = \int_0^T \int_0^T \dfrac{\mathbb E_{\tilde{\mathbb{P}}}\left[ |K(X_t^{(N)}, \Sigma_t^N )|^2\right]}{h^{1+2s}} \, \d h \, \dt.
\]
At this point we notice that
\[
\left| X^{(N)}_{t + h} - X^{(N)}_t  \right| \le \int_t^{t+h} \left| \dfrac{\d X^{(N)}_{\, r}}{\d r}  \right| \, \d r \le h^{\frac12} \left| \left| \dfrac{\d X^{(N)}_{\, r}}{\d r}   \right|\right|_{L^2((0, T))}
\]
and that, by \cite[Chapter 8]{AGSBook}, it also holds
\[
W_p\left(\Sigma^{N}_{t + h}, \Sigma^{N}_t \right) \le \left( \int_t^{t+h} \left|\left| v^{\Sigma^N}_{\, r}  \right|\right|^p_{L^p((0, T))} \, \d r \right)^{\frac1p} \le h^{\frac{2 - p}{2 p}}\left( \int_t^{t + h}  \left|\left| v^{\Sigma^N}_{\, r} \right|\right|^p_{L^p((0, T))}  \, \d r
 \right)^{\frac12}.
\]
Moreover for any $1 < p < 2$ we have
\[
W_p\left(\Sigma^{N}_{t + h}, \Sigma^{N}_t \right) \le h^{\frac{2 - p}{2 p}}\left( \int_t^{t + h}  \left|\left| v^{\Sigma^N}_{\, r} \right|\right|^2_{L^2((0, T))}  \, \d r
 \right)^{\frac12}.
\]
This bound, together with the fact that $K(t, X^{(N), 1}_t \Sigma_t^N) \to K(t, X^{(\infty), 1}_t, {\bm \nu}_t)$ pointwise and in $L^2(\tilde{\mathbb P})$, allows us to apply
Corollary \ref{cor:convpromossa}, which provides the required convergence in $H^s([0, T], \probp{\R^n}{2})$.
\end{proof}

We are now ready to present the proof of Proposition \ref{prop:liminf}.  As usual we only consider the case of the Schr\"odinger problem. The result in the  case of the finite horizon cost functional follows in a similar way, exploiting also the convexity and lower semicontinuity of the entropy $\mathcal{H}$ with respect to the weak convergence of measures.

Morevoer, this proof is in spirit similar to the one of Theorem \ref{thm:existenceMin}, since we treat  separately the convergence of each term in the functional $\mathcal E_{\mu_0,\mu_T}$.

\begin{refproofProp}\ref{prop:liminf} First of all let us notice  that, by the symmetry of $\mu^N$ with respect to the exchange of particles, it holds
    \begin{align*}
\frac{1}{N}\int_0^T\int_{\mathbb{R}^{nN}} \sum_{i=1}^N \left|w^{(N), i}_t\left(y^{(N)}\right)\right|^2 \, \mu^N_t\big(\d y^{(N)}\big) \, \dt &=\int_0^T\int_{\mathbb{R}^{nN}}\left|w^{(N),1}_t\left(y^{(N)}\right)\right|^2\mu^N_t\big(\d y^{(N)}\big)\dt\\
&=\mathbb{E}_{\tilde{\mathbb{P}}}\left[\int_0^T\left|\frac{\d X^{(N), 1}}{\dt}\right|^2\dt\right].
    \end{align*}
Therefore, taking the liminf of the previous expression and applying Fatou's lemma, we get 
\[
\liminf_{N\rightarrow +\infty} \mathbb{E}_{\tilde{\mathbb{P}}}\left[\displaystyle\int_0^T\left|\frac{\d X^{(N), 1}}{\dt}\right|^2\dt\right] \geq \mathbb{E}_{\tilde{\mathbb{P}}}\left[\liminf_{N\rightarrow +\infty} \int_0^T\left|\frac{\d X^{(N), 1}}{\dt}\right|^2\dt\right].
\]
Using now the lower semicontinuity of the $H^1([0, T], \R^n)$-norm with respect to the convergence in $H^{1-s}([0, T], \R^n)$, we obtain
    \[
\mathbb{E}_{\tilde{\mathbb{P}}}\left[\liminf_{N\rightarrow +\infty} \int_0^T\left|\frac{\d X^{(N), 1}}{\dt}\right|^2\dt\right] \geq\mathbb{E}_{\tilde{\mathbb{P}}}\left[ \int_0^T\left|\frac{\d X^{{(\infty), 1}}}{\dt}\right|^2\dt\right] \geq \mathbb{E}_{{\bm \lambda}^{\infty}}\left[\int_0^T\int_{\mathbb{R}^n}|v_t(x,\bm{\nu})|^2\bm{\nu}_t(\d x)\dt\right],
    \]
    where in the last step we made use of inequality \eqref{eq:vtinequality} and the fact that $\Law_{\tilde{\mathbb{P}}}(X^{\infty})={\bm \lambda}^{\infty}$.
    
As for the Fisher information, \cite[Lemma 3.7-$(i)$]{KacChaos} guarantees that for any $k \in \N$
\[
\dfrac 1 N \int_0^T \mathcal I\left(\mu_t^N\right) \, \dt \ge \dfrac 1 k \int_0^T \mathcal I\left(\mu_t^{N, (k)}\right) \, \dt
\]
and, taking the liminf, we get
\[
\liminf_{N \to \infty} \dfrac 1 N \int_0^T \mathcal I\left(\mu_t^N\right) \, \dt \ge \liminf_{N \to \infty} \dfrac 1 k \int_0^T \mathcal I\left(\mu_t^{N, (k)}\right) \, \dt \ge \dfrac 1 k \int_0^T \mathcal I\left(\mu_t^{\infty, (k)}\right) \, \dt.
\]
The desire inequality for the term involving the Fisher information follows just applying
\cite[Theorem 5.7-(1)]{KacChaos} and de Finetti's theorem, which in particular ensures that $\mu_t^{\infty, (k)} = \mathbb E_{{\bm \lambda}^{\infty}}\left[{\bm \nu}_t^{\otimes k}\right]$. Hence we get
\[
\liminf_{k \to \infty} \dfrac 1 k \int_0^T \mathcal I\left(\mu_t^{\infty, (k)}\right) \, \dt = \mathbb E_{{\bm \lambda}^{\infty}} \left[ \int_0^T \mathcal I({\bm \nu}_t) \, \dt  \right].
\]
We can treat the terms involving $|b(x, \mu)|^2, \text{div}_{\R^n}(x, \mu)$ and $\mathcal V(x, \mu)$ using the same argument. In the following we present the proof for $\mathcal V$. Also in this case we start noticing that
\begin{equation}\label{eq:convV}
\dfrac1N \displaystyle\int_0^T \int_{\R^{nN}} \sum_{j=1}^N \mathcal V\left( y^{(N), j}, \iota^N\right) \, \mu^N_t(\d y^{(N)}) \, \dt = \int_0^T \mathbb E_{\tilde{\mathbb{P}}} \left[ \mathcal V\left(X_t^{(N), 1}, \Sigma^N_t\right) \right] \,\dt.
\end{equation}
At this point, from \eqref{eq:Vp}, we have
\[
\sup_{N \in \N} \mathbb E_{\tilde{\mathbb{P}}} \left[ \int_0^T \left| \mathcal V\left( X^{(N), 1}, \Sigma_t^N \right) \right|^{\frac2p} \, \dt \right] \lesssim \sup_{N \in \N} \mathbb E_{\tilde{\mathbb{P}}} \left[ \int_0^T \left| X_t^{(N),1}\right|^2 \, \dt \right] + 1
\]
where in the last inequality we have used the fact that $\{\mu_t^N, w_t^N\}_{t \in [0, T]} \in \argmin{\mathcal E_{\mu_0,\mu_T}^N}$. This implies that the sequence of random variables $\left\{\mathcal V\left( X^{(N), 1}, \Sigma_t^N \right)\right\}_{N \in \N}$, defined on the probability space $(\tilde{\Omega}, \tilde{\mathcal F}, \tilde{\mathbb P})$ (given by Corollary \ref{cor:convas}), is uniformly integrable. Moreover, since
\[
X^{(N),1} \rightarrow X^{(\infty), 1} \, \text{ almost surely and } \, \Sigma^N_t \rightarrow {\bm \nu_t} \, \text{ in } L^p,
\]
by Corollary \ref{cor:convas} and Lemma \ref{lem:convSigma}, respectively, the continuity of $\mathcal V$ in both the variables ensures that
\[\mathcal V \left( X^{(N), 1}, \Sigma_t^N\right) \rightarrow \mathcal V\left( X^{(\infty), 1}, {\bm \nu}_t \right) \, \text{ in probability}.\]
Moreover, by applying \cite[Chapter 13, Section 13.7]{Williams} we can promote this convergence to
\[
\mathcal V \left( X^{(N), 1}, \Sigma_t^N\right) \rightarrow \mathcal V\left( X^{(\infty), 1}, {\bm \nu}_t \right) \, \text{ in } L^1(\tilde {\mathbb P}).
\]
Finally, the term in \eqref{eq:convV} converges to
\[
\int_0^T \mathbb E_{\tilde{\mathbb{P}}} \left[  \mathcal V\left( X^{(\infty), 1}, {\bm \nu}_t \right) \right] \, \dt = \mathbb E_{{\bm \lambda}^\infty} \left[  \int_0^T \int_{\R^n} \mathcal V(y, {\bm \nu_t}) {\bm \nu}_t(\d y) \, \dt \right]
\]
as $N \to \infty$, where in the last equality we make  use of de Finetti's theorem and the measure ${\bm \lambda}^\infty$ is the one obtained in Theorem \ref{theorem:pnconvergence}.

We remark that the convergence of the scalar product $\langle w^{(N)}, b \rangle$ in the energy functional $\mathcal E^N_{\mu_0,\mu_T}$ follows from the symmetry of the law $\mu^{N}_t$ with respect to permutation of variables, and by Lemma \ref{lemma:limitwN}, in the case in which $K(t,x,\mu)=b(x,\mu)$.

In conclusion, since by Proposition \ref{propostion:nucontiuityequation}, $({\bm \nu},v(\cdot,{\bm \nu}))\in \Lambda_{\mu_0,\mu_T}$ almost surely and, by definition, $\Theta_{\mu_0,\mu_T}=\inf_{(\mu,w)\in\Lambda_{\mu_0,\mu_T}}\mathcal{E}(\mu,w)$, the previous discussion guarantees that 
\[\liminf_{N \rightarrow +\infty}\mathcal{E}^N_{\mu_0,\mu_T}(\mu^N,w^N)
\geq \mathbb{E}_{{\bm \lambda}^{\infty}}\left[\mathcal{E}_{\mu_0,\mu_T}({\bm \nu},v(\cdot,{\bm \nu}))\right] \geq \Theta_{\mu_0,\mu_T}\]
being $\mathcal{E}_{\mu_0,\mu_T}({\bm \nu},v(\cdot,{\bm \nu})) \geq \Theta_{\mu_0,\mu_T}$.  This allows to conclude.
\end{refproofProp}

\begin{rmk}
 As a byproduct of the proof of Proposition \ref{prop:liminf}, we obtain the following result: consider the random variables $\{{\bm \nu},v(\cdot,{\bm \nu})\}$, defined on the probability space $((\R^n \times \Omega)^{\infty},\mathcal{B}((\R^n \times \Omega)^{\infty}),{\bm \lambda}^{\infty})$ and taking values in $\Lambda_{\mu_0}$ (or in $\Lambda_{\mu_0,\mu_T}$, in the case of Schr\"odinger problem), constructed in Section \ref{section:deFinetti} as a suitable limit of the minimizer $\{\mu^N_t, w^N_t\}_{t \in [0, T]}$ (up to a subsequence). Then $({\bm \nu},v(\cdot,{\bm \nu}))$ belongs to $\argmin \mathcal{E}_{\mu_0,\mathcal{G}}$ (or $\argmin \mathcal{E}_{\mu_0,\mu_T}$ in the Schr\"odinger case) almost surely. \qed
\end{rmk}

\section{Convergence of the process on the path space}\label{section:KLdivergence}

Aim of this section is to prove that the law of the solution of the SDE \eqref{eq:NpartSDE} converges to the law of the solution of \eqref{eq:SDE}. In particular, in this part we consider the following assumptions:

\begin{assumptions}[Uniqueness of the optimal control]\label{assumption:uniqueness}
The vector field $b \colon \R^n \times \probp{\R^n}{2} \to \R^n$ satisfies $(\mathcal{Q}b)$ in Assumptions \ref{hyp}, while the functionals $\mathcal{V}\colon \R^n \times \probp{\R^n}{2} \to \R$ and $\mathcal{G} \colon \probp{\R^n}{p} \to \R$  satisfy the assumptions $(\mathcal{QV})$ and $(\mathcal{QG})$. Furthermore we assume that for each initial measure $\mu_0 \in \probp{\R^n}{2}$, $\mathcal H(\mu_0) < + \infty$ (see Remark \ref{rmk:FiniteEntropy}), there exists a unique measurable function $A^{\infty}\colon [0,T]\times \R^n \rightarrow \R^n$ with the property that
\[  A^{\infty}\in\mathcal{A}_{\mu_0} \text{ and } \mathcal{C}_{\mu_0,\mathcal{G}}(A^{\infty})=\Theta_{\mu_0,\mathcal{G}}.\]
Similarly, in the case of the Schr\"odinger problem we assume that  for any couple of measures $\mu_0,\mu_T \in \probp{\R^n}{2}$ with $\mathcal{H}(\mu_0), \mathcal{H}(\mu_T) < +\infty$ (recall Remark \ref{rmk:FiniteEntropy}) there exists a unique measurable function $A^{\infty}\colon [0,T]\times \R^n \rightarrow \R^n$ with the property that \[A^{\infty}\in\mathcal{A}_{\mu_0,\mu_T} \text{ and } 
\mathcal{C}_{\mu_0,\mu_T}(A^{\infty})=\Theta_{\mu_0,\mu_T}.\]
\end{assumptions}

\noindent In particular, in Section \ref{section:uniqueness} we will provide some conditions guaranteeing the uniqueness of the optimal control.\\

Under the previous assumptions, it is meaningful to denote by  $\mathbb{P}_{\min}$ the probability measure on $\Omega$ which is the law of the canonical process $X_t$, i.e., of the solution of the equation 
\begin{equation}\label{eq:SDEmin}
\d X_t = \left(A^{\infty}(t,X_t) + b(X_t,\Law(X_t)\right) \dt + \sqrt 2 \d W_t,
\end{equation}
where $W_t$ is a Brownian motion on $\Omega$. Moreover, we set $\mathbb{P}^{(\infty)}_{\min} := \mathbb{P}_{\min}^{\otimes \infty}$, that is the probability law on $\Omega^{\infty}$ given by the solution of an infinite number of (independent) copies of equation \eqref{eq:SDEmin}. In the same spirit,  we denote by $\mu^{\infty}_t := \left(e_t\right)_\sharp \mathbb{P}^{(\infty)}_{\min}$ the probability law on $(\R^{n})^{\infty}$ given by the fixed time marginal of $\mathbb{P}^{(\infty)}_{\min}$. Finally the vector field $w^{\infty} \colon \mathbb{R}_+ \times (\R^{n})^{\infty} \rightarrow \R^{\infty}$ is the one associated with $A^{\infty}$ as in \eqref{eq:Atow} and is such that the couple $\left\{\mu^{\infty}_t, w^{\infty}_t \right\}_{t \in [0, T]}$ solves the continuity equation with initial condition $\mu_0$ (and final condition $\mu_T$ in the case of the Schr\"odinger problem). In particular, since $\mu^{\infty}_t$ is given by infinite independent copies of $\mu^{\infty,1}_t$, this means that 
\[w_t^{\infty,k}(y^{(\infty)})=A^{\infty}(t,y^{(\infty),k})+b(y^{(\infty),k},\mu^{\infty,1}_t)-\frac{\nabla \mu^{\infty,1}(y^{(\infty),k})}{\mu^{\infty,1}(y^{(\infty),k})},\]
and thus $w^{\infty,k}$ depends only on the $k$-th variable $y^{(\infty),k}$.\\

\begin{rmk}\label{rmk:independence}
Under the Assumption \ref{assumption:uniqueness}, the measure $\bm{\nu}_t=(e_t)_{\sharp}\bm{\nu}^{\infty}$, obtained by applying de Finetti's theorem (see Section \ref{section:deFinetti}) is actually deterministic. Thus the limit measure at a fixed time $t \in [0, T]$ is given by the following expression
\[\mu^{\infty}_t=\mathbb{E}[(\bm{\nu}_t)^{\otimes \infty}]=(\bm{\nu}_t)^{\otimes \infty}.\]
From the previous equality we obtain that $\mu^{\infty,(1)}_t = \bm{\nu}_t$, and thus $\mu^{\infty}$ is obtained from the one particle probability law $\mu^{\infty,(1)}$, by taking the infinite tensor product. In other words, under the limit probability measure, the particles are independent and identically distributed. \hfill \qed
\end{rmk}

Hence let $A^{N}:\R_+ \times \R^{n N} \rightarrow \R^{n N}$ be a fixed minimizer of the functional $\mathcal{C}^N_{\mu_0,\mathcal{G}}$ (or $\mathcal{C}^N_{\mu_0,\mu_T}$,  in the case of the Schr\"odinger problem) and, in the same spirit as before, we denote by $\mathbb{P}_{\min}^N$ the probability law on $\Omega^{N}$ given by the solution of the following system of SDEs
\begin{equation}\label{eq:NSDEentropy}
\d X^{(N),j}_t=(A^{N,j}(t,X^N_t)+b(X^{(N),j}_t,\Sigma^N_t))\dt+ \sqrt 2 \d W^{(N),j}_t,\quad j=1,...,N\end{equation}
where $\Sigma^N_t$ is defined in equation \eqref{eq:sigmaN}, and $W^{(N)}$ is a Brownian motion on $\R^{nN}$. Moreover, we denote by $\mu^N_t$ the time marginal of $\mathbb{P}^N_{\min}$ and by $w^{N}:\R_+ \times \R^{nN} \rightarrow \R^{nN}$ the vector field defined by the relation 
\[w^{N,j}_t\left(y^{(N)}\right) :=A^{N,j}_t\left(y^{(N)}\right)+b\left(y^{(N),j},\iota_N\right)-\frac{\nabla^j\mu^N(y^{(N)})}{\mu^N(y^{(N)})},\]
where, as usual, $\iota_N=\frac{1}{N}\sum_{j=1}^N\delta_{y^{(N),j}}$. In the same spirit as before, we use the notation $\mathbb{P}^{(N)}_{\min}$ to indicate the probability law of the process $X^{(N)}$, solution to the SDE \eqref{eq:NSDEentropy} on the path space $\Omega^N$.

Finally $\mathbb{P}_{W,\mu_0}$ will denote the law of Brownian motion on $\Omega$ with initial probability law $\mu_0$ and we set $\mathbb{P}^{(N)}_{W,\mu_0} :=\mathbb{P}^{\otimes N}_{W,\mu_0}$\\

In the following, the Kullback–Leibler divergence between $\mathbb{P}^{(N)}_{\min},\mathbb{P}^{(\infty)}_{\min}$ and $\mathbb{P}_W$ will play a fundamental role. For this reason, the purpose of the following proposition is to  give some explicit expression of these quantities.

\begin{prop}\label{prop:DKLcomputation}
Under the hypotheses of  Assumption \ref{assumption:uniqueness}, the following identities hold true
\[D_{KL}(\mathbb{P}^{(N)}_{\min} | \mathbb{P}^{(N)}_{W,\mu_0}) = N \int_0^T \int_{\mathbb{R}^{n N}} | A^{N, 1} (y^{(N)}) + b (y^{(N), 1},   \iota_N)|^2 \mu^N_t (\d y^{(N)}) \dt  \]
and
\[D_{KL}(\mathbb{P}^{(\infty|N)}_{\min} | \mathbb{P}_{W,\mu_0}^{(N)}) = N \int_0^T \int_{\mathbb{R}^{n}} | A^{\infty} (x) + b (x,   \mu^{\infty}_t)|^2 \mu^{\infty}_t (\d x) \dt. \]
Moreover, for any $k \le N$, we have
\begin{multline}\label{eq:DKLdifference}
D_{KL}(\mathbb{P}^{(N|k)}_{\min}|\mathbb{P}^{(\infty|k)}_{\min}) \leq \\  k \left(\int_0^T \int_{\mathbb{R}^{n N}} | A^{N, 1} (y^{(N)}) + b (y^{(N), 1},
   \iota_N) - A^{\infty, 1} (y^{(N), 1}) - b (y^{(N), 1}, \mu^{\infty}_t) |^2
   \mu^N_t (\d y^{(N)}) \dt\right) 
\end{multline}
\end{prop}
\begin{proof}
    The proof is similar to the one of \cite[Theorem 7.1]{ADVRU22} and it is based on Girsanov's theorem and \cite[Lemma 5.1]{DVU14}.
\end{proof}

A direct consequence of this result and \cite[Lemma 5.1]{DVU14} is the fact that for any fixed $k \in \N$ the distances $D_{KL}(\mathbb{P}^{(N|k)}_{\min}|\mathbb{P}^{(k)}_{W,\mu_0})$ are uniformly bounded in $N \ge k$. In fact, for any $k\in \mathbb{N}$, it holds
\[D_{KL}(\mathbb{P}^{(N|k)}_{\min}|\mathbb{P}^{(k)}_{W,\mu_0}) \lesssim \frac{k}{N} D_{KL}(\mathbb{P}^{(N)}_{\min}|\mathbb{P}^{(N)}_{W,\mu_0}) = k \int_0^T \int_{\mathbb{R}^{n N}} | A^{N, 1} (y^{(N)}) + b (y^{(N), 1},   \iota_N)|^2 \mu^N_t (\d y^{(N)}) \dt. \]

\subsection{Weak convergence of the measures}

Aim of this section is proving Theorem \ref{thm:Kacchaotic}, namely the convergence of $\mathbb{P}^{(N|k)}_{\min}$ to $\mathbb{P}^{(\infty|k)}_{\min}$ in $\probp{\Omega^k}{p}$ for any $1<p<2$ and any fixed $k \in \N$.\\

In order to prove the convergence of the probability laws $\left\{\mathbb{P}^{(N|k)}_{\min}\right\}_{N \in \N}$, we first recall the uniform bound on the distances $\left\{ D_{KL}(\mathbb{P}^{(N|k)}_{\min}|\mathbb{P}^{(k)}_{W,\mu_0})\right\}_{N \in \N}$, observed after Proposition \ref{prop:DKLcomputation}. By the compactness of sublevels of Kullback–Leibler divergence (see Proposition \ref{proposition:DKLcompactness}), we get that the sequence $\left\{\mathbb{P}^{(N|k)}_{\min}\right\}_{N \in \N}$ is tight on $\Omega^k$, guaranteeing the existence  of weakly convergent subsequences. To simplify the notation, we still denote by $\left\{\mathbb{P}^{(N|k)}_{\min}\right\}_{N \in \N}$ a convergent subsequence and with $\tilde{\mathbb{P}}^k$ its limit. The first part of the present section aims to prove that $\tilde{\mathbb{P}}^k=\mathbb{P}^{(\infty|k)}_{\min}$. The proof is based on the following steps: 
\begin{enumerate}
\item  we show that $(e_t)_\sharp \tilde{\mathbb{P}}^k = \mu^{\infty, (k)}_t$ (namely that the time marginals of $\tilde{\mathbb{P}}^k$ coincide with the ones of $\mathbb{P}^{(\infty|k)}_{\min}$), 
\item we prove that the Markovian projection of the drift of $\tilde{\mathbb{P}}^k$ is equal to the drift of $\mathbb{P}^{(\infty|k)}_{\min}$.
\end{enumerate}
Finally, after showing that $D_{KL}(\tilde{\mathbb{P}}^k|\mathbb{P}_{W}^{(k)}) \leq D_{KL}(\mathbb{P}^{(\infty|k)}_{\min}|\mathbb{P}_{W}^{(k)})$, we conclude that $\tilde{\mathbb{P}}^k=\mathbb{P}^{(\infty|k)}_{\min}$.\\

For this purpose, we start by showing the convergence in $C^0([0,T],\probp{\R^{kN}}{p})$ of the curve of probability measures $\{\mu^{N}_t\}_{t \in [0, T]}$, with the property that $\{\mu^N_t,w^N_t\}_{t \in [0, T]}$ is a minimizer of $\mathcal{E}^N_{\mu_0,\mathcal{G}}$ (or $\mathcal{E}_{\mu_0,\mu_T}$),  to the curve   $\{\mu^{\infty}_t\}_{t \in [0, T]}$, which is then its McKean-Vlasov limit.

\begin{lemma}\label{lemma:fixtimeconvergence}
Under the assumptions $\mathcal{QV}$, $\mathcal{Q}b$, $\mathcal{QG}$, and Assumption \ref{assumption:uniqueness}, we have that, for any $k\in\mathbb{N}$, and $1\leq p<2$, $\mu^{N,(k)}$ converges to $\mu^{\infty,(k)}$ in $C^0([0,T],\probp{\R^{kN}}{p})$.
\end{lemma}
\begin{proof}
We start by noticing that for any $k \le N$ it holds \[\sup_{N \in \mathbb{N}} \sup_{t \in [0, T]} \|\dot{\mu}_t^{N,(k)} \|_{L^2(\mathbb P)} < + \infty, \] 
i.e. each of the metric derivatives is bounded, since $\sup_{N \in \mathbb{N}}\Theta^N_{\mu_0, \mu_N} <+\infty$. Furthermore, by Theorem \ref{theorem:pnconvergence} and the uniqueness of the minimizer for the limit problem, we have the convergence $\mu_t^{N,(k)} \rightarrow \mu_t^{\infty, (k)}$ weakly in $\prob{\mathbb{R}^n}$, for any fixed $k \in \N$. Thus a direct application of Lemma \ref{lem:AscArz} ensures the convergence of the curve of measures $\{\mu^{N,(k)}_t\}_{t \in [0, T]}$  to the curve of measures $\{\mu^{\infty,(k)}_t\}_{t \in [0, T]}$ in $C^0([0,T],\probp{\mathbb{R}^n}{p})$, i.e. 
\[ \lim_{N\rightarrow +\infty}\sup_{t\in[0,T]} {W}_p\left(\mu^{N,(k)}_t,\mu^{\infty,(k)}_t\right) = 0 \]
for any $1 \leq p<2$.
\end{proof}

We are now going to prove that $\tilde{\mathbb{P}}^k=\mathbb{P}^{(\infty|k)}_{\min}$ and to do it we first show that $(e_t)_\sharp \tilde{\mathbb{P}}^k = \mu^{\infty, (k)}_t$, indicating explicitly the connection with Lemma \ref{lemma:DKLconvergence}. 

\begin{lemma}\label{lemma:DKLconvergence}
Under Assumptions \ref{hyp} and  Assumption \ref{assumption:uniqueness}, we have 
\[\lim_{N \rightarrow \infty }\frac{1}{N} D_{KL}(\mathbb{P}^{(N)}_{\min}|\mathbb{P}^{(N)}_{W,\mu_0})= D_{KL}(\mathbb{P}^{(\infty|1)}_{\min}|\mathbb{P}_{W,\mu_0}).\]
\end{lemma}
\begin{proof} By Proposition \ref{prop:DKLcomputation} we have

\[ \frac{1}{N} D_{KL} (\mathbb{P}^{(N)}_{\min} |\mathbb{P}^{(N)}_{W,\mu_0})
   = \int_0^T \int_{\mathbb{R}^{n N}} | A^{N, 1} (y^{(N)}) - b (y^{(N), 1},
   \iota_N) |^2 \mu^N_t (\d y^{(N)}) \dt \]
\[ = \int_0^T \int_{\mathbb{R}^{n N}} | A^{N, 1} (y^{(N)}) |^2 \mu^N_t (\d
   y^{(N)}) \dt + \int_0^T \int_{\mathbb{R}^{n N}} | b (y^{(N), 1},
   \iota_N) |^2 \mu^N_t (\d y^{(N)}) \dt \]
\[ - 2 \int_0^T \int_{\mathbb{R}^{n N}} A^{N, 1} (y^{(N)}) \cdot b (y^{(N),
   1}, \iota_N) \mu^N_t (\d y^{(N)}) \dt.\]
For the first term we note that 
\[ \int_0^T \int_{\mathbb{R}^{n N}} | A^{N, 1} (y^{(N)}) |^2 \mu^N_t (\d y^{(N)}) \dt= \Theta^N_{\mu_0, \mu_T}- \int_0^T\int_\mathbb{R^{N n}} \mathcal{V}(y^{(N),1},\iota_N)\mu^{N}_t(\d y^{(N)}) \dt.\]
By Theorem \ref{theorem:main1} $\lim_{N \rightarrow +\infty}\Theta^N_{\mu_0, \mu_T}=\Theta_{\mu_0,\mu_T}$, by the proof of Proposition \ref{prop:liminf}, and by Remark \ref{rmk:independence} (namely the fact that $\bm{\nu}_t=\mu^{\infty,(1)}$), we have
\begin{multline*}
\lim_{N\rightarrow +\infty}\int_0^T\int_\mathbb{R^{N n}} \mathcal{V}(y^{(N),1},\iota_N)\mu^{N}_t(\d y^{(N)}) \dt\\
 =\mathbb E_{{\bm \lambda}^\infty} \left[  \int_0^T \int_{\R^n} \mathcal V(y, {\bm \nu_t}) {\bm \nu}_t(\d y) \, \dt \right]=\int_0^T \int_{\R^n} \mathcal V(y, { \mu^{\infty,(1)}_t}) \mu^{\infty,(1)}_t(\d y) \dt.
\end{multline*}
Since $\int_0^T \int_{\mathbb{R}^{n}} | A^{\infty} (y) |^2 \mu^{\infty,(1)}_t (\d y) \dt= \Theta^{\infty}_{\mu_0, \mu_T}- \int_0^T\int_\mathbb{R^{ n}} \mathcal{V}(y,\mu^{\infty,(1)}_t)\mu^{\infty,(1)}_t(\d y) \dt$, the previous equality implies that 
\[\lim_{N\rightarrow +\infty}\int_0^T \int_{\mathbb{R}^{n N}} | A^{N, 1} (y^{(N)}) |^2 \mu^N_t (\d y^{(N)}) \dt=\int_0^T \int_{\mathbb{R}^{n}} | A^{\infty} (y) |^2 \mu^{\infty,(1)}_t (\d y) \dt.  \]
The fact that the term $\int_0^T \int_{\mathbb{R}^{n N}} | b (y^{(N), 1},
   \iota_N) |^2 \mu^N_t (\d y^{(N)}) \dt $ converges to $\int_0^T \int_{\mathbb{R}^{n }} | b (y,\mu^{\infty,1}_t) |^2 \mu^{\infty,(1)}_t (\d y) \dt $ can be proven as in the proof of Proposition \ref{prop:liminf}.
Finally for the last term we note that
\[ A^{N} (y^{(N)}) = w^N_t - b + \nabla \log (\mu^N),\]
and thus
\begin{multline}\label{eq:convAN}
 \int_0^T \int_{\mathbb{R}^{n N}} A^{N, 1} (y^{(N)}) \cdot b (y^{(N), 1},
   \iota_N) \mu^N_t (\d y^{(N)}) = \int_0^T \int_{\mathbb{R}^{n N}} w_t^{(N),1}
   (y^{(N)}) \cdot b (y^{(N), 1}, \iota_N) \mu^N_t (\d y^{(N)})\dt + \\
 - \int_0^T \int_{\mathbb{R}^{n N}} | b (y^{(N), 1}, \iota_N) |^2 \mu^N_t
   (\d y^{(N)})\dt - \int_0^T \int_{\mathbb{R}^{n N}} \Div (b (y^{(N),
   1}, \iota_N)) \mu^N_t (\d y^{(N)})\dt.
   \end{multline}
By Lemma \ref{lemma:limitwN}, we have that 
\[ \lim_{N \rightarrow +\infty}\int_0^T \int_{\mathbb{R}^{n N}} w_t^{N,1}
   (y^{(N)}) \cdot b (y^{(N), 1}, \iota_N) \mu^N_t (\d y^{(N)})\dt=\int_0^T \int_{\mathbb{R}^{n}} w_t^{\infty,1}
   (y) \cdot b (y, \mu^{\infty,(1)}_t) \mu^{\infty,(1)}_t (\d x)\dt. \]
   The other two terms of the right hand side of expression \eqref{eq:convAN} converge to similar integrals where $\iota_N$ and $\mu^N_t$ are (both) replaced by $\mu^{\infty,1}$, following the same lines of the Proof of  Proposition \ref{prop:liminf}.\\
Thus, finally we get that,   
\[ \lim_{N \rightarrow + \infty} \frac{1}{N} D_{KL}
   (\mathbb{P}^{(N)}_{\min} |\mathbb{P}^{(N)}_{W,\mu_0}) = \]
\[ = \int_0^T \int_{\mathbb{R}^{n N}} | A^{\infty, 1} (y^{(\infty), 1}) |^2
   \mu^{\infty, (1)}_t (\d y^{(\infty), 1}) \dt + \int_0^T
   \int_{\mathbb{R}^{n N}} | b (y^{(\infty), 1}, \mu^{\infty, (1)}_t) |^2
   \mu^{\infty, (1)} (\d y^{(\infty), 1}) \dt \]
\[ - 2 \int_0^T \int_{\mathbb{R}^{n N}} A^{\infty, 1} (y^{(\infty), 1}) \cdot
   b (y^{(\infty), 1}, \mu^{\infty, (1)}_t) \mu^{\infty, (1)}_t (\d y^{(\infty), 1}) = \]
\[ = \int_0^T \int_{\mathbb{R}^{n N}} | A^{\infty, 1} (y^{(\infty), 1}) - b
   (y^{(\infty), 1}, \mu^{\infty}_t) |^2 \mu^{\infty}_t (\d y^{(\infty),
   1}) \dt = D_{KL} (\mathbb{P}^{(\infty|1)}_{\min} |\mathbb{P}_{W,\mu_0}).\]
\end{proof}

By the lower semicontinuity of Kullback–Leibler divergence and the fact that each component of $\mathbb{P}^{\infty}_{\min}$ is independent, we have that 
\begin{equation}\label{eq:inequailtyDKL}
D_{KL}(\tilde{\mathbb{P}}^k|\mathbb{P}^{(k)}_{W,\mu_0}) \leq \lim_{N \rightarrow \infty }\frac{k}{N} D_{KL}(\mathbb{P}^{(N)}_{\min}|\mathbb{P}^{(N)}_{W,\mu_0})=  k D_{KL}(\mathbb{P}^{(\infty|1)}_{\min}|\mathbb{P}_{W,\mu_0}) = D_{KL}(\mathbb{P}^{(\infty| k)}_{\min}|\mathbb{P}^{(k)}_{W,\mu_0}).
\end{equation}
Inequality \eqref{eq:inequailtyDKL} implies that $D_{KL}(\tilde{\mathbb{P}}^k|\mathbb{P}^{(k)}_{W,\mu_0})<+\infty$, and thus the existence of a measurable predictable function $\tilde{A}^{k}:\mathbb{R}_+ \times \Omega^k \to \R^{k n}$ such that
\[\d X^{k,i}_t=\left( \tilde{A}_t^{k,i}(X^k_{[0,t]})+ b(X^{k,i}_t,\mu^{\infty,1}_t) \right)\dt +\d W^{k,i}_t \]
for any $i=1,...,k$. 
By the uniqueness of the weak solution to equation , the last statement implies that $\tilde{\mathbb{P}}^k=\mathbb{P}^{(\infty|k)}_{\min}$ holds if and only if $\tilde{A}^{k,i}_t(X^{k}_{[0,t]})=A^{\infty,1}(X^{k,i}_t)$ $\mathbb{P}_W$-almost surely.

\begin{lemma}\label{lemma:tildePproperties}
We have that 
\begin{equation}\label{eq:pf}
X^k_{t,\sharp}\mathbb{\tilde{P}}^k=\mu^{\infty,(k)}_t=X^k_{t,\sharp}\mathbb{P}^{(\infty|k)}_{\min}.
\end{equation}
Furthermore for any continuous and bounded function $K:\mathbb{R}_+ \times \mathbb{R}^{k n} \rightarrow \mathbb{R}^n$ we have
\[\mathbb{E}_{\tilde{\mathbb{P}}^k}\left[\int_0^T\langle A^{k,1}(X^k_{[0,t]}) , K(t,X^k_t) \rangle \dt \right]= \int_0^T\langle A^{\infty,1}(y^{(k),1}) , K(t,y^{(k)}) \rangle \mu^{\infty,k}(\d y^{(k)}) \dt\]
\end{lemma}
\begin{proof}
The identity \eqref{eq:pf} is a direct consequence of Lemma \ref{lemma:fixtimeconvergence}.\\
As for the second result in this lemma, we note that 
\[\int_0^T\int_{\mathbb{R}^{n N}}\langle A^{N,1}(y^{(N)}) , K(t,y^{(N|k)}) \rangle \mu^{N}(\d y^{(N)}) \dt=\]
\[\mathbb{E}_{\mathbb{P}^{(N|k)}_{\min}}\left[\int_0^T{K(t,X^{(N|k)})_t\cdot dX^{(N|k)}_t }\right]-\int_0^T\int_{{\mathbb{R}^{n N}}} \langle b(y^{(N),1},\iota^N) , K(t,y^{(N|k)}) \rangle \mu^{N,k}(\d y^{(N)}) \dt.\]
By Lemma \ref{lemma:convergencestochasticintegral}, we get 
\[ \lim_{N \rightarrow +\infty} \mathbb{E}_{\mathbb{P}^{(N|k)}_{\min}}\left[\int_0^T{K(t,X^{(N|k)}_t)\cdot dX^{(N|k)}_t }\right]=\mathbb{E}_{\bar{\mathbb{P}}^{k}}\left[\int_0^T{K(t,X^{(k)}_t)\cdot dX^{(k)}_t }\right]=
\]
\[\mathbb{E}_{\bar{\mathbb{P}}^{k}}\left[\int_0^T{\langle K(t,X^{(k)}_t), \tilde{A}(t,X^{(k)}_{[0,t]})+ b(X^{(k),1}_t,\mu^{\infty,1}_t) \rangle \dt}\right].\]
On the other hand we have that
\[\lim_{N \to +\infty} \int_0^T\int_{{\mathbb{R}^{n N}}} \langle b(y^{(N),1},\iota^N) , K(t,y^{(N|k)}) \rangle \mu^{N,k}(\d y^{(N)}) \dt=\]
\[=\int_0^T\int_{{\mathbb{R}^{n}}} \langle b(y^{(k),1},\mu^{\infty,1}_t) , K(t,y^{(k)}) \rangle \mu^{\infty,k}_t(\d y^{(k)}) \dt =\] 
 \[ =\mathbb{E}_{\bar{\mathbb{P}}^{k}}\left[\int_0^T\langle K(t,X^{(k)}_t, b(X^{(k),1}_t,\mu^{\infty,1}_t) \rangle \dt \right],\]
 where the last equality holds for the first result of the present lemma. 
 Furthermore, by Lemma \ref{lemma:limitwN} and the usual lines of proof, we get 
 \[\lim_{N \to +\infty} \int_0^T\int_{\mathbb{R}^{n N}}\langle A^{N,1}(y^{(N)}) , K(t,y^{(N|k)}) \rangle \mu^{N}(\d y^{(N)}) \dt= \int_0^T\int_{\mathbb{R}^{n}}\langle A^{\infty,1}(y^{(k),1}) , K(t,y^{k}) \rangle \mu^{\infty,k}(\d y^{(k)}) \dt.\]
Putting together all the previous identities we get the thesis. 
 \end{proof}

 \begin{cor}
     Under the previous assumptions, it holds that $\bar{\mathbb{P}}^k=\mathbb{P}^{(\infty|k)}_{\min}$.
 \end{cor}

 \begin{proof}
      Directly by Lemma \ref{lemma:tildePproperties} we get that
 \[\mathbb{E}_{\bar{\mathbb{P}}^k}\left[\left.\tilde{A}(X^{(k),i}_{[0,t]})+b(X^{(k),i}_t,\mu^{\infty,1}_t)\right|X^{(k)}_t\right]=A^{\infty}_t\left(X^{(k),i}_t\right)+b\left(X^{(k),i}_t,\mu^{\infty,1}_t)\right).\]
 On the other hand, by the inequality \eqref{eq:inequailtyDKL}, we have that for any $k \in \N$ 
 \[ \sum_{i=1}^{k}\mathbb{E}_{\bar{\mathbb{P}}^k}\left[\int_0^T|\tilde{A}(X^{(k),i}_{[0,t]})+b(X^{(k),i}_t,\mu^{\infty,1}_t)|^2\dt \right]\leq   
k\int_0^T\int_{\mathbb{R}^n}|A^{\infty}_t\left(y\right)+b\left(y,\mu^{\infty,1}_t)\right)|^2\mu^{\infty,1}_t(\d y) \dt.\]
Now, using the fact that the law of $\tilde{\mathbb{P}}^k$ at fixed time is $\mu^{\infty,k}$ and Lemma \ref{lemD}, we get 
\[ \sum_{i=1}^{k}\mathbb{E}_{\bar{\mathbb{P}}^k}\left[\int_0^T|\tilde{A}(X^{(k),i}_{[0,t]})+b(X^{(k),i}_t,\mu^{\infty,1}_t)|^2\dt \right] =  
\sum_{i=1}^{k}\mathbb{E}_{\bar{\mathbb{P}}^k}\left[\int_0^T|A^{\infty}_t\left(X^{(k),i}_t\right)+b\left(X^{(k),i}_t,\mu^{\infty,1}_t)\right)|^2 \dt \right].\]
Finally, since the equality in the relation \eqref{eq:Dexple} is reached if and only if $\mathcal{D}(t,X_{[0,t]})=D(t,X_t)$, we obtain
$\tilde{A}(X^{(k),i}_{[0,t]})=A^{\infty}_t\left(X^{(k),i}_t\right)$ almost surely. \\
This means that the law $\bar{\mathbb{P}}^k$ is the probability law of the solution to $k$ independent copies of equation \eqref{eq:SDEmin}. By the uniqueness in law of the solution to equation \eqref{eq:SDEmin}, proved in Proposition \ref{prop:uniqP}, we get $\bar{\mathbb{P}}^k=\mathbb{P}^{(\infty|k)}_{\min}$.
 \end{proof}

 \begin{lemma}\label{lemma:pathspaceWp}
Under the previous hypotheses we have that 
\[\sup_{N\in\mathbb{N}} \mathbb{E}_{\mathbb{P}_{\min}^{(N)}}\left[\sup_{t\in[0,T]}\left|X_t^{(N),1}\right|^2\right] <+\infty. \]
 \end{lemma}
 \begin{proof} We just give here a sketch of the proof, see \cite[Section 7.1]{ADVRU22} for a more detailed argumentation.
By Doob's martingale inequality, applied to the Brownian motion $W^{(N),1}$, we have that
\[ \mathbb{E}_{\mathbb{P}_{\min}^{(N)}}\left[\sup_{t\in[0,T]}\left|X_t^{(N),1}\right|^2\right] \leq 2T \mathbb{E}_{\mathbb{P}_{\min}^{(N)}}\left[\int_0^T\left|A^{N,1}(X^{(N)}_t)\right|^2\dt\right]+ 2T\mathbb{E}_{\mathbb{P}_{\min}^{(N)}}\left[\int_0^T\left|b(X^{(N),j}_t,\Sigma^N_t))\right|^2\dt\right]+T.\]
The thesis follows from the fact that $\sup_{N\in \mathbb{N}} \Theta^N_{\mu_0,\mathcal{G}}$ (or $\sup_{N\in \mathbb{N}} \Theta^N_{\mu_0,\mu_T}$) is finite.
\end{proof}

 We can now conclude the proof of Theorem \ref{thm:Kacchaotic}.

\begin{refproof}{\ref{thm:Kacchaotic}} 
 Since $\bar{\mathbb{P}}^k$ is a generic limit of a weakly converging subsequence of $\mathbb{P}^{(N|k)}_{\min}$, and $k$ is arbitrary, this proves that, for any $k\in \mathbb{N}$, $\lim_{N \rightarrow +\infty}\mathbb{P}_{\min}^{(N|k)}=\mathbb{P}^{(\infty|k)}_{\min}$ on $\Omega^k$ in a weak sense. In order to conclude the proof, we use Lemma \ref{lemma:pathspaceWp}. Indeed, since the norm on $\Omega^k$ is induced by the uniform norm on $C^0([0,T],\mathbb{R}^n)$, Lemma \ref{lemma:pathspaceWp} implies that
\begin{equation}\label{eq:boundednessPN}
\sup_{N\in\mathbb{N}}W_2(\mathbb{P}_{\min}^{(N|k)},\delta_0) <+\infty.
\end{equation}
Recalling Remark \ref{Rmk:Simon}, the thesis follows from the weak convergence of $\mathbb{P}_{\min}^{(N|k)}$ to $\mathbb{P}^{(\infty|k)}_{\min}$ and the bound \ref{eq:boundednessPN}.
\end{refproof}

\subsection{Convergence with respect to Kullback–Leibler divergence}

In this section, we want to improve the type of convergence proved in 
Theorem \ref{thm:Kacchaotic}, which is the weak convergence in some Wasserstein space, to a convergence with respect to the Kullback–Leibler divergence, which in particular implies the convergence of two measures in total variation. In order to get this stronger result, we need a stronger hypothesis on the optimal control of the McKean-Vlasov system.

\begin{assumptions}\label{hyp:primo}
Assumption \ref{assumption:uniqueness} holds. Furthermore the unique minimizer $A^{\infty}$ is $C^1([0,T] \times \mathbb{R}^n,\mathbb{R}^n)$ function and there is some $1 \leq p<2$ such that $|A^{\infty}(x)|^2\lesssim |x|^p +1 $.
\end{assumptions}

We need Assumption \ref{hyp:primo} in order to get the convergence of the Kullback–Leibler divergence to zero, since the strategy of the proof of Theorem \ref{thm:KLdiv} is similar to the one of Lemma \ref{lemma:DKLconvergence} and it is based on Lemma \ref{lemma:fixtimeconvergence}, where the convergence of time marginal probability law is in $\probp{\R^{kn}}{p}$ for any $p<2$.

\begin{prop}\label{prop:limtA}
Under assumptions $\mathcal{QV}$, $\mathcal{Q}b$, $\mathcal{QG}$ and \ref{hyp:primo}, we have that \[\lim_{N \rightarrow +\infty}\int_0^T \int_{\mathbb{R}^{n N}} | A^{N, 1} (y^{(N)}) + b (y^{(N), 1},
   \iota_N) - A^{\infty, 1} (y^{(N), 1}) - b (y^{(N), 1}, \mu^{\infty,(1)}_t) |^2
   \mu^N_t (\d y^{(N)}) \dt=0.\]
\end{prop}
\begin{proof} Our plan consists in developing the square in the integral
\[
\int_0^T \int_{\mathbb{R}^{n N}} \left| A^{N, 1} (y^{(N)}) + b (y^{(N), 1},
   \iota_N) - A^{\infty, 1} (y^{(N), 1}) - b (y^{(N), 1}, \mu^{\infty,(1)}_t) \right|^2
   \mu^N_t (\d y^{(N)}) \dt
\]
and, similarly to what done in Lemma \ref{lemma:DKLconvergence}, proving that each term of this sum converges to an expression which just involves  $A^{\infty},\mu^{\infty,(1)},b(x,\mu^{\infty,(1)})$, and that these terms sum up to zero.

Following the proof of Lemma \ref{lemma:DKLconvergence}, we already know that \[\lim_{N\rightarrow +\infty}\int_0^T \int_{\mathbb{R}^{n N}} \left| A^{N, 1} (y^{(N)}) \right|^2 \mu^N_t (\d y^{(N)}) \dt = \int_0^T \int_{\mathbb{R}^{n}} \left| A^{\infty} (y) \right|^2 \mu^{\infty,(1)}_t (\d y) \dt.  \]
Then the terms
\[
\begin{split}
\int_0^T \int_{\mathbb{R}^n} | A^{\infty, 1} (y^{(N), 1}) |^2
   \mu^{N, (1)}_t (\d y^{(N), 1}) \dt, \qquad \int_0^T \int_{\mathbb{R}^{n
   N}} A^{\infty, 1} (y^{(N), 1}) \cdot b (y^{(N), 1}, \mu^{\infty,(1)}_t) \mu^N_t
   (\d y^{(N),1}) \dt 
   \end{split}
\]
  converge to the respective integrals where $\mu^{N,(1)}_t$ is replaced by $\mu^{\infty,(1)}_t$, since the sequence of measure $\mu^{N,(1)}_t$ converges to $\mu^{\infty,(1)}_t$ in $\probp{\mathbb{R}^n}{p}$. The proof of convergence of the  terms 
\[
\begin{split}
\int_0^T \int_{\mathbb{R}^{n N}} | b (y^{(N), 1}, \iota_N) |^2 \mu^N_t
   (\d y^{(N)}) \dt, \qquad  \int_0^T \int_{\mathbb{R}^{n N}} b (y^{(N), 1}, \iota_N) \cdot b 
   (y^{(N), 1}, \mu^{\infty,(1)}_t) & \mu^{N, (1)}_t (\d y^{(N), 1}) \dt \\
   \text{ and } \quad \int_0^T \int_{\mathbb{R}^{n N}} A^{\infty, 1} (y^{(N), 1}) \cdot b
   (y^{(N), 1}, \iota_N) \mu^N_t(\d y^{(N)}) \dt  &
   \end{split}
\]
is similar to the proof of convergence of the expression \eqref{eq:convV} in the proof of Proposition \ref{prop:liminf}.

While, the convergence of 
\[
\begin{split}
\int_0^T \int_{\mathbb{R}^n} | b (y^{(N), 1},
   \mu^{\infty,(1)}_t) |^2 \mu^{N, (1)}_t (\d y^{(N), 1}) \dt, \quad \int_0^T \int_{\mathbb{R}^{n N}} A^{N, 1} (y^{(N)}) \cdot & b (y^{(N),
   1}, \iota_N) \mu^N_t (\d y^{(N)}) \dt\\
   \text{and } \int_0^T \int_{\mathbb{R}^n} A^{N, 1} (y^{(N)}) \cdot b (y^{(N), 1},
   \mu^{\infty,(1)}_t) \mu^N_t (\d y^{(N)}) \dt &
   \end{split}
\]
can be proven in a similar way to the proof of the convergence of the term \eqref{eq:mixedterm} in Proposition \ref{prop:liminf}. The only convergence left to prove is then the one of the term 
\[
\int_0^T \int_{\mathbb{R}^{n N}} A^{N, 1}_t (y^{(N)}) \cdot A^{\infty,
   1}_t (y^{(N), 1}) \mu^N_t (\d y^{(N)}) \dt.
\]
We use the explicit form of the control to rewrite this term as
\[
  \int_0^T \int_{\mathbb{R}^{n N}} \left( w^{N, 1} (y^{(N)}) - b (y^{(N),
   1}, \iota_N) - \frac{\nabla \mu^N_t}{\mu_t^N} \right) \cdot A^{\infty, 1}_t
   (y^{(N), 1}) \mu^N_t (\d y^{(N)}) \dt
   \]
which, using the expectation with respect to the measure $\tilde{\mathbb{P}}$ introduced in Corollary \ref{cor:convas}, can also be rewritten as
{\small \[
\mathbb{E}_{\tilde{\mathbb{P}}} \left[ \int_0^T \frac{\d X^{(N), 1}_t}{\dt} \cdot
   A^{\infty, 1}_t (X^{(N), 1}_t) \dt \right] -\mathbb{E}_{\tilde{\mathbb{P}}} \left[ \int_0^T
   b (X^{(N), 1}_t, \Sigma^N_t) \cdot A^{\infty, 1}_t (X^{(N), 1}_t) \dt \right] +\mathbb{E}_{\tilde{\mathbb{P}}} \left[ \int_0^T \text{div} (A^{\infty, 1} (X^{(N), 1}_t))
   \dt \right].
   \]}
Hence this term converges to   
{\small \[
\mathbb{E}_{\tilde{\mathbb{P}}} \left[ \int_0^T \frac{\d X^{(\infty),
   1}_t}{\dt} \cdot A^{\infty, 1}_t (X^{(\infty), 1}_t) \dt \right]
   -\mathbb{E}_{\tilde{\mathbb{P}}} \left[ \int_0^T b (X^{(\infty), 1}_t, \mu^{\infty, (1)}_t)
   \cdot A^{\infty, 1}_t (X^{(N), 1}_t) \dt \right] +\mathbb{E}_{\tilde{\mathbb{P}}} \left[ \int_0^T \text{div} (A^{\infty, 1} (X^{(\infty), 1}_t))
   \dt \right],
\]}
which is equal to
\[
\begin{split}
\int_0^T \int_{\mathbb{R}^n} & | A^{\infty, 1}_t (x) |^2 \mu^{\infty,
   (1)}_t (\d x) \dt =  \int_0^T \int_{\mathbb{R}^n} w^{\infty}_t (x) \cdot A^{\infty, 1}_t (x)
   \mu_t^{\infty} (\d x) \dt \\
&  - \int_0^T b (x, \mu^{\infty, (1)}_t)
   \cdot A^{\infty, 1}_t (x) \mu^{\infty, (1)}_t (\d x) \dt - \int_0^T \int_{\mathbb{R}^n} \frac{\nabla \mu_t^{\infty, (1)}
   (x)}{\mu_t^{\infty, (1)} (x)} \cdot A^{\infty, 1}_t (x) \mu_t^{\infty, (1)}
   (\d x) \dt.
\end{split}
\]
Summing up all these terms, we get
\[
\lim_{N \rightarrow +\infty}\int_0^T \int_{\mathbb{R}^{n N}} | A^{N, 1} (y^{(N)}) + b (y^{(N), 1},
   \iota_N) - A^{\infty, 1} (y^{(N), 1}) - b (y^{(N), 1}, \mu^{\infty,(1)}_t) |^2
   \mu^N_t (\d y^{(N)}) \dt = 0,
   \]
which allows to conclude.
\end{proof}

\begin{refproof}{\ref{thm:KLdiv}}
The proof is a consequence of Proposition \ref{prop:DKLcomputation} (in particular of inequality \eqref{eq:DKLdifference}) and of Proposition \ref{prop:limtA}.
\end{refproof}

\newpage

\section{On the uniqueness and some generalizations}

In this section we discuss some cases where Assumption \ref{assumption:uniqueness} holds. Moreover we also propose some possible generalization of the results in Section \ref{section:results}.

\subsection{Uniqueness in the case of convex energy functional}\label{section:uniqueness}

We start our discussion by presenting a condition on the cost functional $\mathcal{C}_{\mu_0,\mu_T}$ under which we get the uniqueness of optimal control drift $A^{\infty} \in \mathcal A_{\mu_0, \mu_T}$. We present this result only in the Schr\"odinger case, since it can be easily extended, under the same assumptions, to the functional $\mathcal{C}_{\mu_0, \mathcal{G}}$ associated to the finite horizon problem. For this purpose, we consider vector fields $b$ of the following special form
\begin{equation}\label{eq:bconvex}
    b(x,\mu) = b_0 (x) + \nabla \delta_{\mu} U(x,\mu)
\end{equation} 
where $U \colon \PX_2(\R^n) \to \R$ is a convex and Fr\'echet differentiable functional such that the functional derivatives $\delta_{\mu} U$ is a $C^{2}$ function of the space variable $x$ and a $\alpha$-H\"older continuous with respect to the variable $\mu$. In this case, the functional $\mathcal{E}_{\mu_0}$, introduced in Section \ref{sez:BB}, takes the form
\[\begin{split}\mathcal{E}_{\mu_0,\mu_T}(\mu_t,w_t) =\int_0^T \int_{\R^n} \dfrac{|w_t|^2}{2} \, \mu_t(\d x)\dt + \int_0^T \mathcal I(\mu_t) \, \dt + \Big(U(\mu_T) + \mathcal{H}(\mu_T) \Big) - \Big(U(\mu_0) + \mathcal{H}(\mu_0) \Big)& \\ 
   +\int_0^T\int_{\mathbb{R}^n} \langle w_t(x), b_0(x) \rangle \mu_t(\d x) \dt + \int_0^T K(\mu_t) \, \dt ,&\end{split}\]
where
\begin{equation}  \label{eq:Kconvex}
    K(\mu_t) := \int_{\R^n} \Big( \mathcal V(x, \mu_t) + \big| \nabla \delta_\mu U(x, \mu_t)  \big|^2 + \Delta \delta_{\mu}U(x,\mu_t) \Big) \mu_t(\d x).
\end{equation}
Here we have used the fact that, being $U$ a $C^1$ functional, by \cite[Chapter III, Lemma 3.3]{BrezisBook}, it holds 
\[U(\mu_T)-U(\mu_0)=\int_0^T\langle \nabla \delta_{\mu}U(x,\mu), w_t(x) \rangle \mu(\d x) \dt. \]

\begin{thm}\label{theorem:convexity}
Let $b$ be a vector field satisfying $(\mathcal Q b)$ in Assumption \ref{hyp} of the form \eqref{eq:bconvex}, where $b_0$ and $U$ are as in equation \eqref{eq:bconvex}. If the functional $K$ defined in equation \eqref{eq:Kconvex} is strictly convex, then the control $A^{\infty}$ minimizing the functional $\mathcal{C}_{\mu_0,\mu_T}$ is unique.
\end{thm}
\begin{proof}
The proof is a suitable extension of the fact that the minimizer of a strictly convex functional is unique. We cannot use directly the standard result since the functional $\mathcal{C}_{\mu_0,\mu_T}$, and thus the energy functional $\mathcal{E}_{\mu_0,\mu_T}$, are not defined on a convex set (due to the nonlinear relation between a flow of measures $\{\mu_t\}_{t \in [0, T]}$ and  its velocity.

In order to solve this issue, we extend the convex hull operation in the following way: for any $(\mu,w)$ and $(\mu',w')$ satisfying the continuity equation and $\lambda \in [0,1]$, we define a new couple $(\mu^{\lambda},w^{\lambda})$ by setting
\[\mu^{\lambda}=\lambda \mu +(1-\lambda) \mu'\quad w^{\lambda}= \frac{\lambda w \mu}{\lambda \mu+ (1-\lambda) \mu'} + \frac{(1-\lambda) w' \mu'}{\lambda \mu+ (1-\lambda) \mu'}.\]
We prove that, under the hypotheses of the theorem, it holds
\begin{equation}\label{eq:convElambda}
\mathcal{E}_{\mu_0,\mu_T}(\mu^{\lambda},w^{\lambda})<\lambda \mathcal{E}_{\mu_0,\mu_T}(\mu,w)+(1-\lambda) \mathcal{E}_{\mu_0,\mu_T}(\mu',w'),\end{equation}
whenever $(\mu,w)\not=(\mu',w')$ and $\lambda \not= 0$. Indeed, by the convexity of the squared norm of $\R^n$, we have
\[\int_0^T\int_{\R^n} |w^{\lambda}_t(x)|^2\mu^{\lambda}_t(\d x)\dt \leq 
\lambda \int_0^T\int_{\R^n} |w_t(x)|^2\mu_t(\d x)\dt +(1-\lambda) \int_0^T\int_{\R^n} |w'_t(x)|^2\mu'_t(\d x)\dt.  \]
Furthermore, by the convexity of the Fisher information, we obtain
\[\mathcal{I}(\mu^{\lambda}_t) \leq \lambda \mathcal{I}(\mu_t) +(1-\lambda) \mathcal{I}(\mu'_t).\]
We also get that
\[\int_{\mathbb{R}^n} \langle w_t(x), b_0(x) \rangle \mu_t^{\lambda}(\d x) = \lambda \int_{\mathbb{R}^n} \langle w_t(x), b_0(x) \rangle \mu_t(\d x) + (1-\lambda) \int_{\mathbb{R}^n} \langle w_t(x), b_0(x) \rangle \mu'_t(\d x).\]
Finally, by the strict convexity of $K$, we obtain 
\[K(\mu^{\lambda}_t) < \lambda K(\mu_t) +(1-\lambda) K(\mu'_t)\]
whenever $\lambda\not= 0,1$ and $\mu\not =\mu'$. Summing all the previous terms together, we get the thesis. 
\end{proof}

We describe here some more concrete conditions under which the hypotheses of Theorem \ref{theorem:convexity} hold. Let us consider $U$ and $\mathcal{V}$ of the following form
\[\begin{split}
    U(\mu)=&\int_{\R^{2n}}{v_0(x-y) \mu(\d x) \mu(\d y)} \\
    \mathcal{V}(x,\mu)=&\int_{\R^{n}}{v_1(x-y)\mu(\d y)}-4 \left|\int_{\R^{n}}{(\nabla v_0)(x-y)\mu(\d y)}  \right|^2
\end{split}    \]
where $v_0$ is a $C^2$ function, symmetric with respect to the reflection around the origin and which is the Fourier transform of a positive measure, and also $v_1$ is a symmetric function with respect to the reflection around the origin and the Fourier transform of a positive measure of full support. Under the previous conditions, we have that $U$ is convex and 
\[K(\mu)=\int_{\R^{2n}} v_1(x-y) \mu(\d x) \mu(\d y) +2\int_{\R^{2n}} (\Delta v_0)(x-y) \mu(\d x) \mu(\d y)\]
is strictly convex. Indeed we have that 
\[\delta^2_{\mu}U(\mu)(x,y)=2 v_0(x-y) \quad \text{ and } \quad \delta^2_{\mu}K(\mu)(x,y)=2v_1(x-y)+4\Delta(v_0)(x-y)\]
where $\delta^2$ is the second (Fr\'echet) derivatives with respect to the variable $\mu$. Since $v_0$ and $\Delta v_0$ are Fourier transform of positive measures and $v_1$ is the Fourier transform of a positive measure with full support, the linear integral operators associated to the functions $\delta^2_{\mu}U(x,y)$ and $\delta^2_{\mu}K(x,y)$ are positive definite, and also  $\delta^2_{\mu}K(x,y)$ has trivial kernel. This fact implies that $U$ is a convex functional and $K$ is strictly convex.

Following the same lines, it is possible to generalize the previous example to the case of generic (non-quadratic) functionals $U$ and $K$ (see Section 2.1 of \cite{ADVRU22} for more details).

\subsection{Some simple generalizations}\label{section:generalization}

The proofs of Theorem \ref{theorem:main1}, Theorem \ref{thm:Kacchaotic} and \ref{thm:KLdiv} can be generalized with slightly different hypotheses without changing the structures and the main ideas of their proofs.\\

The first generalization, which in reality is  a simplification, is to consider the compact $n$ dimensional torus $\mathbb{T}^n$ instead of the whole space $\R^n$ as the state space of the single particle. The main difference, which simplifies the proofs, is that $\prob{\mathbb{T}^n}$ is a compact space (with respect to the weak convergence) and thus also $\probp{\mathbb{T}^n}{p}$ for any $p\geq 1$ (this is because all the Wasserstein metrics are equivalent on a compact space, see, e.g., \cite[Corollary 6.13]{Villani}). Thus $(\mathcal{QV})$ and $(\mathcal{Q}b)$ in Assumptions \ref{hyp} can be modified in the following way:
\begin{itemize}
    \item[$(\mathcal{QV})_c$] $\mathcal V \colon \mathbb{T}^n \times \prob{\mathbb{T}^n} \to \R$ such that $\mathcal V$ is continuous in both the variables (with respect to the weak convergence of measures). 
    \item[$(\mathcal{Q}b)_c$] $b \colon \mathbb{T}^n \times \prob{\mathbb{T}^n} \to \R^n$ such that
    \begin{enumerate}
    \item there is $p\in [1,2)$ such that $ b$ is $\alpha$-H\"older continuous, with $\alpha > 0$, with respect to both the variables in $\mathbb{T}^n \times \probp{\mathbb{T}^n}{p}$;
    \item for any fixed $\mu \in \PX_2(\mathbb{T}^n)$, we have $b \in D(\Div_{\mathbb{T}^n})$;
    \item  $ \Div_{\mathbb{T}^n}b(x, \cdot)$ is continuous with respect to both the variables in $\mathbb{T}^n \times \probp{\mathbb{T}^n}{p}$;
    \end{enumerate}
     \item[$(\mathcal{QG})_c$] The function $\mathcal{G}:\prob{\mathbb{T}^n} \rightarrow \mathbb{R}$ is continuous with respect to the weak convergence.
\end{itemize}

With the previous assumption we get the following result.
\begin{thm}\label{thm:compact}
Consider the problems \eqref{eq:SDE} and \eqref{eq:NpartSDE} defined on $\mathbb{T}^n$ and $\mathbb{T}^{n N}$ respectively and assume that Assumptions $(\mathcal{QV})_c$, $(\mathcal{Q}b)_c$ and $(\mathcal{QG})_c$ hold. Then the converge \eqref{eq:thmMinC0T} and \eqref{eq:thmMinC0} of the value functions is still valid, for any $\mu_0,\mu_T\in\prob{\mathbb{T}^n}$ with finite entropy.

Furthermore, if the limit problem \eqref{eq:SDE} on $\mathbb{T}^n$ admits a unique solution, we obtain the convergence of the fixed time marginals $\mu^{(N|k)}_t$ to the corresponding tensor product $\mu^{\otimes k}_t$. 

Finally, if the optimal control $A_{\min}$ of the limit problem \eqref{eq:SDE} satisfies \[A_{\min}(t, x) \in C^{0, \alpha}([0, T], C^1(\mathbb{T}^n, \R^n))\] (without the growth assumption in Assumption \ref{hyp:primo}) then the convergence of the Kullback–Leibler divergence \eqref{eq:KLmain} between the probability laws on the path space holds. 
\end{thm}

Another simple generalization consists in considering the functionals $\mathcal{V}(t,x,\mu)$ and $b(t,x,\mu)$ depending on time $t\in[0,T]$. A suitable version of Theorem \ref{theorem:main1}, Theorem \ref{thm:Kacchaotic}, Theorem \ref{thm:KLdiv} and Theorem \ref{thm:compact} hold if assumptions $(\mathcal{QV})$ and $(\mathcal{Q}b)$ (respectively $(\mathcal{QV})_c$ and $(\mathcal{Q}b)_c$) are satisfied uniformly in time $t\in[0,T]$ and  the functional $\mathcal{V}(t,x,\mu)$ is continuous in $t\in[0,T]$, while the vector field $b(t,x,\mu)$ is $\alpha$-H\"older continuous in $t\in[0,T]$. The condition of $\alpha$-H\"older with respect to the time for the functional $b$ is necessary: in fact,  in the proof of Proposition \ref{prop:liminf}, the $\beta$-H\"older continuity ,for some $\beta>0$, of the process $b(t,X_t^{N,1},\mu^N_t)$ is used in an essential way.\\

Finally we can extend Theorem \ref{theorem:main1}, Theorem \ref{thm:Kacchaotic}, and Theorem \ref{thm:KLdiv}, to the case where the functional $b$ satisfies the following assumption:
\begin{itemize}
    \item[$(\mathcal{Q}b')$] the functional $b$ has the form $b(x,\mu)=b_0(x)+b_1(x,\mu)$, where $b_0(x)=-\nabla V (x)$
for some $V\in C^2(\mathbb{R}^n) $ with the property that $b_0$ grows at most linearly at infinity and $V$ is bounded form below, it grows at most quadratically at infinity, and $\Delta V$ is bounded from above with at most quadratic growth at infinity. Notice that $b_1$ satisfies the assumption $(\mathcal{Q}b)$.
\end{itemize}
Under assumption $(\mathcal{Q}b')$, the energy functional in the Benamou-Brenier formulation of the problem takes the form 
\begin{equation}\label{eq:energyV}
    \begin{split}
\mathcal E  \Big(\{\mu_t, w_t \}_{t \in [0, T]}\Big) = & \int_0^T \int_{\R^n} \big(   |w_t(x)|^2 + |b_1(x, \mu_t)|^2\big) \mu_t(\d x) \, \dt + \int_0^T \mathcal I(\mu_t) \, \dt + \mathcal H(\mu_T) - \mathcal H(\mu_0)\\ & - 2 \int_0^T \int_{\R^n} \Big(\langle w_t(x), b_1(x, \mu_t) \rangle + \Div_{\R^n} b_1(x, \mu_t)  +\frac{1}{2} \mathcal V(x, \mu_t) \Big)\, \mu_t(\d x) \, \dt\\
&+2\int_{\mathbb{R}^n}V(x)\mu_T(\d x)-2\int_{\mathbb{R}^n}V(x)\mu_0(\d x)-2\int_0^T\int_{\mathbb{R}^n}{\nabla V(x) \cdot b(x,\mu_t)\mu_t(\d x)\dt}\\
&-\int_0^T\int_{\mathbb{R}^n}\Delta V(x) \mu_t(\d x) \dt+\int_0^T\int_{\mathbb{R}^n}|\nabla V|^2(x)\mu_t(\d x)\dt.
\end{split}
\end{equation}
A simple example of this situation is when $V = C |x|^2$, for some $C \in \R$. In this case, the equation \eqref{eq:SDE} is a perturbation of the Ornstein-Uhlenbeck process (having the Gaussian as the invariant measure). 

\begin{thm}\label{thm:compact2}
Consider the problems \eqref{eq:SDE} and \eqref{eq:NpartSDE} and assume that Assumptions $(\mathcal{QV})$ and $(\mathcal{Q}b')$ hold. Then, the converges \eqref{eq:thmMinC0T} and \eqref{eq:thmMinC0} of the value functions are still valid for any $\mu_0,\mu_T\in\probp{\R}{2}$ with finite entropy.

Furthermore, if the limit problem \eqref{eq:SDE} on $\mathbb{R}^n$ admits a unique solution, we obtain the convergence of the fixed time marginals $\mu^{(N|k)}_t$ to the corresponding tensor product $\mu^{\otimes k}_t$. 

Finally, if the optimal control $A_{\min}$ of the limit problem \eqref{eq:SDE} satisfies Assumption \ref{hyp:primo}), then the convergence of the Kullback–Leibler divergence \eqref{eq:KLmain} between the probability laws on the path space holds. 
\end{thm}
\begin{proof}
The proof is very similar to the one of Theorem \ref{theorem:main1}, Theorem \ref{thm:Kacchaotic}, and Theorem \ref{thm:KLdiv} with two main differences. The first consists in the proof of Proposition \eqref{prop:liminf}. In fact, in the present case we note that the term 
\[\int_0^T\int_{\mathbb{R}^n}{\nabla V(x^{N,1}) \cdot b(x^{N,1},\mu^N_t)\mu^N_t(\d x^N)\dt}
\]
converges to $\mathbb E_{{\bm \lambda}^\infty} \left[  \int_0^T\int_{\mathbb{R}^n}{\nabla V(x) \cdot b(x,{\bm \nu}_t){\bm \nu}_t(\d x)\dt} \right]$, since the product $|\nabla V| |b(x,\mu)|$ grows as a power strictly less then $2$ at infinity.  For the remaining terms of expression \eqref{eq:energyV}, we note that 
 by Fatou's lemma the integrals 
$\int_{\mathbb{R}^n}V(x)\mu_T(\d x)$, $-\int_0^T\int_{\mathbb{R}^n}\Delta V(x) \mu_t(\d x) \dt$ and $\int_0^T\int_{\mathbb{R}^n}|\nabla V|^2(x)\mu_t(\d x)\dt$ are lower-semicontinuous with respect to the weak convergence of the measure $\mu$. This implies that 
\[ 
\begin{split}
\liminf_{N\rightarrow +\infty}\left( 2\int_{\mathbb{R}^n}V(x)\mu_T^{(N, 1)}(\d x)-\int_0^T\int_{\mathbb{R}^n}\Delta V(x) \mu_t^{(N, 1)}(\d x) \dt+\int_0^T\int_{\mathbb{R}^n}|\nabla V|^2(x)\mu^{(N, 1)}_t(\d x)\dt\right)&\\
\geq \mathbb E_{{\bm \lambda}^\infty} \left[2\int_{\mathbb{R}^n}V(x){\bm \nu}_t(\d x)-\int_0^T\int_{\mathbb{R}^n}\Delta V(x) {\bm \nu}_t(\d x) \dt+\int_0^T\int_{\mathbb{R}^n}|\nabla V|^2(x){\bm \nu}_t(\d x)\dt \right],&
\end{split}
\]
and thus we get the same result as in Proposition \ref{prop:liminf}.

Finally we notice that under Assumption $(\mathcal{Q}b')$,  it holds
\begin{multline*}| A^{N, 1} (y^{(N)}) + b (y^{(N), 1},
   \iota_N) - A^{\infty, 1} (y^{(N), 1}) - b (y^{(N), 1}, \mu^{\infty}_t) |^2=\\
   | A^{N, 1} (y^{(N)}) + b_1 (y^{(N), 1},
   \iota_N) - A^{\infty, 1} (y^{(N), 1}) - b_1 (y^{(N), 1}, \mu^{\infty}_t) |^2\end{multline*}
and so we are able to prove Proposition \ref{prop:limtA}, arguing as in Section \ref{section:KLdivergence}.
\end{proof}

\bibliographystyle{plain}

\begin{thebibliography}{99}
	
	\bibitem{ADVRU20}
	S.~Albeverio, F.~C. De~Vecchi, A.~Romano, and S.~Ugolini. Strong Kac's chaos in the mean-field Bose-Einstein
	condensation. {\em Stoch. Dyn.}, 20(5):2050031, 21, 2020.
	
	\bibitem{ADVRU22}
	S.~Albeverio, F.~C. De~Vecchi, A.~Romano, and S.~Ugolini.
	\newblock Mean-field limit for a class of stochastic ergodic control problems.
	\newblock {\em SIAM J. Control Optim.}, 60(1):479--504, 2022.
	
	\bibitem{ADVU17}
	S.~Albeverio, F.~C. De~Vecchi, and S.~Ugolini.
	\newblock Entropy chaos and {B}ose-{E}instein condensation.
	\newblock {\em J. Stat. Phys.}, 168(3):483--507, 2017.
	
	\bibitem{AGSBook}
	L.~Ambrosio, N.~Gigli, and G.~Savar\'{e}.
	\newblock {\em Gradient flows in metric spaces and in the space of probability
		measures}.
	\newblock Lectures in Mathematics ETH Z\"{u}rich. Birkh\"{a}user Verlag, Basel,
	second edition, 2008.
	
	\bibitem{AmbrosioTrevisan}
	L.~Ambrosio and D.~Trevisan.
	\newblock Well-posedness of {L}agrangian flows and continuity equations in
	metric measure spaces.
	\newblock {\em Anal. PDE}, 7(5):1179--1234, 2014.
	
	\bibitem{ConfortiLeonard}
	J.~Backhoff, G.~Conforti, I.~Gentil, and C.~L\'{e}onard.
	\newblock The mean field {S}chr\"{o}dinger problem: ergodic behavior, entropy
	estimates and functional inequalities.
	\newblock {\em Probab. Theory Related Fields}, 178(1-2):475--530, 2020.
	
	\bibitem{BookBesov}
	H.~Bahouri, J.-Y. Chemin, and R.~Danchin.
	\newblock {\em Fourier analysis and nonlinear partial differential equations},
	volume 343 of {\em Grundlehren der mathematischen Wissenschaften [Fundamental
		Principles of Mathematical Sciences]}.
	\newblock Springer, Heidelberg, 2011.
	
	\bibitem{BarbuRo}
	V.~Barbu and M.~R\"{o}ckner.
	\newblock From nonlinear {F}okker-{P}lanck equations to solutions of
	distribution dependent {SDE}.
	\newblock {\em Ann. Probab.}, 48(4):1902--1920, 2020.
	
	\bibitem{BayraktarCossoPham2018}
	E.~Bayraktar, A.~Cosso, and H.~Pham.
	\newblock Randomized dynamic programming principle and {F}eynman-{K}ac
	representation for optimal control of {M}c{K}ean-{V}lasov dynamics.
	\newblock {\em Trans. Amer. Math. Soc.}, 370(3):2115--2160, 2018.
	
	\bibitem{BB}
	J.-D. Benamou and Y.~Brenier.
	\newblock A computational fluid mechanics solution to the {M}onge-{K}antorovich
	mass transfer problem.
	\newblock {\em Numer. Math.}, 84(3):375--393, 2000.
	
	\bibitem{BogachevRockner2015}
	V.~I. Bogachev, N.~V. Krylov, M.~R\"{o}ckner, and S.~V. Shaposhnikov.
	\newblock {\em Fokker-{P}lanck-{K}olmogorov equations}, volume 207 of {\em
		Mathematical Surveys and Monographs}.
	\newblock American Mathematical Society, Providence, RI, 2015.
	
	\bibitem{RossiFrancesco2017}
	M.~Bongini, M.~Fornasier, F.~Rossi, and F.~Solombrino.
	\newblock Mean-field {P}ontryagin maximum principle.
	\newblock {\em J. Optim. Theory Appl.}, 175(1):1--38, 2017.
	
	\bibitem{RossiFrancesco2019}
	B.~Bonnet and F.~Rossi.
	\newblock The {P}ontryagin maximum principle in the {W}asserstein space.
	\newblock {\em Calc. Var. Partial Differential Equations}, 58(1):Paper No. 11,
	36, 2019.
	
	\bibitem{RossiFrancesco2021}
	B.~Bonnet and F.~Rossi.
	\newblock Intrinsic {L}ipschitz regularity of mean-field optimal controls.
	\newblock {\em SIAM J. Control Optim.}, 59(3):2011--2046, 2021.
	
	\bibitem{BrezisBook}
	H.~Br\'{e}zis.
	\newblock {\em Op\'{e}rateurs maximaux monotones et semi-groupes de
		contractions dans les espaces de {H}ilbert}.
	\newblock North-Holland Mathematics Studies, No. 5. North-Holland Publishing
	Co., Amsterdam-London; American Elsevier Publishing Co., Inc., New York,
	1973.
	\newblock Notas de Matem\'{a}tica [Mathematical Notes], No. 50.
	
	\bibitem{LionsMastereq}
	P.~Cardaliaguet, F.~Delarue, J.-M. Lasry, and P.-L. Lions.
	\newblock {\em The master equation and the convergence problem in mean field
		games}, volume 201 of {\em Annals of Mathematics Studies}.
	\newblock Princeton University Press, Princeton, NJ, 2019.
	
	\bibitem{cardaliaguet2023sharp}
	P.~Cardaliaguet, J.~Jackson, N.~Mimikos-Stamatopoulos, and P.~E. Souganidis.
	\newblock Sharp convergence rates for mean field control in the region of
	strong regularity.
	\newblock {\em arXiv preprint arXiv:2312.11373}, 2023.
	
	\bibitem{CardaliaguetLionsPorretta2012}
	P.~Cardaliaguet, J.-M. Lasry, P.-L. Lions, and A.~Porretta.
	\newblock Long time average of mean field games.
	\newblock {\em Netw. Heterog. Media}, 7(2):279--301, 2012.
	
	\bibitem{CarSou2023}
	P.~Cardaliaguet and P.~E. Souganidis.
	\newblock Regularity of the value function and quantitative propagation of
	chaos for mean field control problems.
	\newblock {\em NoDEA Nonlinear Differential Equations Appl.}, 30(2):Paper No.
	25, 37, 2023.
	
	\bibitem{CDI}
	R.~Carmona and F.~Delarue.
	\newblock {\em Probabilistic theory of mean field games with applications.
		{I}}, volume~83 of {\em Probability Theory and Stochastic Modelling}.
	\newblock Springer, Cham, 2018.
	\newblock Mean field FBSDEs, control, and games.
	
	\bibitem{CDII}
	R.~Carmona and F.~Delarue.
	\newblock {\em Probabilistic theory of mean field games with applications.
		{II}}, volume~84 of {\em Probability Theory and Stochastic Modelling}.
	\newblock Springer, Cham, 2018.
	\newblock Mean field games with common noise and master equations.
	
	\bibitem{CDL}
	R.~Carmona, F.~Delarue, and A.~Lachapelle.
	\newblock Control of {M}c{K}ean-{V}lasov dynamics versus mean field games.
	\newblock {\em Math. Financ. Econ.}, 7(2):131--166, 2013.
	
	\bibitem{CavagnariLisiniOrrieriSavare2022}
	G.~Cavagnari, S.~Lisini, C.~Orrieri, and G.~Savar\'{e}.
	\newblock Lagrangian, {E}ulerian and {K}antorovich formulations of multi-agent
	optimal control problems: equivalence and gamma-convergence.
	\newblock {\em J. Differential Equations}, 322:268--364, 2022.
	
	\bibitem{claisse2023mean}
	J.~Claisse, G.~Conforti, Z.~Ren, and S.~Wang.
	\newblock Mean field optimization problem regularized by fisher information.
	\newblock {\em arXiv preprint arXiv:2302.05938}, 2023.
	
	\bibitem{conforti2024hamilton}
	G.~Conforti, R.~C. Kraaij, L.~Tamanini, and D.~Tonon.
	\newblock Hamilton--{J}acobi equations for {W}asserstein controlled gradient
	flows: existence of viscosity solutions.
	\newblock {\em arXiv preprint arXiv:2401.02240}, 2024.
	
	\bibitem{Cosso2023master}
	A.~Cosso, F.~Gozzi, I.~Kharroubi, H.~Pham, and M.~Rosestolato.
	\newblock Master {B}ellman equation in the {W}asserstein space: {U}niqueness of
	viscosity solutions.
	\newblock {\em Transactions of the American Mathematical Society}, 2023.
	
	\bibitem{CossoGozzi2023}
	A.~Cosso, F.~Gozzi, I.~Kharroubi, H.~Pham, and M.~Rosestolato.
	\newblock Optimal control of path-dependent {M}c{K}ean-{V}lasov {SDE}s in
	infinite-dimension.
	\newblock {\em Ann. Appl. Probab.}, 33(4):2863--2918, 2023.
	
	\bibitem{DVU14}
	F.~De~Vecchi and S.~Ugolini.
	\newblock An entropy approach to {B}ose-{E}instein condensation.
	\newblock {\em Commun. Stoch. Anal.}, 8(4):517--529, 2014.
	
	\bibitem{DNPV12}
	E.~Di~Nezza, G.~Palatucci, and E.~Valdinoci.
	\newblock Hitchhiker's guide to the fractional {S}obolev spaces.
	\newblock {\em Bull. Sci. Math.}, 136(5):521--573, 2012.
	
	\bibitem{Djete2022}
	M.~F. Djete.
	\newblock Extended mean field control problem: a propagation of chaos result.
	\newblock {\em Electron. J. Probab.}, 27:Paper No. 20, 53, 2022.
	
	\bibitem{DjeteDylan2022}
	M.~F. Djete, D.~Possama\"{\i}, and X.~Tan.
	\newblock Mc{K}ean-{V}lasov optimal control: the dynamic programming principle.
	\newblock {\em Ann. Probab.}, 50(2):791--833, 2022.
	
	\bibitem{DupuisEllis1997}
	P.~Dupuis and R.~S. Ellis.
	\newblock {\em A weak convergence approach to the theory of large deviations}.
	\newblock Wiley Series in Probability and Statistics: Probability and
	Statistics. John Wiley \& Sons, Inc., New York, 1997.
	\newblock A Wiley-Interscience Publication.
	
	\bibitem{FornasierLisiniOrrieriSavare2019}
	M.~Fornasier, S.~Lisini, C.~Orrieri, and G.~Savar\'{e}.
	\newblock Mean-field optimal control as gamma-limit of finite agent controls.
	\newblock {\em European J. Appl. Math.}, 30(6):1153--1186, 2019.
	
	\bibitem{KacChaos}
	M.~Hauray and S.~Mischler.
	\newblock On {K}ac's chaos and related problems.
	\newblock {\em J. Funct. Anal.}, 266(10):6055--6157, 2014.
	
	\bibitem{Watanabe1989}
	N.~Ikeda and S.~Watanabe.
	\newblock {\em Stochastic differential equations and diffusion processes},
	volume~24 of {\em North-Holland Mathematical Library}.
	\newblock North-Holland Publishing Co., Amsterdam; Kodansha, Ltd., Tokyo,
	second edition, 1989.
	
	\bibitem{Kall}
	O.~Kallenberg.
	\newblock {\em Probabilistic symmetries and invariance principles}.
	\newblock Probability and its Applications (New York). Springer, New York,
	2005.
	
	\bibitem{Lacker}
	D.~Lacker.
	\newblock Limit theory for controlled {M}c{K}ean-{V}lasov dynamics.
	\newblock {\em SIAM J. Control Optim.}, 55(3):1641--1672, 2017.
	
	\bibitem{Lacker2018}
	D.~Lacker.
	\newblock On a strong form of propagation of chaos for {M}c{K}ean-{V}lasov
	equations.
	\newblock {\em Electron. Commun. Probab.}, 23:Paper No. 45, 11, 2018.
	
	\bibitem{Lackerclosedloop2020}
	D.~Lacker.
	\newblock On the convergence of closed-loop {N}ash equilibria to the mean field
	game limit.
	\newblock {\em Ann. Appl. Probab.}, 30(4):1693--1761, 2020.
	
	\bibitem{Lackerclosedloop2023}
	D.~Lacker and L.~Le~Flem.
	\newblock Closed-loop convergence for mean field games with common noise.
	\newblock {\em Ann. Appl. Probab.}, 33(4):2681--2733, 2023.
	
	\bibitem{Santambrogio2018}
	H.~Lavenant and F.~Santambrogio.
	\newblock Optimal density evolution with congestion: {$L^\infty$} bounds via
	flow interchange techniques and applications to variational mean field games.
	\newblock {\em Comm. Partial Differential Equations}, 43(12):1761--1802, 2018.
	
	\bibitem{Leonard2014}
	C.~L\'{e}onard.
	\newblock A survey of the {S}chr\"{o}dinger problem and some of its connections
	with optimal transport.
	\newblock {\em Discrete Contin. Dyn. Syst.}, 34(4):1533--1574, 2014.
	
	\bibitem{OrrieriPorrettaSavare}
	C.~Orrieri, A.~Porretta, and G.~Savar\'{e}.
	\newblock A variational approach to the mean field planning problem.
	\newblock {\em J. Funct. Anal.}, 277(6):1868--1957, 2019.
	
	\bibitem{Rudin}
	W.~Rudin.
	\newblock {\em Principles of mathematical analysis}.
	\newblock International Series in Pure and Applied Mathematics. McGraw-Hill
	Book Co., New York-Auckland-D\"{u}sseldorf, third edition, 1976.
	
	\bibitem{Ruf}
	J.~Ruf.
	\newblock The martingale property in the context of stochastic differential
	equations.
	\newblock {\em Electron. Commun. Probab.}, 20:No. 34, 10, 2015.
	
	\bibitem{SantambrogioPDE2018}
	F.~Santambrogio.
	\newblock Regularity via duality in calculus of variations and degenerate
	elliptic {PDE}s.
	\newblock {\em J. Math. Anal. Appl.}, 457(2):1649--1674, 2018.
	
	\bibitem{Trevisan2016}
	D.~Trevisan.
	\newblock Well-posedness of multidimensional diffusion processes with weakly
	differentiable coefficients.
	\newblock {\em Electron. J. Probab.}, 21:Paper No. 22, 41, 2016.
	
	\bibitem{Villani}
	C.~Villani.
	\newblock {\em Optimal transport. Old and new}, volume 338 of {\em Grundlehren
		der mathematischen Wissenschaften -- Fundamental Principles of Mathematical
		Sciences}.
	\newblock Springer-Verlag, Berlin, 2009.
	
	\bibitem{Williams}
	D.~Williams.
	\newblock {\em Probability with martingales}.
	\newblock Cambridge Mathematical Textbooks. Cambridge University Press,
	Cambridge, 1991.
	
\end{thebibliography}

\end{document}